\documentclass[twoside, final]{article}

\title{Ergodic and Mixing Properties of the Boussinesq Equations with a Degenerate Random Forcing\footnote{\today}}
\author{Juraj F\"oldes$^\varheart$, Nathan Glatt-Holtz$^\vardiamond$, Geordie Richards$^{\clubsuit}$, Enrique Thomann$^{\spadesuit}$}
\date{}

\usepackage{amsfonts, amssymb, amsmath, amsthm, amsfonts, verbatim, cases, multicol}
\usepackage[usenames,dvipsnames]{color}
\usepackage[colorlinks=true, pdfstartview=FitV, linkcolor=ForestGreen,
            citecolor=ForestGreen, urlcolor=ForestGreen]{hyperref}

\definecolor{Red}{rgb}{0.7,0,0.1}
\definecolor{Green}{rgb}{0,0.6,0}

\usepackage{comment}
\usepackage[font=small,labelfont=bf, margin={.8cm,.8cm}]{caption}

\usepackage[notref,notcite,color]{showkeys}
\definecolor{labelkey}{rgb}{0,0,1}

\usepackage{graphicx}
\usepackage{amssymb}
\usepackage{epstopdf}
\numberwithin{equation}{section}
\numberwithin{figure}{section}

\newtheorem{Thm}{Theorem}[section]
\newtheorem{Lem}[Thm]{Lemma}
\newtheorem{Prop}[Thm]{Proposition}

\newtheorem{Def}[Thm]{Definition}
\newtheorem{Rmk}[Thm]{Remark}

\newtheorem*{Thm*}{Theorem}

\DeclareSymbolFont{extraup}{U}{zavm}{m}{n}
\DeclareMathSymbol{\varheart}{\mathalpha}{extraup}{86}
\DeclareMathSymbol{\vardiamond}{\mathalpha}{extraup}{87}

\newcommand{\pd}{\partial}
\newcommand{\indFn}[1]{1 \! \! 1_{#1}}
\newcommand{\E}{\mathbb{E}}
\newcommand{\Prb}{\mathbb{P}}
\newcommand{\RR}{\mathbb{R}}
\newcommand{\ZZ}{\mathbb{Z}}
\newcommand{\NN}{\mathbb{N}}

\newcommand{\om}{\omega}
\newcommand{\de}{\theta}
\newcommand{\bfU}{\mathbf{u}}
\newcommand{\pres}{p}
\newcommand{\gDir}{\mathbf{g}}
\newcommand{\TT}{\mathbb{T}^2}
\newcommand{\gphi}{g_{\phi}}

\newcommand{\Obs}{\Phi}
\newcommand{\ObsGC}{\varsigma}
\newcommand{\ObsSet}{\mathcal{O}}

\newcommand{\JJ}{\mathcal{J}}
\newcommand{\AAA}{\mathcal{A}}
\newcommand{\KK}{\mathcal{K}}
\newcommand{\MM}{\mathcal{M}}
\newcommand{\FF}{\mathcal{F}}

\newcommand{\ZZZ}{\mathcal{Z}}
\newcommand{\MTI}{\mathcal{R}^\beta}

\newcommand{\dStep}{\delta}
\newcommand{\aproxB}{\tilde{b}}

\newcommand{\MD}{\mathfrak{D}}
\newcommand{\Mspc}{\mathbb{D}}

\newcommand{\DM}{\mathfrak{D}}

\newcommand{\Ushft}{\bar{U}}

\definecolor{purple}{rgb}{0.5,0,0.5}

\newcommand{\bl}[1]{\textcolor{blue}{#1}}
\newcommand{\rd}[1]{\textcolor{red}{#1}}
\newcommand{\pp}[1]{\textcolor{purple}{#1}}

\begin{document}
\markboth{J. F\" oldes, N. Glatt-Holtz, G. Richards, E. Thomann}
{Ergodic and Mixing Properties of the Boussinesq Equations with a Degenerate Random Forcing}

\maketitle

\vskip-4mm

\centerline{\footnotesize{\it $^{\varheart}$Institute for Mathematics and its Applications, University of Minnesota}}
\vskip-1mm
\centerline{\footnotesize{\it Minneapolis, MN 55455}}

\vskip2mm

\centerline{\footnotesize{\it $^\vardiamond$Department of Mathematics, Virginia Polytechnic Institute and State University} }
\vskip-1mm
\centerline{\footnotesize{\it Blacksburg, VA 24061}}

\vskip2mm

\centerline{\footnotesize{\it $^{\clubsuit}$Department of Mathematics, University of Rochester}}
\vskip-1mm
\centerline{\footnotesize{\it Rochester, NY 14627}}

\vskip2mm

\centerline{\footnotesize{\it $^{\spadesuit}$Department of Mathematics, Oregon State University}}
\vskip-1mm
\centerline{\footnotesize{\it Corvallis, OR 97331}}

\begin{abstract}
We establish the existence, uniqueness and attraction properties of an ergodic invariant measure for the Boussinesq Equations
in the presence of a degenerate stochastic forcing acting only in the temperature equation and only at the largest spatial scales.
The central challenge is to establish time asymptotic smoothing properties of the Markovian dynamics corresponding
to this system.  Towards this aim we encounter a Lie bracket structure in the associated vector fields with a complicated 
dependence on solutions.  This leads us to develop a novel
H\"{o}rmander-type condition for infinite-dimensional systems.  Demonstrating the sufficiency of this condition requires
 new techniques for the spectral analysis of the Malliavin covariance matrix.
\end{abstract}

\setcounter{tocdepth}{1}
\tableofcontents

\newpage

\section{Introduction}
\label{sec:intro}

In this work we analyze the stochastically forced Boussinesq equations for the velocity field $\bfU=(u_1,u_2)$, (density-normalized) pressure $\pres$, and temperature $\theta$ of
a viscous incompressible fluid.  These equations take the form
\begin{align}
  &d \bfU + (\bfU \cdot \nabla \bfU) dt = (-\nabla \pres + \nu_1 \Delta \bfU + \gDir \de) dt, \quad \nabla \cdot \bfU = 0,
  \label{eq:B1}\\
  &d \de + (\bfU \cdot \nabla \de ) dt = \nu_2 \Delta \de dt + \sigma_\de d W,
  \label{eq:B2}
\end{align}
where the parameters $\nu_1, \nu_2 > 0$ are respectively the kinematic viscosity and thermal diffusivity of the fluid
and $\gDir = (0,g)^T$ with $g \neq 0$ is the product of the gravitational
constant and the thermal expansion coefficient.
The spatial variable $x = (x_1, x_2)$ belongs to a two-dimensional torus $\mathbb{T}^2$. That is, we impose periodic
boundary conditions in space.  We consider a degenerate stochastic forcing $\sigma_\de d W$, which acts only on a few Fourier modes and
exclusively through the temperature equation.

We prove that there exists a unique statistically invariant state of the system \eqref{eq:B1}--\eqref{eq:B2}.  More precisely, we establish:

\begin{Thm}
\label{thm:Main:res}
With white noise acting only on the two largest standard modes of the temperature equation \eqref{eq:B2},
\begin{align*}
	\sigma_\de dW = \alpha_1 \cos x_1 dW^{1} + \alpha_2 \sin x_1 dW^{2} +
   \alpha_3 \cos x_2 dW^{3} + \alpha_4 \sin x_2 dW^{4},
\end{align*}
the Markov semigroup corresponding to \eqref{eq:B1}--\eqref{eq:B2} possesses a unique ergodic invariant
measure.  Moreover this measure is mixing, and it obeys a law of large numbers and a central limit theorem.
\end{Thm}

The interaction between the nonlinear and stochastic terms in \eqref{eq:B1}--\eqref{eq:B2} is delicate, and leads us to develop a novel infinite-dimensional form of the
H\"{o}rmander bracket condition. Our analysis generalizes techniques developed in the
recent works \cite{MattinglyPardoux1,HairerMattingly06, HairerMattingly2008, HairerMattingly2011}, and we believe it has broader interest for systems of SPDEs.

\subsection{Historical Background and Motivations}

Going back to the early 1900's Rayleigh \cite{Rayleigh1916} proposed the study of buoyancy driven fluid convection problems
using the equations of Boussinesq \cite{Boussinesq1897} in order to explain the experimental work of B\'enard \cite{Benard1900}.
Today this system of equations plays a fundamental role in a wide variety of physical settings including climate and weather,
the study of plate tectonics, and the internal dynamical structure of stars,
see e.g. \cite{Pedlosky, Busse1989, Getling1998, DoeringGibbon1995, BodenschatzPeschAhlers2000} and references therein
for further background.

Physically speaking, the system \eqref{eq:B1}--\eqref{eq:B2} (with $\sigma_\theta=0$) arises as follows.
Consider a fluid with velocity $\bfU$ confined between two horizontal plates, where one fixes the temperature $\theta$ of the fluid on the top
$\theta_t$, and bottom $\theta_b$, with $\theta_t  \leq \theta_b$ (i.e. heating from below). It is typical to assume a linear relationship between density and
temperature, and to impose the Boussinesq approximation, which posits that the only significant role played by density variations in the
fluid arise through the gravitational terms, so that the fluid velocity $\bfU$ and temperature $\theta$ evolve according to \eqref{eq:B1}--\eqref{eq:B2}.
Due to the presence of viscosity, the fluid is not moving at the plates and one assumes no-slip boundary conditions $\bfU = 0$.\footnote{One can consider
consider the equations posed on an infinite channel, or assume periodic boundary conditions in the horizontal direction for both the velocity field and temperature.}

The form of the Boussinesq equations we consider in this work, that is \eqref{eq:B1}--\eqref{eq:B2} supplemented with periodic boundary conditions, is sometimes referred to in the physics community
as the `homogeneous Rayleigh-B\'enard' or `HRB' system.  It is derived as follows:  One transforms the governing equations we have just described
into an equivalent homogenous system by subtracting off a linear temperature profile.  This
introduces an additional excitation term in the temperature equation, and makes the temperature vanish at the plates.  As a numerical simplification, one then replaces these boundary conditions with periodic ones
(see \cite{LohseToschi2003, CalzavariniLohseToschiTripiccione2005}).  This periodic setting is controversial in the physics community since
it can produce unbounded (`grow-up') solutions, as has been observed both numerically and through explicit solutions, \cite{Calzavarini2006}.
We will consider \eqref{eq:B1}--\eqref{eq:B2} in situations with \emph{no temperature differential} ($\theta_t = \theta_b$, i.e. zero Rayleigh number),
so that the additional excitation term is not present and such unbounded solutions do not exist.  Extensions to more physically realistic boundary conditions for \eqref{eq:B1}--\eqref{eq:B2} will be addressed in forthcoming works.

In the mathematical community the deterministic Boussinesq equations with various boundary conditions on bounded and unbounded domains have attracted considerable attention.  In one line of work the
2D system has been interpreted as an analogue of 3D axisymmetric flow
where `vortex stretching' terms appear in the `vorticity formulation', see e.g. \cite{CannonDiBenedetto1980, ChaeImanuvilov1999,
CordobaFeffermandelaLlave2004, HouLi2005, Chae2006, DanchinPaicu2008,
DanchinPaicu2009, HmidiKeraani2009,AdhikariCaoWu2010, LariosLunasinTiti2010,AdhikariCaoWu2010, DanchinPaicu2011,
ChaeWu2012, CaoWu2013}.  Other authors have sought to provide a rigorous mathematical framework
for various physical and numerical observations in fluid convection problems,  see e.g. \cite{ConstantinDoering1996, ConstantinDoerin1999, ConstantinDoering2001,
MaWang2004, Wang2004a, Wang2004b, Wang2005, Wang2007, MaWang2007, Wang2008a, Wang2008b, SengulWang2013}.

Let us briefly motivate the stochastic forcing appearing in \eqref{eq:B1}--\eqref{eq:B2}.  Due to sensitivity with respect to initial data and parameters, individual solutions of the basic equations of fluid mechanics are unpredictable and seemingly chaotic.  However,
some of their statistical properties of solutions are robust.  As early as the 19th century J.V. Boussinesq conjectured that turbulent flow
cannot be solely described by deterministic methods, and indicated that a
stochastic framework should be used, see \cite{Stanisic1985}.  More recently the study of the Navier-Stokes equations
with degenerate white noise forcing has been proposed a proxy for the large-scale `generic'
stirring which is assumed in the basic theories of turbulence; this setting is ubiquitous in the turbulence literature, see e.g. \cite{Novikov1965, VishikKomechFusikov1979, Eyink96} and containing references.  In this view, invariant measures of the stochastic
equations of fluid dynamics would presumably contain the statistics posited by these theories.\footnote{In our context the 2D Batchelor-Krichanan
theory \cite{Kraichnan67,Batchelor69} is probably the most relevant statistical theory.  Note that its applicability would presumably require the imposition of a large scale damping term in
the momentum equation as explained in \cite{GlattHoltzSverakVicol2013} and see also e.g. \cite{KraichnanMontgomery80,Tabeling02, FoiasJollyManleyRosa02, Kupiainen10}.
Such a damping operator would not affect any of the conclusions drawn in the main results below.  We also mention recent work on the statistics of turbulence
in axisymmetric 3D flows \cite{ThalabardDubrulleBouchet2013,NasoThalabardColletteChavanisDubrulle2010}.
The governing equations for these systems bare a strong structural resemblance to \eqref{eq:B1}--\eqref{eq:B2} as has been pointed
out in e.g. \cite{HouLi2005}.}
The closely related question of unique ergodicity and mixing provides rigorous justification for the explicit and implicit measurement assumptions invoked by
physicists and engineers when measuring statistical properties of turbulent systems.

The existence of invariant measures for forced-dissipative systems is often easy to prove with classical tools, namely by making use of
the Krylov-Bogoliubov averaging procedure with energy (compactness) estimates,
but the uniqueness of these measures is a deep and subtle issue.
To establish this uniqueness one can follow the path laid out by the Doob-Khasminskii Theorem \cite{Doob1948,Khasminskii1960,ZabczykDaPrato1996},
and more recently expanded upon in \cite{HairerMattingly06,HairerMattingly2008}.
This strategy requires the proof of certain smoothing properties of the associated Markov semigroup, and to show that a common state can be
reached by the dynamics regardless of initial conditions (irreducibility).
Without stochastic forcing, solutions of our system converge to the trivial equilibrium, so that
the proof of irreducibility is straightforward in our context.  Thus the main challenge of this work is
to establish sufficient smoothing properties for the Markov semigroup associated to \eqref{eq:B1}--\eqref{eq:B2}.

\subsection{Smoothing and Hypoellipticity in Infinite Dimensional Systems}

In order to discuss the difficulties in establishing smoothing properties for the Markov semigroup
it is useful to recall the canonical relationship between stochastic evolution equations
and their corresponding Fokker-Planck (Kolmogorov) equations.
Consider an abstract equation on a Hilbert space $H$,
\begin{align}
	dU = F(U) dt + \sigma dW(t); \quad U(0) = U_0\,,
	\label{eq:abs:SPDE:intro}
\end{align}
where $\sigma dW(t) := \sum_{k =1}^N \sigma_k dW^k(t)$,
and $\{W_k\}_{k=1}^N$ is a (finite or infinite) collection of independent 1D Brownian motions.  We denote
solutions $U$ with the initial condition $U_0$ at time $t \geq 0$ by $U(t,U_0)$, and define the Markov semigroup
associated to \eqref{eq:abs:SPDE:intro} according to $P_t \Obs(U_0) = \E \Obs(U(t,U_0))$, where $\Obs: H \to \RR$ is any `observable'.
Then $\Psi(t) := P_t \Obs$ solves the Fokker-Planck equation corresponding to \eqref{eq:abs:SPDE:intro} given by
 \begin{align}
     \partial_t \Psi = \frac{1}{2} Tr[ (\sigma \sigma^*) D^2 \Psi] + \langle F(U), D\Psi \rangle, \quad \Psi(0) := \Obs
     \label{eq:fokk:er:plank}
 \end{align}
where we view $\sigma$ as an element in $\mathcal{L}(\RR^N, H)$.  The interested reader should consult \cite{Cerrai2001,DaPratoZabczyk2002} for more on the general theory of
second order PDEs posed on a Hilbert space.

There is a wide literature devoted to proving uniqueness and associated mixing properties of invariant measures for nonlinear stochastic PDEs
when $\sigma \sigma^*$ is non-degenerate or mildly degenerate.  See e.g. \cite{VishikKomechFusikov1979, Cruzeiro1, FlandoliMaslowski1, ZabczykDaPrato1996,
Ferrario1997,Mattingly1,Mattingly2, E2001, EMattinglySinai2001, BricmontKupiainenLefevere2001,EckmannHairer2001, MasmoudiYoung2002,
KuksinShirikyan1,KuksinShirikyan2, Mattingly2,Mattingly03, LeeWu2004, GoldysMaslowski2005,DaPratoRocknerRozovskiiWang2006,  AlbeverioFlandoliSinai2008,
Debussche2011a, KuksinShirikian12, ConstantinGlattHoltzVicol2013} and references therein.  Roughly speaking,
the fewer the number of driving stochastic terms in \eqref{eq:abs:SPDE:intro}, the more degenerate the diffusion in \eqref{eq:fokk:er:plank}, and the more difficult it becomes
to establish smoothing properties for the Markov semigroup $\{P_t\}_{t \geq 0}$.\footnote{On the other hand with more driving terms in \eqref{eq:abs:SPDE:intro},
 the well-posedness  theory may become more difficult.  For example this was the primary mathematical
challenge in the pioneering work \cite{FlandoliMaslowski1}.}
Moreover, while even the non-degenerate setting poses many interesting mathematical challenges, such stochastic
forcing regimes are highly unsatisfactory from the point of view of turbulence where one typically assumes a clear
separation between the forced and dissipative scales of motion.

Going back to the seminal work of H\"{o}rmander \cite{Hormander1967} (and cf. \cite{Kolmogoroff1934}), a theory of parabolic
regularity for finite dimensional PDEs of the general form  \eqref{eq:fokk:er:plank} with degenerate diffusion terms was developed.
This theory of `hypoellipticity' can be interpreted in terms of finite-dimensional stochastic ODEs for which these degenerate parabolic PDEs are
the corresponding Kolmogorov equations.  This connection suggested the potential for a more probabilistic approach, initiating
the development of the so-called Malliavin calculus; see \cite{Malliavin1978} and subsequent authors
\cite{Shigekawa1980, Bismut1981, Bismut1981a, Stroock1981, IkedaWatanabe1989, KusuokaStroock1984,KusuokaStroock1985,Norris1986, KusuokaStroock1987}.
In any case, the work of H\"{o}rmander and Malliavin has led to an extensive theory of unique ergodicity and mixing properties for finite-dimensional
stochastic ODEs.

By comparison, for stochastic PDEs (which correspond to the situation when \eqref{eq:fokk:er:plank} is posed on an infinite dimensional space) this theory of hypoellipticity remains in its
infancy.  Recently, however, in a series of groundbreaking works \cite{MattinglyPardoux1,HairerMattingly06, HairerMattingly2008, HairerMattingly2011},
a theory of unique ergodicity for degenerately forced infinite-dimensional stochastic systems has emerged.
These works produced two fundamental contributions: Firstly the authors demonstrated that to establish the
uniqueness of the invariant measure it suffices to prove time asymptotic smoothing (asymptotic strong Feller property) instead of `instantaneous' smoothing (strong Feller property).
This is an abstract result from probability and it applies to very general settings including ours.
Secondly the authors generalize the methods of Malliavin (and subsequent authors)
in order to prove the asymptotic strong Feller property for certain infinite-dimensional stochastic systems.  These works resulted in an infinite-dimensional analogue of the H\"{o}rmander bracket condition.

For the second point the application of the methods in \cite{HairerMattingly2011} is more delicate and must be considered on a case-by-case basis; it requires a careful analysis of the
interaction of the nonlinear and stochastic terms of the system.  In our situation the bracket condition in \cite{HairerMattingly2011} is not satisfied and needs to be replaced by a weaker notion.  This required us to rework and generalize many basic elements of their approach.

\subsection{The H\"{o}rmander bracket condition in Infinite Dimensions}\label{sec:Hor:int}

To explain our contributions at a more technical level, let us recall
what is meant by a `H\"ormander bracket condition' in the context of systems
of the general form \eqref{eq:abs:SPDE:intro}.
Define
\begin{align*}
  \mathcal{V}_0 := \mbox{span}\{ \sigma_k : k = 1, \ldots d\},
\end{align*}
and for $m > 0$ take
\begin{align}
  \mathcal{V}_m := \mbox{span}\{ [E ,  F], [E , \sigma_k], E : E \in \mathcal{V}_{m-1}, k =1, \ldots, d\} \,,
  \label{eq:finite:braks:span:set}
\end{align}
where $F=F(U)$ is the drift term in \eqref{eq:abs:SPDE:intro}, and for any Fr\'echet differentiable $E_1, E_2 : H \to H$,
\begin{align}
[E_1, E_2](U) := \nabla E_2(U) E_1(U) - \nabla E_1(U)  E_2(U).
\label{eq:Lie:brak:abs}
\end{align}
This operation $[E_1, E_2]$ is referred to as the  \emph{Lie bracket} of the two `vector fields' $E_1, E_2$.
In finite dimensions, that is when $H = \RR^N$, the classical H\"ormander condition is satisfied if $ \cup_{k \geq 0} \mathcal{V}_k = H$ for all $U \in H$.

This condition is not suitable for infinite-dimensional settings since the effect of randomness from the directly forced modes seems to weaken as it propagates out in infinite dimensional space. In \cite{HairerMattingly06,HairerMattingly2011} it was shown that, due to contractive properties of the flow on high modes, it is enough to require that the Lie brackets span all the unstable directions of the system. More precisely, the following non-degeneracy assumption is required:
For every $N > 0$, there is an $M > 0$ and a finite set $\mathfrak{B} \subset \mathcal{V}_M$ such that the quadratic form
\begin{align}
	\langle \mathcal{Q}(U) \phi, \phi \rangle := \sum_{b(U) \in  \mathfrak{B}} \langle \phi, b(U) \rangle^2
	\label{eq:quad:form:def}
\end{align}
satisfies, for each $\alpha > 0$,
\begin{align}
   \langle \mathcal{Q}(U) \phi, \phi \rangle \geq \alpha C \| \phi\|^2, \quad
   \textrm{ for every } \phi \in \mathcal{S}_{\alpha, N} := \{ \phi \in H: \| P_N \phi\|^2 > \alpha \|\phi\|^2\},
   \label{eq:Hormander:inf:dim:HM}
\end{align}
where $U$ solves \eqref{eq:B1}-\eqref{eq:B2}, $C > 0$ is a constant independent of $\alpha$, and $\{P_N\}$ is a sequence of projection operators onto successively larger spaces.\footnote{In fact the condition given
in \cite{HairerMattingly2011} is slightly more general than this. They also allow for the case when $C$ might dependent on $U$ (subject to suitable moment bounds) but they require that $C$ be
almost everywhere positive which is \emph{not} sufficient for our purposes.}

It was suggested in \cite{HairerMattingly2011} that for many equations with polynomial nonlinearities
brackets of the type $[\cdots [[F(U), \sigma_{k_1}], \cdots], \sigma_{k_m}]$
(where $\sigma_{k_1}, \cdots, \sigma_{k_m}$ are previously generated constant vector fields) suffice to build quadratic forms satisfying \eqref{eq:Hormander:inf:dim:HM}.
Indeed, this approach has now been successfully employed for several important examples, including the 2D and 3D Navier-Stokes equations
and the Ginzburg-Landau equations; see \cite{EMattingly2001,Romito2004,HairerMattingly06, HairerMattingly2011}.
In these works algebraic conditions on the set of stochastically forced modes have been derived which guarantee that any finite-dimensional space $H_N$ can be generated from these types of brackets; one obtains a collection of $U$-independent
elements $\mathfrak{B}$ which form an orthonormal basis for $H_N$.
This strategy has proven effective for certain scalar equations, but its
limitations are evident in slightly more complicated situations.  We believe
that \eqref{eq:B1}--\eqref{eq:B2} provides an illuminating case study of these difficulties
which has lead us to generalize \eqref{eq:Hormander:inf:dim:HM}.

Observe that our model is distinguished by two key structural properties.
Firstly, the buoyancy term $\mathbf{g}\theta$ is the only means of spreading the effect of the stochastic forcing
from the temperature equation, \eqref{eq:B1}, to the momentum equations, \eqref{eq:B2}.
In particular note that this buoyancy term is linear, and therefore vanishes after two Lie bracket operations with constant vector fields.  Secondly,
the advective structure in \eqref{eq:B1}--\eqref{eq:B2} leads to a delicate `asymmetry' in the nonlinear terms.
For example this means that a
more refined analysis is needed to address the spread of noise in the temperature equation \eqref{eq:B2} alone.\footnote{In simple
language this `advective structure' means that $\bfU \cdot \nabla \theta$ is the only second order term in the temperature equation \eqref{eq:B2}.}
Concretely we find that, by combining these observations, one obtains $[[F(U), \sigma_{k_1}],\sigma_{k_2}] = 0$
for any constant vector fields $\sigma_{k_1},\sigma_{k_2}$ concentrated in the
temperature component of the phase space.
In view of these discussions it is worth emphasizing
that our chosen forcing - that is, noise acting only through the temperature equations -
is the most interesting situation to consider from the point of view of ergodicity.

To get around these difficulties we make careful use of the interaction between the nonlinear, buoyancy and noise terms in \eqref{eq:B1}--\eqref{eq:B2} as follows:  We see that noise activated through the temperature equation is spread to the momentum equations.
It is then advected in the temperature equations and combines again with the noise terms to produce new directions in the temperature
component of the phase space.  These loose observations find concrete expressions in a series of Lie brackets of the form
$[[[[F, \sigma_{k_1}], F], \sigma_{k_2}]$ for constant vector fields $\sigma_{k_1}$ and $\sigma_{k_2}$ which belong to the $\theta$-component.  Remarkably,
we found that this chain of admissible brackets leads to new constant vector fields exclusively in the $\theta$-component of the phase space.
This surprising observation requires a series of detailed computations but is perhaps anticipated by the advective structure of the nonlinear terms.

It is in addressing the spread of noise in the momentum components of the phase space that the condition \eqref{eq:Hormander:inf:dim:HM} breaks down.  Forced directions in the temperature component are pushed to the momentum components through the buoyancy term.  However, due to the presence of the nonlinearity, they are `mixed' with terms which have an unavoidable and complicated directional (non-frequency-localized) dependence on $U$.
Crucially, due to the advective structure in \eqref{eq:B1}--\eqref{eq:B2}, these `error terms' are concentrated only in the temperature component of the phase
space.  We are therefore able to push these error terms to small scales by using the `pure' $\theta$ directions already generated (following the procedure described in the previous paragraph).  More precisely, in the language of \eqref{eq:finite:braks:span:set}, \eqref{eq:quad:form:def}, we are able to show that for every $N, \tilde{N}$ we can find $M> 0$ and sets $\mathfrak{B} \subset \mathcal{V}_M$ consisting of elements of the form $\Psi_k + J_k^{\tilde{N}}(U)$, where the sequence $\{\Psi_k\}$ (which are essentially
vectors consisting of trigonometric functions)
forms an orthonormal basis for the phase space and $J_k^{\tilde{N}}$ are
functions taking values in $H$ with a complicated dependence on $U$ but which are supported on `high' frequencies, i.e. wavenumbers larger than $\tilde{N}$.

These structural observations for \eqref{eq:B1}--\eqref{eq:B2} lead us to formulate the following generalization of \eqref{eq:Hormander:inf:dim:HM}.
\begin{Def}\label{def:our:cond}
Let $H$ and $\tilde{H}$ be Hilbert spaces with $\tilde{H}$ compactly embedded in $H$.
We say that \eqref{eq:abs:SPDE:intro} satisfies the \emph{generalized H\"ormander condition}
if for every $N> 0$ and every $\epsilon > 0$ there exist
$M = M(\epsilon, N) > 0$, $p \geq 1$, and a finite set $\mathfrak{B} \subset \mathcal{V}_M$ (where $\mathcal{V}_M$
is defined according to \eqref{eq:finite:braks:span:set})
such that
\begin{align}
 \langle \mathcal{Q}(U) \phi, \phi \rangle \geq C \left(\alpha - \epsilon(1 + \| U \|^p_{\tilde{H}}) \right)\| \phi\|^2
  \quad \textrm{ for every } \phi \in \mathcal{S}_{\alpha, N}\,,
 \label{eq:our:new:Hormander}
\end{align}
where $C$ is independent of $\alpha$ and $\epsilon$,
and $\mathcal{Q} = \mathcal{Q}_{N, \epsilon}$, $\mathcal{S}_{\alpha,N}$ are defined in \eqref{eq:quad:form:def},   \eqref{eq:Hormander:inf:dim:HM} respectively.
\end{Def}
Below we demonstrate that the condition \eqref{eq:our:new:Hormander}
is sufficient to establish suitable time asymptotic smoothing properties for the Markov semigroup associated to \eqref{eq:abs:SPDE:intro}\footnote{For clarity and simplicity we present all the arguments in the context of the Boussinesq system but the implications for the Markov semigroup for systems satisfying Definition~\ref{def:our:cond} could be shown to hold in a more general setting with essentially the same analysis.}
This allows us to apply the abstract results from \cite{HairerMattingly06, HairerMattingly2008,KomorowskiWalczuk2012} to complete the proof of Theorem~\ref{thm:Main:res}.  The main technical challenge arising from the modified condition \eqref{eq:our:new:Hormander} is that it requires us to significantly rework the spectral analysis of the Malliavin matrix appearing in \cite{MattinglyPardoux1,HairerMattingly06, HairerMattingly2011}.  The technically oriented reader can skip immediately to Section \ref{sec:Mal:spec:bnds} for further details.

\subsection{Organization of the Manuscript}

The manuscript is organized as follows:  In Section~\ref{sec:math:setting} we restate our problem in an abstract functional setting
and introduce some general definitions and notations. Then we reduce the question of uniqueness of the invariant measure to establishing
a time asymptotic gradient estimate on the Markov semigroup.
Next, in Section~\ref{sec:grad:control}
we explain how, using the machinery of Malliavin calculus, this gradient bound reduces to a control problem for a linearization of
\eqref{eq:B1}--\eqref{eq:B2}.   In turn
we show that this control problem may be solved by establishing appropriate spectral bounds for the Malliavin covariance matrix $\MM$.  Section~\ref{sec:Mal:spec:bnds}
is devoted to proving that our new form of the H\"ormander condition, \eqref{eq:our:new:Hormander}, implies these spectral bounds, modulo some technical estimates postponed for Section~\ref{sec:Braket:Est:Mal:Mat}.
Section~\ref{sec:Hormander:brak}
provides detailed Lie bracket computations leading to the modified condition \eqref{eq:our:new:Hormander}.
Finally in Section~\ref{sec:mixing}
we establish mixing properties, a law of large numbers and a central limit theorem for the invariant measure by making careful use of
recent abstract results from \cite{Shirikyan2006,HairerMattingly2008,KomorowskiWalczuk2012}.
Appendices \ref{sec:moment:est} and \ref{subsect:malliavin}  collect respectively statistical moment bounds for
\eqref{eq:B1}--\eqref{eq:B2} (and associated linearizations)
and a brief review of some elements of the Malliavin calculus used in our analysis.

\section{Mathematical Setting and Background}
\label{sec:math:setting}

In this section we formulate
\eqref{eq:B1}--\eqref{eq:B2} as an abstract evolution equation on a Hilbert space and  define its associated Markovian framework.
Then we formulate our main result in Theorem~\ref{thm:mainresult}
and give an outline of the proof, which
sets the agenda for the work below.

In the rest of the paper, we consider \eqref{eq:B1} in the equivalent, vorticity formulation.
Namely, if we denote $\om := \nabla^\perp \cdot \bfU = \pd_{x}u_2 - \pd_{y} u_1$, then by a standard
calculation we obtain
\begin{align}
  &d \om + (\bfU \cdot \nabla \om - \nu_1 \Delta \om) dt = g \partial_x \de dt, \quad
  \label{eq:B1:Vort}\\
  &d \de + (\bfU \cdot \nabla \de - \nu_2 \Delta \de)dt = \sigma_\de d W\,.
  \label{eq:B3:Vort}
\end{align}
The system \eqref{eq:B1:Vort}--\eqref{eq:B3:Vort} is posed on $\TT \times (0, \infty)$, where
 $\TT$ is the square torus $\TT = [-\pi, \pi]^2 = \RR^2/(2\pi \ZZ^2)$.\footnote{Of course, by rescaling we obtain our
 results for any square torus $\RR^2/(L \ZZ^2)$ and at the cost of more complicated notation and expressions below,
 one can also prove our results for non-square tori.}
 To close the system \eqref{eq:B1:Vort}--\eqref{eq:B3:Vort}, we, as usual, calculate $\bfU$ from $\om$
by Biot-Savart law, that is
$\bfU = K \ast \om$, where $K$ is the Biot-Savart kernel, so that $\nabla^\perp \cdot \bfU = \omega$
and $\nabla \cdot \bfU =0$, see e.g. \cite{Temam2001, MajdaBertozzi2002} for further details.
 As mentioned in the introduction, the physically interesting case of non-periodic
 domains will be the subject of a future work.

We next introduce a functional setting for the system \eqref{eq:B1:Vort}--\eqref{eq:B3:Vort}.
The phase space is composed of mean zero, square integrable
functions
\begin{align}
  H := \left\{ U := (\om, \de)^T \in (L^{2}(\TT))^{2}: \int_{\TT} \om dx = \int_{\TT} \de dx=0 \right\}
  \label{eq:H:def}
\end{align}
equipped with the norm
\begin{align}
   \|U \|^{2} :=  \frac{\lambda_1 \nu_1 \nu_2}{ g^{2}}\|\om\|_{L^{2}}^{2} +  \| \de \|^{2}_{L^{2}} \,,
   \label{eq:norm:def}
\end{align}
where $\lambda_1=1$ (we omit $\lambda_1$ below) is the principal eigenvalue of $-\Delta$ on $H$.
Observe that  this norm is equivalent to the standard norm on the space $(L^{2}(\TT))^{2}$.\footnote{Here and below we drop
the dependence of various function spaces on $\mathbb{T}^2$.} Our choice of the norm
is natural as both terms on the right hand side of \eqref{eq:norm:def} have been suitably non-dimensionalized.
The associated inner product on $H$ is denoted by $\langle \cdot, \cdot \rangle$.
Note furthermore that the zero mean property embodied in the definition of $H$, \eqref{eq:H:def} is
to be maintained by the flow \eqref{eq:B1:Vort}--\eqref{eq:B3:Vort}. The higher order Sobolev spaces are denoted
\begin{align*}
  H^{s} := \left\{ U := (\om, \de)^T \in (W^{s, 2}(\TT))^{2}: \int_{\TT} \om dx = \int_{\TT} \de dx=0 \right\}  \qquad  \textrm{ for any } s \geq 0,
\end{align*}
where $W^{s, 2}(\TT)$ is classical Sobolev-Slobodeckii space, and $H^s$ is equipped with the  norm
\begin{align*}
   \|U \|^{2}_{H^{s}} :=  \frac{\nu_1 \nu_2}{g^{2}}\|\om\|_{W^{s,2}}^{2} +  \| \de \|^{2}_{W^{s,2}}.
\end{align*}
For $s > 0$, we also denote $H^{-s} := (H^s)^*$, the dual space to $H^s$.

Since we need to estimates for linearizations of \eqref{eq:B1:Vort}--\eqref{eq:B3:Vort}
around initial conditions and noise paths, we encounter elements in $\mathcal{L}(H)$
and $\mathcal{L}(H, \mathcal{L}(H))$ (where $\mathcal{L}(X) = \mathcal{L}(X,X)$, $\mathcal{L}(X,Y)$ is the space
of linear operators from $X$ to $Y$).
See, for example \eqref{eq:def:Mal:J}, below.   We will sometimes abuse notation and donate
\begin{align}
  \|\mathcal{J}\| := \|\mathcal{J}\|_{\mathcal{L}(H)}, \qquad \|\mathcal{J}^{(2)}\| := \|\mathcal{J}^{(2)}\|_{\mathcal{L}(H, \mathcal{L}(H))},
  \label{eq:operator:abuse:call:911}
\end{align}
for any $\mathcal{J} \in \mathcal{L}(H)$, $\mathcal{J}^{(2)} \in \mathcal{L}(H, \mathcal{L}(H))$.

We  also frequently work with finite dimensional subspaces of $H$ along with
the projection operators onto these spaces.  Fix the trigonometric basis:
\begin{align}
  \sigma_k^0(x) :=
   \left(0, \cos(k\cdot x)\right)^T, \quad
  \sigma_k^1(x) :=  \left(0, \sin(k \cdot x)\right)^T\,,
  \label{eq:def:basis:sig}
\end{align}
and
\begin{align}
  \psi_k^0(x) :=
   \left( \cos(k\cdot x), 0\right)^T, \quad
  \psi_k^1(x) :=  \left(\sin(k \cdot x), 0\right)^T.
  \label{eq:def:basis:psi}
\end{align}
We denote
\begin{align*}
 \ZZ^{2}_{+} := \left\{ j = (j_1, j_2) \in \ZZ^2_0: j_1 > 0 \textrm{ or } j_1 = 0, j_2 > 0 \right\}
\end{align*}
and for any $N \geq 1$ define
\begin{align}
H_N:= \mbox{span}\{\sigma_k^l, \psi_k^l : |k| \leq N, l \in \{0, 1\}\},
\label{eq:H:N:finite:dim}
\end{align}
along with the associated projection operators
\begin{align}
&P_N: H \to H_N \textrm{ the orthogonal projection onto $H_N$ }, \quad Q_N :=  I - P_N\,.
\label{eq:H:N:proj:ops}
\end{align}
Note that $Q_N$ maps $H$ onto $\mbox{span}\{\sigma_k^l, \psi_k^l : |k| > N, l \in \{0,1\}\}$.

In order to rewrite \eqref{eq:B1:Vort}--\eqref{eq:B3:Vort} in a functional form we
introduce the following abstract operators associated to the various terms in the equation.
For $U := (\omega, \de)$ and $\tilde{U} = (\tilde{\omega}, \tilde{\de})$,
let $A : D(A) \subset H \to H$ be the linear symmetric positive definite operator defined by
\begin{align*}
  &A U := (-\nu_1 \Delta \om, - \nu_2 \Delta \de)^T,
\end{align*}
for any  $U \in D(A)$. Note that the scale of spaces $H^s$, $s \in \RR$ coincides with the interpolation spaces between $D(A)$ and $H$ and $A$ is a bounded operator from $H^{s+2}$ to $H^{s}$.

For the inertial (non-linear) terms define $B : H^1 \times H^1 \to H$ by
\begin{align}
  &B(U,\tilde{U}) := ( (K \ast \om) \cdot \nabla \tilde{\om}, (K \ast \om) \cdot \nabla \tilde{\de})^T,
    \label{eq:convection:term}
\end{align}
for $U, \tilde{U} \in H^1$. It is well known that
$\|K \ast \om \|_{H^s} \leq C \| \om \|_{H^{s-1}}$, and since $H^2 \hookrightarrow L^\infty $, we indeed obtain that
$B(U, \tilde{U}) \in H$. Also set $B(U) := B(U,U)$. Finally, for the `buoyancy term' define $G : H^1 \to H$ by
\begin{align}
  &GU = ( g \partial_x \de, 0)^T,
  \label{eq:boy:lin:term}
\end{align}
for $U \in H^1$.

Next we focus on the stochastic forcing terms appearing in \eqref{eq:B3:Vort}. We introduce a finite set $\mathcal{Z}\subset \ZZ^{2}_{+}$
which represents the forced directions in Fourier space.
The driving noise process $W := (W^{k, l})_{k \in \ZZZ, l = 0,1}$ is a $d := 2\cdot |\ZZZ|$-dimensional Brownian motion
defined relative to a filtered probability space $(\Omega, \FF,  \{\FF_t\}_{t \geq 0}, \Prb)$ and we refer to the resulting
tuple $\mathcal{S} = (\Omega, \FF,  \{\FF_t\}_{t \geq 0}, \Prb, W)$ as a \emph{stochastic basis}.\footnote{We may take the stochastic basis to be
the standard `Wiener space'.  Here we let $\Omega = \{\omega \in C([0, \infty); \RR^{2|\ZZZ|}): \omega(0) = 0\}$ with its Borelian $\sigma$-algebra and take  $\Prb$ to be the Wiener measure.
On this space the stochastic process defined by the evaluation map $W(t,\omega) = \omega(t)$ produces the statistics of Brownian motion.  The filtration $\FF_t$ is
then defined by the (completion) of the $\sigma$-algebra generated by $W(s)$ for $s \leq t$.  See e.g. \cite{KaratzasShreve} for further details.}  Let $\{e_k^l\}_{k \in \ZZZ, l = 0,1}$ be the standard
basis of $\RR^{2|\ZZZ|}$ and let $\{\alpha_{k}^l\}_{k \in \ZZZ, l = 0,1}$ be a sequence of non-zero numbers. We define a linear map $\sigma_\de : \RR^{2|\ZZZ|} \to H$
such that
\begin{align}
\sigma_\de e_k^l := \alpha_{k}^l \sigma_k^l \qquad \textrm{ for any } k \in \ZZZ, l \in \{0,1\}\,.
\label{eq:noise:cog:def}
\end{align}
where, $\sigma_k^l$ are the basis elements defined in \eqref{eq:def:basis:sig}.
Denote the Hilbert-Schmidt norm of $\sigma_{\theta}$
by
\begin{equation*}
  \|\sigma_\de \|^{2} := \|\sigma_\de^\ast \sigma_\de\| =  \sum_{\substack{k \in \mathcal{Z}\\ l \in \{0, 1\}}} (\alpha_k^l)^{2} \,.
\end{equation*}
We consider a stochastic forcing of the form\footnote{Although we assume that $|\ZZZ|$ is finite, our results also hold true if we consider random stirring
in all of the Fourier directions ($\ZZZ =\mathbb{Z}_+^2$) provided that we posit sufficient decay in the $\alpha_k^l$'s so that the resulting solutions are
sufficiently spatially smooth.
Note that it is for small values of $|\ZZZ|$ (and in particular when $|\ZZZ|$ is independent of $\nu_1, \nu_2, g$) that makes proof of ergodicity for \eqref{eq:B1:Vort}--\eqref{eq:B3:Vort}
difficult.  As such we focus on our attention on the `smallest' possible $\ZZZ$.
}
\begin{align}
 \sigma_\theta dW :=   \sum_{\substack{k \in \mathcal{Z}\\ l \in \{0,1\} }} \alpha_{k}^l  \sigma_k^l dW^{k,l}.
 \label{eq:stochastic:forcing:exact:form}
\end{align}
The index $\de$ of $\sigma_\de$ indicates that $\sigma_\de$ attains nontrivial values only in the second component, in other words
only the $\de$ component is directly forced.

\begin{Rmk}\label{rmk:constants:conven}\textnormal{
\begin{itemize}
\item[(i)] For the rest of the manuscript we fix the physical constants $\nu_1, \nu_2 > 0$, $g \neq 0$ and non-zero noise
coefficients $(\alpha_{k}^l)_{k \in \ZZZ, l \in \{0, 1\}}$.
Below constants $C$, $C_0$, $C^\ast$ etc. may change line by line and they implicitly
depend on $\nu_1, \nu_2 > 0$, $g \neq 0$, and $\alpha_{k}^l \neq 0$.  All other parameter dependencies
are indicated explicitly.
\item[(ii)] Here and below $d$ will always denote the number of driving Brownian motions $2 \cdot |\ZZZ |$.  Since
	the use of double indices in e.g. \eqref{eq:noise:cog:def}--\eqref{eq:stochastic:forcing:exact:form} can become
	notationally involved we  sometimes simply denote the basis for $\RR^d$
	as $\{e_1 \ldots e_d\}$.
\end{itemize}
}
\end{Rmk}

With these preliminaries in hand, the equations \eqref{eq:B1:Vort}--\eqref{eq:B3:Vort}
may be written as an abstract stochastic evolution equation on $H$
\begin{align}
  d U + (A U + B(U))dt = GUdt + \sigma_\de dW,\quad U(0) = U_0 \,,
  \label{eq:BE:abs}
\end{align}
where $U_0 \in H$.  We say that $U = U(t,U_0)$ is a solution of \eqref{eq:BE:abs} if it is $\mathcal{F}_t$-adapted,
\begin{align}
  U \in C([0,\infty); H) \cap L^{2}_{loc}([0,\infty); H^{1}) \quad  a.s.,
  \label{eq:basic:weak:reg}
\end{align}
and $U$ satisfies \eqref{eq:BE:abs} in the mild sense, that is,
\begin{align}
  U(t) = e^{-At} U_0 - \int_0^t e^{-A(t-s)} (B(U(s)) - GU(s)) ds + \int_0^t e^{-A(t-s)} \sigma_\de dW(s).
  \label{eq:BE:mild:form}
\end{align}
Note that $B, G : H^1 \to H$
and the semigroup $e^{-tA}$ maps $H$ to $H^2 \hookrightarrow H^1$.

The following proposition summarizes the basic well-posedness, regularity, and smoothness with respect to data for
\eqref{eq:BE:abs}.

\begin{Prop}[Existence/Uniqueness/Continuous Dependence on Data]
\label{prop:wellposed}
Fix $\nu_1, \nu_2 > 0$, $g  \in \RR$, and  a stochastic basis $\mathcal{S}$.
Given any $U_{0} \in H$, there exists a unique solution $U : [0,\infty) \times \Omega \to H$ of \eqref{eq:BE:mild:form} which is an $\mathcal{F}_t$-adapted process on $H$ with the regularity \eqref{eq:basic:weak:reg}.

For any $t \geq 0$ and any realization of the noise $W(\cdot,\omega)$, the map $U_0 \mapsto  U(t, U_{0})$ is Fr\'echet differentiable on $H$.
On the other hand, for every fixed $U_0 \in H$ and $t \geq 0$, $W \mapsto U(t, W)$ is Fr\' echet differentiable from
$C((0, t), \RR^{2|\ZZZ|})$ to $H$. Moreover, $U$ is (spatially) smooth for all positive times, that is, for any $t_{0} > 0$
and any $m \geq 0$,
\begin{align*}
U \in  C([t_0,\infty); H^{m}) \quad a.s.
\end{align*}
Finally,  $U$ satisfy certain moment bounds as detailed below in Lemma~\ref{lem:exp:moments}.
\end{Prop}
The well-posedness theory for \eqref{eq:BE:abs} is standard and
follows along the line of classical proofs for the stochastic 2D Navier-Stokes equations,
particularly since we are considering
the case of a spatially smooth, additive noise.  See e.g. \cite{KuksinShirikian12} for a detailed recent account close to our
setting and e.g.  \cite{Rozovskii1990,ZabczykDaPrato1992, PrevotRockner2007, HairerMattingly2011} and \cite{ConstantinFoias88,Temam2001, MajdaBertozzi2002}
for more general background on the theory of infinite
dimensional stochastic systems and mathematical fluids respectively.  Some moment estimates are less standard, but can be found in a similar
setting to ours in \cite{HairerMattingly06,Debussche2011a, KuksinShirikian12}.  For the purpose of completeness and in order to carefully track
dependencies on parameters we include some details in Appendix~\ref{sec:moment:est}.

For simplicity we set
\begin{align}
  F(U) := -AU - B(U) + GU
  \label{eq:drift:part:BE}
\end{align}
and rewrite \eqref{eq:BE:abs} in a more compact notation
\begin{align*}
   dU = F(U) dt + \sigma dW, \quad U(0) = U_0,
\end{align*}
which is particularly useful for the H\" ormander type Lie bracket computations
in Section~\ref{sec:Hormander:brak}.

\subsection{The Markovian Framework and the Main Result}
\label{sec:Markov:Framework}
With the basic well-posedness of \eqref{eq:BE:abs} in hand we next
describe its associated Markovian semigroup.  Let us first recall some further spaces.
Denote by $M_{b}(H)$ and $C_{b}(H)$ respectively the spaces of bounded measurable and bounded continuous
real valued functions on $H$ equipped with the supremum norm. We also define
\begin{align}
  \ObsSet_\ObsGC := \{ \Obs \in C^1(H) :  \|\Obs\|_\ObsGC < \infty \},
  \quad \textrm{ where } \quad
   \| \Obs \|_{\ObsGC} := \sup_{U_0 \in H} \left( \exp(- \ObsGC \|U_0\|) ( |\Obs(U_0)| + \| \nabla\Obs(U_0) \|) \right) \,,
   \label{eq:W1infty:space:norm}
\end{align}
for any $\ObsGC > 0$. Finally take $Pr(H)$ to be the collection of Borelian probability measures on $H$.

Fix $U = U(t, U_0) = U(t, U_0, W)$
and define the \emph{Markovian transition function} associated to \eqref{eq:BE:abs} by
\begin{align}
  P_t(U_0, E) = \Prb(U(t,U_0) \in E) \qquad \textrm{ for any } U_0 \in H, E \in \mathcal{B}(H), t \geq 0\,,
  \label{eq:Markov:trans:fn}
\end{align}
where $\Prb$ is defined relative to the fixed stochastic basis (cf. Proposition~\ref{prop:wellposed}) and $\mathcal{B}(H)$ is the collection of Borel sets on $H$.
We define the Markov semigroup $\{P_t\}_{t \geq 0}$ with $P_t : M_{b}(H) \to M_{b}(H)$ associated to \eqref{eq:BE:abs}
by
\begin{align}
   P_{t} \Obs(U_{0}) := \E \Obs(U(t, U_{0})) = \int_{H} \Obs(U) P_{t}(U_{0}, dU) \qquad \textrm{ for any } \Obs \in M_{b}(H), t \geq 0.
   \label{eq:Markov:semigroup}
\end{align}
By Proposition~\ref{prop:wellposed} and the dominated convergence theorem,
$\{P_{t}\}_{t \geq 0}$ is \emph{Feller} meaning that $P_{t}: C_{b}(H)
\to C_{b}(H)$ for every $t \geq 0$.  The dual operator $P_{t}^{*}$
of $P_{t}$, which maps $Pr(H)$ to itself,
is given by
\begin{align}\label{eq:Markov:dual}
   P_{t}^{*} \mu(A) := \int_{H}  P_{t}(U_{0}, A) d\mu(U_{0}),
\end{align}
over $\mu \in \mbox{Pr}(H)$. Recall that $\mu \in Pr(H)$ is an \emph{invariant measure} if it is a fixed point of $P_{t}^{*}$ for every $t \geq 0$, that is, $P_{t}^{*} \mu = \mu$.
Such a measure $\mu$ is \emph{ergodic} if $P_{t}$ is an ergodic map relative to $\mu$ for every $t \geq 0$.
In other words $P_{t} \chi_{A} = \chi_{A}$, $\mu$ a.e. implies $\mu(A) \in \{0, 1\}$.

We now formulate our main result, which  asserts
that statistically invariant states of \eqref{eq:B1:Vort}--\eqref{eq:B3:Vort} are unique
and have strong attraction properties.

\begin{Thm}\label{thm:mainresult}
  If $\ZZZ = \{(1,0),( 0,1)\}$, then there exists a unique invariant measure $\mu_\ast$ associated to \eqref{eq:BE:abs} and
for each $t \geq 0$ the map $P_t$ is ergodic relative to $\mu_\ast$. Moreover
there exists a constant $\ObsGC^*$ such that $\mu_*$ satisfies for each $\ObsGC \in (0, \ObsGC^*)$
\begin{itemize}
\item[(i)] (Mixing)  There is $\gamma = \gamma(\ObsGC) > 0$ and $C = C(\ObsGC)$ such that
\begin{align}
     \left| \E \Obs(U(t, U_0)) - \int_H \Obs(\bar{U}) d \mu_\ast(\bar{U}) \right| \leq C \exp(- \gamma t +\ObsGC \|U_0\|) \| \Obs\|_{\eta} \,.
     \label{eq:main:thm:mixing:cond}
\end{align}
for any $\Obs \in \ObsSet_\ObsGC$, $U_0 \in H$ and any $t \geq 0$.
\item[(ii)] (Weak law of large numbers) For any $\Obs \in \ObsSet_\ObsGC$ and any $U_0 \in H$
\begin{align}
\lim_{T \to \infty} \frac{1}{T} \int_0^T \Obs(U(t, U_0))\, dt = \int_H \Obs(\bar{U}) d \mu_\ast(\bar{U}) =: m_\Obs, \quad \textrm{  in probability}
\,.
\label{eq:WLLN}
\end{align}
\item[(iii)] (Central limit theorem) For every $\Obs \in \mathcal{B}_\eta$ and every $U_0 \in H$,
\begin{align}
\lim_{T \to \infty} \Prb \left( \frac{1}{\sqrt{T}} \int_0^T (\Obs(U(t, U_0)) - m_\Obs) \, dt < \xi \right) = \mathcal{X}(\xi)
\label{eq:CLT}
\end{align}
for any $\xi \in \RR$ where $\mathcal{X}$ is the distribution function of a normal random variable with zero mean and variance equal to
\begin{align*}
\lim_{T \to \infty} \frac{1}{T} \E \left( \int_0^T (\Obs(U(t, U_0)) - m_\Obs) \, dt \right)^2 \,.
\end{align*}
\end{itemize}
\end{Thm}

\begin{Rmk}
\mbox{}
\begin{itemize}
\item[(i)]  We take $\ZZZ = \{ (0,1), (1,0) \}$ in the statement of Theorem~\ref{thm:mainresult}
for simplicity and clarity of exposition only. There are many other choices of finite $\ZZZ$ that imply our main results.
For example, our approach applies trivially to any $\ZZZ$ with $\{ (0,1), (1,0) \} \subset \ZZZ$.   More generally, using
combinatorial arguments as in \cite{HairerMattingly06} and  straightforward modifications
of our proofs, one can show that our results hold if the integer linear combinations of
elements in $\ZZZ$ generate $\ZZ^2$, for an easily verifiable criterion see \cite[Remark 2.2]{HairerMattingly06}.

\item[(ii)] The most interesting case for the study of ergodicity for the Boussinesq system, is to consider
stochastic forcing acting through the temperature equation only.  Indeed, if the random stirring acts exclusively
through the momentum equation (i.e. in \eqref{eq:B1:Vort}), then the temperature $\theta$, which is advected by the flow
in \eqref{eq:B3:Vort}, decays exponentially.  Thus, in this case, conditions on the configuration of the forcing
and the associated analysis is very close to \cite{HairerMattingly06}.
On the other hand if the noise acts in both equations the proof is similar to ours and in many ways easier.\footnote{Indeed one observes
that $[[F(U), \sigma_k^l], \psi_{k'}^{l'}] \propto B(\psi_{k'}^{l'}, \sigma_k^l) \neq 0,  [[F(U), \psi_k^l], \psi_{k'}^{l'}] \propto B(\psi_{k'}^{l'}, \psi_k^l) \neq 0$ which allows us to generate
new pure directions in both the temperature and vorticity components of the phase space; compare with Figure~\ref{fig:brak:1} in Section~\ref{sec:Hormander:brak} and the discussion
in Section~\ref{sec:Hor:int}.}  We remark that the case
when the forcing is non-degenerate (nontrivial on all Fourier modes) in both the momentum and temperature equations, was
addressed in \cite{LeeWu2004} via coupling methods closely following the approach in \cite{EMattinglySinai2001}.  See also
\cite{Ferrario1997}.

\item[(iii)] \label{rem:forcing} We emphasize that, when the random perturbation acts in the temperature equation only,
this leads to a different geometric
criteria for the noise structure compared to \cite{HairerMattingly06}.
To see this difference at a heuristic level
we write \eqref{eq:B1:Vort}--\eqref{eq:B3:Vort} in the Fourier representation:
\begin{align}
	&\frac{d \om_k}{dt}   + \nu |k|^2\om_{k} + \sum_{l+m=k}\langle l,m^{\perp}\rangle\Big(\frac{1}{|l|^2}-\frac{1}{|m|^2}\Big)\om_{l}\om_{m} = -ig \cdot k \de_k,
	\label{eq:om:fourier}\\
  &d \de_k + \left( \eta |k|^2\de_{k}  -\sum_{l+m=k}\frac{\langle l,m^{\perp}\rangle}{|m|^2}\de_{l}\om_{m} \right) dt = \indFn{k \in \mathcal{Z}} dW^k.	\label{eq:de:fourier}
\end{align}
Observe that at first only Fourier modes of $\de$ in $\ZZZ$ are excited.  Then, through the buoyancy term
on the right hand side of \eqref{eq:om:fourier}, the Fourier modes of $\om$ in $\ZZZ$
become excited.  This is a purely formal argument as at the same time many modes in $\theta$ become excited.
If all elements of $\mathcal{Z}$ have the same norm,
such an excitation is not sufficient for the nonlinearity in
 \eqref{eq:om:fourier}, acting on its own, to propagate the noise to higher Fourier modes.\footnote{In fact, this structure in the
 nonlinearity  is the reason for the additional condition that two modes of different length need to be stochastically forced in \cite{HairerMattingly06}}
However, excitation in
the Fourier modes of $\omega$ in $\mathcal{Z}$ propagates to higher Fourier
modes in $\theta$ via the nonlinearity of \eqref{eq:de:fourier}; here the norm restriction is clearly absent. Thus,
there is an additional mixing mechanism in the Boussinesq system compared to the Navier-Stokes equation.

\item[(iv)] The class of functions for which the mixing condition \eqref{eq:main:thm:mixing:cond} holds
is slightly restrictive.  While it does allow for observables like individual Fourier coefficients of the solution or
the total energy of solutions, a further analysis is required to extend to $\Obs$'s that involve pointwise
spatial observations of the flow, for example `structure functions'.  We leave these questions for future work.
\end{itemize}
\end{Rmk}

\subsection{Existence and Uniqueness of Invariant Measures and the Asymptotic Smoothing of the Markov Semigroup}
\label{sec:MH:ASF:to:uniqueness}

Following \cite{HairerMattingly06} (and cf.  \cite{HairerMattingly2008, KomorowskiWalczuk2012}), we
explain how the proof of  Theorem~\ref{thm:mainresult} can be
essentially reduced to establishing a time asymptotic  gradient estimate,
\eqref{eq:main:grad:est}, on the Markov semigroup.

By the Krylov-Bogoliubov averaging method \cite{KryloffBogoliouboff1937}
it is immediate to prove the existence of an invariant measure  in the present setting.
Indeed, fix any $U_{0} \in H$, $T > 0$ and define the probability measures $\mu_T \in \Pr(H)$ as
\begin{align*}
  \mu_{T}(A) = \frac{1}{T} \int_{0}^{T} \Prb(U(t,U_{0}) \in A) dt
\qquad (A \in \mathcal{B}(H))\,.
\end{align*}
From \eqref{eq:est:h11} and \eqref{eq:est:h12} it follows that
\begin{align}
    \min\{ \nu_1, \nu_2\} \frac{1}{T} \E \int_{0}^{T} \| U \|^{2}_{H^{1}} dt \leq \frac{\|U_{0}\|^{2}}{T} + \|\sigma_\theta\|^{2}.
   \label{eq:feller:bnd}
\end{align}
Thus for any $R > 0$ the set $B_{H^1}(R) := \{U \in H: \|U\|_{H^1} \leq R\}$ is compact in $H$ and by the Markov inequality
\begin{align*}
\mu_T(B_{H^1}(0, R)) = \frac{1}{T} \int_{0}^{T} \Prb(\|U(t,U_{0})\|_{H^1} \leq R) dt
&\geq 1 - \frac{1}{T R^2} \E \int_{0}^{T} \| U \|^{2}_{H^{1}} dt  \\
&\geq
1 -  \frac{1}{R^2 \cdot \min\{ \nu_1, \nu_2\}} \left( \frac{\|U_{0}\|^{2}}{T} + \|\sigma_\theta\|^{2} \right) \,,
\end{align*}
and therefore $\{ \mu_{T}\}_{T \geq 1}$
is tight, and hence weakly compact.  Making use of the Feller property it then follows that
any weak limit of this sequence is an invariant measure of \eqref{eq:BE:abs}.\footnote{
Note that the Feller property and the bound \eqref{eq:feller:bnd} also show that the set of invariant
measures $\mathcal{I}$ for \eqref{eq:BE:abs} is a compact, convex set.  Since
the extremal points of
$\mathcal{I}$ are
ergodic invariant measures for \eqref{eq:BE:abs}, we therefore infer the existence of an ergodic invariant measure
for \eqref{eq:BE:abs}.  This also shows that if the invariant measure is unique, it is necessarily ergodic.
See e.g. \cite{ZabczykDaPrato1996} for further details.}

We now turn to the question of uniqueness which in contrast to existence is highly non-trivial.
The classical theoretical foundation to our
approach is the Doob-Khasminskii theorem,  see \cite{ZabczykDaPrato1996}.
While this approach has been fruitful for a stochastic perturbations
acting on all of the Fourier modes (see e.g. \cite{ZabczykDaPrato1996}); it requires an instantaneous (or at least finite time) smoothing of $P_{t}$,
known as the \emph{strong Feller property}.\footnote{More precisely, $P_t$ is said to be strong
Feller for some $t > 0$ if $P_t: M_b(H) \to C_b(H)$.}  This property is not
expected to hold in the current hypo-elliptic setting.

In recent works \cite{HairerMattingly06, HairerMattingly2008,
HairerMattingly2011} it has been shown that the strong Feller property
can be replaced by a much weaker notion.  The following theorem from
is \cite{HairerMattingly06} is the starting point for all of the work that follows below.

\begin{Thm}[Hairer--Mattingly \cite{HairerMattingly06}]
	\label{thm:MH:ASF:WI}
   Let $\{P_{t}\}_{t \geq 0}$ be a Feller Markov semigroup
   on a Hilbert space $H$ and assume that the set of invariant measures $\mathcal{I}$ of $\{P_{t}\}_{t \geq 0}$
   is compact.   Suppose that
   \begin{itemize}
      \item[(i)] the semigroup $\{P_{t}\}_{t \geq 0}$ is \emph{weakly irreducible} namely there exists $U_0 \in H$ such that
   $\mu( B(\epsilon, U_0)) > 0$ for every $\epsilon > 0$ and every $\mu \in Pr(H)$ which is invariant under $P_t^*$.  In other words
   there exists a point common to the support of every invariant measure.
   \item[(ii)] There exists a non-decreasing sequence $\{t_n\}_{n \geq 0}$ and a sequence $\{\delta_n\}_{n \geq 0}$, with $\delta_n \to 0$ such that
     \begin{align}
	\| \nabla P_{t_n} \Obs(U_0) \| \leq C \left( \| \Obs\|_\infty  + \delta_n \| \nabla \Obs\|_\infty\right) \,,
	\label{eq:main:grad:est:abs}
  \end{align}
  for every $\Obs \in C^1_b(H)$ and where the constant $C$ may depend on $\|U_0\|$ (but not on $\Obs$).\footnote{
  Actually, the `gradient estimate' \eqref{eq:main:grad:est:abs} is a sufficient condition for a more general notion of infinite time
smoothing referred to as the \emph{asymptotically strong Feller property}.   A precise topological definition
using the Kantorovich-Wasserstein distance can be found in \cite{HairerMattingly06}.}
   \end{itemize}
   Then the collection of invariant measures $\mathcal{I}$ contains at most one element.
\end{Thm}

In our situation the proof of (i) is more or less standard and follows precisely as in \cite{EMattingly2001, ConstantinGlattHoltzVicol2013}.
The main difficulty is to establish the asymptotic smoothing property (ii) of Theorem~\ref{thm:MH:ASF:WI}.
We prove the following stronger version  of \eqref{eq:main:grad:est:abs}, which is also useful for other parts of Theorem \ref{thm:mainresult} the proof
of the mixing \eqref{eq:main:thm:mixing:cond} and pathwise convergence properties \eqref{eq:WLLN}, \eqref{eq:CLT}.
We recall the convention in Remark \ref{rmk:constants:conven}.

\begin{Prop}
   \label{prop:grad:est:MS}
  For every $\eta, \gamma_0 > 0$ and every $U_0 \in H$, the Markov semigroup $\{P_t\}_{t \geq 0}$  defined by \eqref{eq:Markov:semigroup} satisfies the estimate
  \begin{align}
	\| \nabla P_{t} \Obs(U_0) \| \leq C \exp( \eta \|U_0\|^2) \left(  \sqrt{P_t (| \Obs |^2)(U_0)}  +  e^{- \gamma_0 t} \sqrt{P_t (\| \nabla \Obs\|^2)(U_0)}\right),
	\label{eq:main:grad:est}
  \end{align}
  for every $t \geq 0$ and $\Obs \in C^1_b(H)$, where $C = C(\eta, \gamma_0)$ is independent of $t$ and $\Obs$.
\end{Prop}
With Proposition~\ref{prop:grad:est:MS} the uniqueness of the invariant measure follows immediately from Theorem~\ref{thm:MH:ASF:WI}.
With slightly more work we can also use \eqref{eq:main:grad:est} to establish the attraction
properties (i)--(iii) in Theorem~\ref{thm:mainresult}.  Since this mainly requires the introduction of some further abstract machinery
from  \cite{HairerMattingly2008, KomorowskiWalczuk2012} we postpone the rest of the proof of Theorem~\ref{thm:mainresult} to
final part of Section~\ref{sec:mixing}.

\section{Gradient Estimates for the Markov Semigroup}
\label{sec:grad:control}

In this section we explain how the estimate on $\nabla P_t \Obs$ in \eqref{eq:main:grad:est} can be translated
to a control problem through the Malliavin integration by parts
formula.    This lead us to study the so called
 Malliavin covariance matrix $\MM$, which links the existence of a desirable
control to the properties of successive H\"ormander-type Lie brackets of
vector fields (on $H$) associated to \eqref{eq:BE:abs}.  Suitable spectral bounds
for $\MM$ are given in Proposition~\ref{prop:mal:cor:1}
and we conclude this section by explaining how these bounds can be used in conjunction
with a control built around $\MM$  to complete the proof of Proposition~\ref{prop:grad:est:MS}.

The involved proof of Proposition~\ref{prop:mal:cor:1} is delayed
for Sections~\ref{sec:Mal:spec:bnds}, \ref{sec:Hormander:brak}, \ref{sec:Braket:Est:Mal:Mat}
below. As we already noted in the introduction, although the statement of Proposition~\ref{prop:mal:cor:1} looks similar to
corresponding results in \cite{MattinglyPardoux1,HairerMattingly06,BakhtinMattingly2007,HairerMattingly2011}, the proof is significantly different due to the particular
nonlinear structure of \eqref{eq:B1:Vort}--\eqref{eq:B3:Vort}.  As such Proposition~\ref{prop:mal:cor:1} constitutes
the main mathematical novelty of this work.

\subsection{Deriving the Control Problem}
\label{sec:derive:control}
Let $U = U(\cdot, U_0)$ be the solution of \eqref{eq:BE:abs} and assume the convention from Remark \ref{rmk:constants:conven}
where we let $d := 2 \cdot |\ZZZ|$.
Then for any
$\Obs \in C^1_b(H)$, $\xi \in H$ we have\footnote{For differentiability of
$U_0 \mapsto U(t, U_0)$ see Proposition \ref{prop:wellposed} and \cite[Section 3.3]{HairerMattingly2011}}
\begin{align}
   \nabla P_t \Obs( U_0) \cdot \xi &= \E (\nabla \Obs(U(t,U_0)) \cdot \JJ_{0,t} \xi), \qquad t \geq 0 \,,
   \label{eq:grad:comp:msg:1}
\end{align}	
where for $0 \leq s \leq t$, $\JJ_{s,t} \xi$ denotes the unique solution of
\begin{align}
& \partial_t \rho + A \rho + \nabla B(U) \rho = G \rho, \quad \rho(s) = \xi,
\label{eq:def:J}
\end{align}
and $\nabla B(U) \rho := B(U, \rho) + B(\rho,U)$.

The crucial step in establishing \eqref{eq:main:grad:est} is to `approximately remove' the gradient from $\Obs$ in \eqref{eq:grad:comp:msg:1}.
As such we seek to (approximately)
identify $\JJ_{0,t} \xi$ with a Malliavin derivative of some suitable
random process and integrate by parts, in the Malliavin sense.  In Appendix~\ref{subsect:malliavin}
we recall some elements of this calculus which are used throughout this section.
For an extended treatment of the Malliavin theory we refer to e.g. \cite{Bell1987,Malliavin1997,Nualart2009,Nualart2006}.

Recall, that in our situation the Malliavin derivative, $\MD: L^2(\Omega, H) \to L^2(\Omega; L^2(0,t;\RR^{d})\otimes H)$
satisfies
\begin{align*}
	\langle \DM U, v\rangle_{L^2([0, T], \RR^d)} =
	\lim_{\epsilon \to 0} \frac{1}{\epsilon} \left(  U\bigl(T, U_0, W + \epsilon \smallint _0^{\cdot} v ds\bigr) - U(T, U_0, W) \right).
\end{align*}
We may infer that for $v \in L^2(\Omega ; L^2([0,T]; \RR^{d}))$ one has (cf. \cite{HairerMattingly2011})
\begin{align}\label{eq:al:def:A}
\langle \DM U, v\rangle_{L^2([0, T], \RR^d)} = \int_{0}^{T} \JJ_{s,T} \sigma_\de v(s) \, ds \,,
\end{align}
and hence, by the Riesz representation theorem,
\begin{equation}
\DM_s^j U(T) =
\JJ_{s,T} \sigma_\de e_j \qquad  \textrm{ for any } s \leq T,\, j =1, \dots d,
\label{eq:def:Mal:J}
\end{equation}
where the linearization $\JJ_{r, t} \xi$ is the solution of \eqref{eq:def:J}, $\sigma_\de$ is given by \eqref{eq:noise:cog:def}, and $\{e_j\}_{j = 1,\ldots, d}$ is the standard basis of
$\RR^{d}$.\footnote{Using the Malliavin chain rule and \eqref{eq:Fs:t:meas:zero:cond} we apply $\DM$ to \eqref{eq:BE:abs}
and observe, at least formally that
\begin{align*}
  \MD_s^j U(T) + \int_s^T \bigl( A \MD^j_s U(r) + \nabla B(U(r)) \MD^j_s U(r)  - G \MD^j_s U(r)\bigr) dr = \sigma_\theta e_j,
\end{align*}
for $s < T$ and $j = 1, \ldots ,d$.}
Here and below, we adopt the standard notation $\MD_s^jF := (\MD F)^j(s)$, that is, $\MD_s^jF$ is the $j$th component of $\MD F$ evaluated
at time $s$.

Motivated by \eqref{eq:al:def:A}, we define the random operator $\AAA_{s,t} : L^2( [s,t]; \RR^{d}) \to H$
by
\begin{align}
\AAA_{s,t} v := \int_s^t \JJ_{r,t} \sigma_\de v(r) dr.
\label{eq:A:op:def}
\end{align}
Notice that, by the Duhamel formula, for any $0\leq s < t$ the function $\rho(t) := \AAA_{s,t} v$ satisfies
\begin{align*}
 \pd_{t}\rho + A\rho + \nabla B(U)\rho = G\rho + \sigma_\theta v , \qquad  \rho(s) = 0 \,.
\end{align*}

With these preliminaries we now continue the computation started in
\eqref{eq:grad:comp:msg:1}.  Using the Malliavin chain rule and integration by parts formula,
as recalled in \eqref{eq:M:chain:rule1}, \eqref{eq:M:IBP:rule:trad:notation},
we infer that for any $t \geq 0$ and any suitable (Skorokhod integrable) $v  \in L^2(\Omega \times [0,t], \RR^{d})
$\footnote{Note that for non-adapted $v$, $\int_0^t v \cdot dW$
in \eqref{eq:control:derivation:Mal:calc} is understood
as a stochastic integral in a generalized sense; see Appendix~\ref{subsect:malliavin} below or e.g. \cite{Nualart2006} for further details.}
\begin{align}
   \nabla P_t \Obs( U_0) \cdot \xi
   &= \E (\nabla \Obs(U) \cdot (\AAA_{0,t} v + \JJ_{0,t} \xi - \AAA_{0,t} v))
   \notag\\
    &=  \E (\nabla \Obs(U) \cdot \langle \DM U,  v  \rangle) +
    \E (\nabla \Obs(U) \cdot (\JJ_{0,t} \xi - \AAA_{0,t} v))
   \notag\\
    &= \E (\langle \DM \Obs(U),  v  \rangle_{L^2([0, t], \RR^d)}) +  \E (\nabla \Obs(U) \cdot (\JJ_{0,t} \xi - \AAA_{0,t} v))
   \notag\\
   &= \E \left(\Obs(U)\int_0^t v \cdot dW  \right) +  \E (\nabla \Obs(U) \cdot (\JJ_{0,t} \xi - \AAA_{0,t} v)),
   \label{eq:control:derivation:Mal:calc}
\end{align}
where for brevity we have denoted $U = U(t, U_0)$.
Observe that $\rho(t) := \JJ_{0,t} \xi - \AAA_{0,t} v$ satisfies:
\begin{align}
& \partial_t \rho + A \rho + \nabla B(U) \rho = G \rho - \sigma_\de v, \quad \rho(0) = \xi.
\label{eq:JmA:def:control}
\end{align}
Directly from \eqref{eq:control:derivation:Mal:calc}
we obtain (recalling the notation \eqref{eq:Markov:semigroup})
\begin{align*}
 |\nabla P_t \Obs( U_0) \cdot \xi| \leq
  \left( \E \left| \int_0^t v \cdot dW \right|^2 \right)^{1/2}  \sqrt{P_t |\Obs |^2(U_0)}   + \left(\E \|\rho(t, \xi, v)\|^2\right)^{1/2} \sqrt{ P_t \|\nabla \Obs\|^2(U_0)},
\end{align*}
for any $t \geq 0$ and $\xi \in H$.
As such, \eqref{eq:main:grad:est}
has been translated to the following control problem:
For each $\gamma_0, \eta > 0$ and each unit length element $\xi \in H$, find a (locally Skorokhod integrable) $v = v(\xi) \in L^2(\Omega; L^2_{loc} ([0,\infty), \RR^{d}))$
such that
\begin{align}
  \sup_{\|\xi\| = 1} \E \|\rho(t, \xi, v)\|^2 &\leq C e^{-\gamma_0 t} \exp( \eta \|U_0\|^2)  \label{eq:decay}
\end{align}
and
\begin{align}
 \sup_{\|\xi\| = 1, t \geq 0} \E \left| \int_0^t v \cdot dW \right|^2 &\leq C \exp( \eta \|U_0\|^2) \label{eq:bound}\,,
\end{align}
where $C = C(\gamma_0, \eta)$ is independent of $t$.

\subsection{Choosing the Control}
\label{subsec:control:choice}

In the case when sufficiently many directions in Fourier space are
forced or if $\nu_1, \nu_2$ are sufficiently large, we can choose the control $v$ in \eqref{eq:JmA:def:control} from
the use of determining modes.  See \cite{FoiasProdi1967} or more recently \cite{FoiasManleyRosaTemam01}.
Specifically, we might choose $\sigma_\de v = \lambda P_N \rho$ in  \eqref{eq:JmA:def:control}, where $P_N$ is the projection onto the subspace $H_N$ defined in \eqref{eq:H:N:finite:dim}
and $\lambda > 0$ is a sufficiently large constant.
The construction of such a control
relies on the assumption that all of the modes with wave-numbers less $N$ are directly forced, that is, they are in the range of $\sigma_\de$.
We remark, moreover, that $N$ is a function of $\nu_1, \nu_2 > 0$ and it diverges to infinity as $\nu_1$ and $\nu_2$ approach zero.
The idea of using determining modes in the context of the ergodic  theory of mathematical fluids equations (and other nonlinear SPDEs)
has played a central role in a number of recent works.  See e.g. \cite{Mattingly2,Mattingly03, HairerMattingly06, KuksinShirikian12,ConstantinGlattHoltzVicol2013}.

If the range of $\sigma_\de$ is not sufficiently large,
we are hindered by the fact that we cannot `directly control' all the low (unstable) modes
that are not dissipated by the diffusion. A different approach
inspired by the finite dimensional case (cf. \cite{Hairer2011}) would be to try to find an exact control.  To achieve this
we would seek for each $\xi \in H$ a corresponding $v$ such
$\JJ_{0,t} \xi = \AAA_{0,t} v$.  With the ansatz that $v$ has the form $v = \AAA_{0,t}^*\eta$ for some $\eta \in H$ (where $\AAA^*_{0,t}$ the
adjoint of $\AAA_{0,t}$ defined below in \eqref{eq:A:star:op:def}) we could choose the control
\begin{align}\label{eq:def:v:mot}
	v := \AAA_{0,t}^* ( \AAA_{0,t} \AAA_{0,t}^*)^{-1} \JJ_{0,t} \xi.
\end{align}
The object $\MM_{0,t} :=\AAA_{0,t} \AAA_{0,t}^*$, referred to as the Malliavin
covariance matrix, plays an important role in the theory of stochastic analysis.
If one can establish the invertibility of $\MM_{0,t}$ one finds an exact control
in \eqref{eq:JmA:def:control} and with suitable bounds on $v$ one shows that
the Markov semigroup is smoothing in finite time (i.e. it is strongly Feller).

Sufficient conditions for the invertibility of the Malliavin matrix are well understood
for finite dimensional problems. However, this invertibility is much harder to deduce in infinite
dimensions and may not hold in general.\footnote{Although $\MM_{0, t}$ is a linear, non-negative definite,
self adjoint operator on a Hilbert space, which
can be shown to be non-degenerate it is difficult to quantify the range of $\MM_{0, t}$, or equivalently we are unable to characterize the domain of
$(\MM_{0, t})^{-1}$.}
Instead, following the insights in \cite{HairerMattingly06}, we now combine the strategy identified in \eqref{eq:def:v:mot} with the use of determining modes \cite{FoiasProdi1967}, and use a
Tikhonov regularization of the Malliavin matrix to construct a control $v$, and corresponding $\rho$, which satisfy \eqref{eq:decay}--\eqref{eq:bound}.

To make this more precise we first define several random operators.
For any $s < t$, let $\AAA_{s,t}^*: H \to L^2([s,t]; \RR^{d})$ be the adjoint of
$\AAA_{s,t}$ defined in \eqref{eq:A:op:def}.  We observe that
\begin{align}
  (\AAA^*_{s,t}\xi)(r) = \sigma^*_\de \JJ_{r,t}^* \xi   =: \sigma^*_\de \KK_{r,t} \xi
  \qquad  \textrm{ for any } \xi \in H, r \in [s, t],
  \label{eq:A:star:op:def}
\end{align}
where $\sigma_\de^* : H \to \RR^{d}$ is the adjoint of $\sigma_\de$ defined in \eqref{eq:noise:cog:def}.
Here, for $s < t$,  $\KK_{s,t} \xi = \JJ^*_{s,t} \xi$ is the solution of the `backward'
system (see \cite{HairerMattingly2011})
\begin{align}
  \pd_s \rho^* = A \rho^* + (\nabla B(U(s)))^* \rho^* - G^* \rho^*
  = - (\nabla F(U))^* \rho^*, \quad \rho^{*}(t) = \xi \,.
  \label{eq:def:K:backwards:eq}
\end{align}
We then define
the Malliavin Matrix
\begin{align}
\MM_{s,t} := \AAA_{s,t} \AAA_{s,t}^*: H \to H.  \label{eq:Mal:def}
\end{align}

We now build the control $v$ and derive the associated $\rho$ in \eqref{eq:JmA:def:control} using the following iterative construction.
Denote by $v_{s,t}$ the control $v$ restricted to the time interval $[s, t]$ and let $\rho_n := \rho(n)$.
Observe that $\rho_0 = \xi \in H$ by definition.  For each even non-negative integer $n \in 2\NN$,
having determined $\rho_n$ and $v_{0, n}$, we set
\begin{align}
   v_{n,n+1}(r) &= (\AAA_{n,n+1}^* (\MM_{n,n+1}  + I \beta)^{-1}  \JJ_{n,n+1}\rho_{n})(r), \qquad v_{n+1, n+2}(r) = 0 \,,
   \label{eq:control:def:iterative}
\end{align}
for $r \in [n, n+2]$. If we denote
\begin{align*}
\MTI_{n,n+1}  := \beta (\MM_{n,n+1}  + I \beta)^{-1} \,,
\end{align*}
then using \eqref{eq:JmA:def:control} we determine $\rho_{n+2}$ according to
\begin{align}
 \rho_{n+2} &:= \JJ_{n+1, n+2} \rho_{n+1} \notag\\
      &= \JJ_{n+1,n+2} ( \AAA_{n,n+1} \AAA_{n,n+1}^* (\MM_{n,n+1}  + I \beta)^{-1} \JJ_{n,n+1} + \JJ_{n, n+1}) \rho_n \notag\\
 	&=  \JJ_{n+1,n+2}  \MTI_{n,n+1}  \JJ_{n,n+1} \rho_{n}.
 \label{eq:two:tm:step:cont:err}
\end{align}

\subsection{Spectral Properties of $\MM$ and Decay}
\label{sec:approx:control:cost}

Having defined the control $v$, and the associated error $\rho$,  by
\eqref{eq:control:def:iterative} and \eqref{eq:two:tm:step:cont:err} respectively, we now state and prove (modulo a
spectral bound on $\MM_{n, n+1}$, Proposition \ref{prop:mal:cor:1})
the key decay estimate on $\rho$.
This estimate is used in Section~\ref{sec:proof:main:prop} to complete the proof of Proposition~\ref{prop:grad:est:MS}.

\begin{Lem}
\label{lem:rho:bound}
For any $\varpi, \dStep > 0$,  there  exists $\beta = \beta(\varpi, \dStep) >0$
which determines $\rho$ in \eqref{eq:two:tm:step:cont:err} so that for every even $n \geq 0$,
\begin{align}
  \E ( \|\rho_{n+2}\|^8 | \mathcal{F}_n) \leq \dStep \exp(\varpi \| U(n)\|^2) \|\rho_n\|^8.
  \label{eq:stepwise:cond:decay:est}
\end{align}
Moreover, we have the block adapted structure
\begin{align}
  \rho(t), v(t) \textrm{ are } \mathcal{F}_{\vartheta(t)} \textrm{ measurable }
  \label{eq:block:adapted:struct:rho:v}
\end{align}
where $\vartheta: \RR^+ \to \RR^+$
\begin{align}
       \vartheta(t) :=
  \begin{cases}
       \lceil t \rceil& \quad \textrm{ when } \lceil t \rceil \textrm{ is odd, }\\
       t & \quad \textrm{ when } \lceil t \rceil \textrm{ is even, }
  \end{cases}
  \label{eq:block:adapt:fn}
\end{align}
and $\lceil t \rceil$ is the smallest integer greater than or equal to $t$.
\end{Lem}

\begin{Rmk}
\label{rmk:prelim:lem:rho:bound}
We choose the exponent 8 as it is sufficient for the estimates on $v$; similar estimates are valid for any power greater or equal to two.
\end{Rmk}

To prove Lemma \ref{lem:rho:bound}, we show that the control $v$, when it is active, is effective in
 pushing energy into small scales where it is
 dissipated by diffusion.
 To make this precise recall the definition of $P_N$, $Q_N$, and $H_N$ from \eqref{eq:H:N:proj:ops} and \eqref{eq:H:N:finite:dim}.
  Fix $N = N(\dStep, \varpi)$ specified below and for $n \in 2\NN$ split $\rho_{n+2} = \rho_{n+2}^L +\rho_{n+2}^H$,
defining
\begin{align}
\rho_{n+2}^H &:=  \JJ_{n+1,n+2} Q_N \MTI_{n,n+1} \JJ_{n,n+1}\rho_{n}, \label{eq:rho:high:n}\\
\rho_{n+2}^L &:=  \JJ_{n+1,n+2}  P_N \MTI_{n,n+1} \JJ_{n,n+1}\rho_{n}.
\label{eq:rho:low:n}
\end{align}
While for large $N$
estimates for $\rho_{n+2}^H$ essentially make use of the parabolic character of \eqref{eq:def:J},
establishing suitable bounds on $\rho_{n+2}^L$ requires a detailed understanding of the operator $\MTI_{n,n+1}$.  We
need to show, for sufficiently small $\beta$, that  $\MTI_{n,n+1}$ indeed pushes
energy into small scales, that is, $\|P_N \MTI_{n,n+1}\|$ is small.  The following lemma from
\cite[Lemma 5.14]{HairerMattingly2011},
shows that this in turn follows from uniform positivity of $\MM_{n,n+1}$ on a cone around $H_N$.

\begin{Lem}
\label{lem:good:t:reg:as}
Suppose that $\MM$ is a positive, self-adjoint linear operator on a separable Hilbert space
$H$.  Suppose that for some $\alpha, \epsilon >0$ and $N \in \mathbb{N}$
we have that
\begin{align}\label{eq:lb:abs}
  \inf_{\phi \in \mathcal{S}_{\alpha, N}} \frac {\langle \MM \phi, \phi \rangle}{\|\phi\|^{2}} \geq \epsilon,
\end{align}
where
\begin{align}
	\mathcal{S}_{\alpha, N} := \left\{ \phi \in H: \|P_{N} \phi \|^{2} \geq \alpha \|\phi\|^{2} \right\}.
	\label{eq:test:fn:set:N:alpha}
\end{align}
Then, for any $\beta > 0$,
\begin{align*}
   \| P_{N} \beta (\MM + I \beta)^{-1} \| \leq \alpha \vee \sqrt{\beta/\epsilon}.
\end{align*}
\end{Lem}

The next proposition shows that \eqref{eq:lb:abs} holds true for $\MM_{n, n+1}$ on a large subset of the probability space.

\begin{Prop}
\label{prop:mal:cor:1}
Let $\MM_{n,n+1}$ be as in \eqref{eq:Mal:def}, relative to $U$ solving \eqref{eq:BE:abs}. For any $N \geq 1$, $\alpha \in (0,1]$, and $\eta > 0$
there exists a positive constant
$\epsilon^{*} = \epsilon^{*}(\alpha, \eta, N) > 0$, such that, for any $n \geq 0$, and any
$0 < \epsilon  \leq \epsilon^{*}$
\begin{align}
  \Prb \left( \inf_{ \phi \in \mathcal{S}_{\alpha, N}} \frac{\langle \MM_{n,n+1} \phi, \phi \rangle}{\| \phi \|^2} < \epsilon | \mathcal{F}_n \right)
  \geq r(\epsilon) \exp( \eta \|U(n)\|^2),
  \label{eq:mal:matrix:inv:bnd:cond}
\end{align}
where $\mathcal{S}_{\alpha, N}$ is defined by \eqref{eq:test:fn:set:N:alpha} and
$r = r(\alpha, N, \eta) : (0, \epsilon^*] \to (0, \infty)$ is a non-negative, decreasing function
 with $r(\epsilon) \to 0$ as $\epsilon \to 0$.  We emphasize that
$r$ is independent of $n$.
\end{Prop}

Proposition~\ref{prop:mal:cor:1} is a direct consequence of the Markov property and Theorem~\ref{thm:mind:blowing:conclusion} below.
While similar results have appeared in previous works, the proof required us to develop a novel approach due to the particular nonlinear
structure in \eqref{eq:B1:Vort}--\eqref{eq:B3:Vort}.

We are now prepared to prove Lemma~\ref{lem:rho:bound}.

\begin{proof}[Proof of Lemma \ref{lem:rho:bound}]
We use the splitting $\rho_{n+2} = \rho^H_{n+2} + \rho^L_{n+2}$ from \eqref{eq:rho:high:n}, \eqref{eq:rho:low:n}.
The constant $N$ appearing in the definition of this splitting is fixed in the estimate on $\rho^H_{n+2}$ which we address first.

By the positive definiteness of $\MM_{n,n+1}$, it follows that $\|\MTI_{n,n+1} \| \leq 1$ almost surely,
for any $\beta > 0$.  Then, as $\rho_n$ is $\mathcal{F}_n$-measurable, from \eqref{eq:lin:growth:est},
\eqref{eq:lin:growth:high:mode:est}, and \eqref{eq:exp:no:time:growth} we infer
\begin{align}
    \E (\| \rho_{n +2}^H \|^8 | \mathcal{F}_n)
    \leq&   \|\rho_{n}\|^8 \cdot \E\left( \E(\|\JJ_{n+1,n+2} Q_N \|^8| \mathcal{F}_{n+1}) \cdot \|\JJ_{n,n+1} \|^8  | \mathcal{F}_n \right)
    \notag\\
    \leq& \frac{\dStep}{2^8}   \exp( \varpi  \|U(n)\|^2)  \|\rho_{n}\|^{8}
    \,,
    \label{eq:rho:H:n:est:gamma:eta}
\end{align}
for appropriate $N = N(\varpi, \dStep)$.  Fix such an $N$ in \eqref{eq:rho:high:n}, \eqref{eq:rho:low:n}.
Note that the bound \eqref{eq:rho:H:n:est:gamma:eta} holds independently of the value of $\beta$
appearing in \eqref{eq:control:def:iterative}.

Next, we estimate $\rho_{n+2}^L$.  By \eqref{eq:lin:growth:est},
we infer
\begin{align}
    \E (\| \rho_{n +2}^L \|^8 | \mathcal{F}_n)
    \leq&
       \|\rho_{n}\|^8  \cdot  \left(\E(\|\JJ_{n+1,n+2} \|^{24} | \mathcal{F}_n) \right)^{1/3}
       \left( \E (\|P_N \MTI_{n,n+1}\|^{24} | \mathcal{F}_n) \right)^{1/3}
       \left( \E (\| \JJ_{n,n+1}\|^{24}| \mathcal{F}_n) \right)^{1/3}
    \notag\\
    \leq& C^* \|\rho_{n}\|^8  \exp(\varpi/2  \|U(n)\|^2) \left( \E (\|P_N \MTI_{n,n+1}\|^{24} | \mathcal{F}_n ) \right)^{1/3}\,,
   \label{eq:rho:l:est:2}
\end{align}
where $C^*= C^*(\varpi)$.
For  $N$ fixed by \eqref{eq:rho:H:n:est:gamma:eta} and for any $\epsilon, \alpha > 0$, and $n \in 2\NN$ consider the set
\begin{align*}
  \Omega_{\epsilon, \alpha, N}^{n} := \left\{  \inf_{ \phi \in  \mathcal{S}_{\alpha, N}} \frac{\langle \MM_{n,n+1} \phi, \phi \rangle}{\| \phi\|^{2}} \geq \epsilon \right\},
\end{align*}
where $\mathcal{S}_{\alpha,N}$ is defined in \eqref{eq:test:fn:set:N:alpha}.
Using  Lemma \ref{lem:good:t:reg:as}, Proposition~\ref{prop:mal:cor:1} with $\eta$ replaced by $3\varpi/2$,
and $\|\MTI_{n,n+1}\| \leq 1$ we have
\begin{align}
 \E \left( \|P_N \MTI_{n,n+1}\|^{24}  | \mathcal{F}_n\right)
 =&  \E \left( \|P_N \MTI_{n,n+1}\|^{24} \indFn{  \Omega_{\epsilon, \alpha, N}^{n} }| \mathcal{F}_n\right)
 +  \E \left( \|P_N \MTI_{n,n+1}\|^{24}\indFn{  (\Omega_{\epsilon, \alpha, N}^{n})^{c} }  | \mathcal{F}_n\right)
 \notag\\
 \leq& \left(\alpha \vee \sqrt{\beta/\epsilon} \right)^{24} + \Prb \left( (\Omega_{\epsilon, \alpha, N}^{n})^{c}  | \mathcal{F}_n\right)
  \notag\\
 \leq& \left(\left(\alpha \vee \sqrt{\beta/\epsilon}\right)^{24} + r(\epsilon) \right)  \exp(3\varpi/2 \cdot \|U(n)\|^2)\,,
 \label{eq:partial:inversion:est}
 \end{align}
which holds for any $\alpha \in (0, 1]$, $\beta > 0$, and any $\epsilon < \epsilon^*(\alpha, \varpi, N)
= \epsilon^*(\alpha, \varpi, \dStep)$, and where the function $r$ is given by \eqref{eq:mal:matrix:inv:bnd:cond}.

Next, we choose $\alpha, \beta, \epsilon > 0$, such that $\epsilon < \epsilon^*(\alpha, \varpi, N)$, and so that
 \begin{align}
  \left(\left(\alpha \vee \sqrt{\beta/\epsilon}\right)^{24} + r(\epsilon) \right)^{1/3} \leq \frac{\dStep}{2^8 C^*},
  \label{eq:alpha:epsilon:beta:juggle}
 \end{align}
 where $C^*$ is the constant from \eqref{eq:rho:l:est:2}.
First choose a sufficiently small $\alpha  = \alpha(\varpi, \dStep) > 0$.  This choice of $\alpha$ fixes
$\epsilon^\ast = \epsilon^\ast(\alpha, \varpi, \dStep) = \epsilon^\ast(\varpi, \dStep) > 0$. We then choose a sufficiently small
$\epsilon < \epsilon^\ast$, $\epsilon = \epsilon (\varpi, \dStep)$  to control the $r(\epsilon)$ term.  Finally, based on this choice of
$\epsilon$, we determine $\beta = \beta(\varpi, \dStep) > 0$.  Thus we infer that
\begin{align*}
  \E (\| \rho_{n +2}^L \|^8 | \mathcal{F}_n) \leq \frac{\dStep}{2^8}   \exp( \varpi  \|U(n)\|^2)  \|\rho_{n}\|^{8}\,.
\end{align*}
Combining this bound with \eqref{eq:rho:H:n:est:gamma:eta} and $\|\rho_{n+2}\|^8 \leq 2^7 (\|\rho_{n+2}^H\|^8 + \|\rho_{n+2}^L\|^8 )$ establishes \eqref{eq:stepwise:cond:decay:est}.

Pursuing the definitions of $\rho$ and $v$ through \eqref{eq:control:def:iterative}, \eqref{eq:two:tm:step:cont:err} the block
adapted structure in \eqref{eq:block:adapted:struct:rho:v} clearly follows by induction and the definitions of the operators
$\JJ_{n, n+1}$, $\AAA_{n,n+1}$, $\AAA_{n,n+1}^*$.  The proof of Lemma~\ref{lem:rho:bound} is thus complete.
\end{proof}

\subsection{Proof of Proposition \ref{prop:grad:est:MS}}
\label{sec:proof:main:prop}

This final section is devoted to the proof of Proposition~\ref{prop:grad:est:MS} following the strategy identified in Sections~\ref{sec:derive:control}--\ref{subsec:control:choice}:
using the machinery of  Malliavin calculus the desired estimate on the Markov semigroup \eqref{eq:main:grad:est} has been translated to the control problem \eqref{eq:decay}, \eqref{eq:bound}, where $\rho$
is a solution of \eqref{eq:JmA:def:control} with $v$ defined by \eqref{eq:control:def:iterative}.
Then, in Section~\ref{sec:approx:control:cost}, we derived a one time step decay estimate on $\rho$ from Lemma \ref{lem:rho:bound} which
we now use as follows.

\begin{proof}[Proof of Proposition \ref{prop:grad:est:MS}]
Fix any $\eta, \gamma_0 > 0$ and any $\rho(0) = \xi \in H$ with $\|\xi\| = 1$. We successively demonstrate \eqref{eq:decay}, \eqref{eq:bound}
for $v$ and $\rho$ determined by \eqref{eq:JmA:def:control}, \eqref{eq:control:def:iterative} and \eqref{eq:two:tm:step:cont:err}.

First we prove \eqref{eq:decay}, using Lemma~\ref{lem:rho:bound}. Define $\Psi: \RR^2 \to \RR$ according to
\begin{align*}
  \Psi(x,y) =
  \begin{cases}
     x/y& \textrm{ for } y \not =0,\\
     0& \textrm{ for } y =0.\\
  \end{cases}
\end{align*}
By \eqref{eq:block:adapted:struct:rho:v}, $\rho_n$ is $\mathcal{F}_n$ measurable, and therefore with \eqref{eq:stepwise:cond:decay:est}
we infer
 \begin{align}
     \E \left( \Psi(\| \rho_{n +2} \|,\|\rho_{n}\|)^8  \exp(- \varpi  \|U(n)\|^2)  | \mathcal{F}_n \right)
     =  \indFn{\|\rho_n\| > 0}  \frac{\E \left( \| \rho_{n +2} \|  | \mathcal{F}_n \right)}{\|\rho_{n}\|^8 \exp( \varpi  \|U(n)\|^2)}   \leq \dStep,
     \label{eq:rho:est:comp}
 \end{align}
  for every $\varpi, \dStep > 0$, and each $n \in 2 \NN$.
 For any $k \geq 0$, let
 \begin{align*}
   X_{k} := \Psi(\| \rho_{2k +2} \|,\|\rho_{2k}\|)^{4} \exp(- \varpi/2 \cdot \|U(2k)\|^2)\,, \qquad Y_{k} := \exp( \varpi/2  \cdot \|U(2k)\|^2) \,.
 \end{align*}
Repeated use of conditional expectation with \eqref{eq:rho:est:comp} implies
 \begin{align*}
 \E \left( \prod_{k =0}^n X_{k}^2 \right)=&\E\left( \E\left(\prod_{k =0}^n X_{k}^2 | \mathcal{F}_{2n} \right)\right)
 	=\E\left( \prod_{k =0}^{n-1} X_{k}^2 \E\left(X_{n}^2 | \mathcal{F}_{2n} \right)\right)
	\notag\\
	\leq& \dStep \E\left( \prod_{k =0}^{n-1} X_{k}^2\right)
	\leq \cdots \leq  \dStep^{n}\,.
  \end{align*}
 On the other hand, noting that for each $n\in \NN$, $\indFn{\|\rho_n\| = 0} \rho(t)= 0$  for any $t \geq n$ and that $\|\rho_0\| = 1$
 we have
 \begin{align*}
   \prod_{k =0}^n X_{k} Y_{k} = \|\rho_{2n+2}\|^{4} \prod_{k = 1}^n \indFn{\|\rho_{2k}\| > 0} = \|\rho_{2n+2}\|^{4}.
 \end{align*}
Consequently, if $\varpi \leq \eta^\ast$ (see Lemma \ref{lem:exp:moments}), H\" older's inequality and \eqref{eq:exp:mom:sums} yield
 \begin{align*}
   \E (\|\rho_{2n+2}\|^{4}) \leq&  \left(\E \prod_{k =0}^n X_{k}^2 \right)^{1/2} \left(\E \prod_{k =0}^n  Y_{k}^2\right)^{1/2}
   \leq \dStep^{n/2}  \left(\E \exp \left(\varpi  \sum_{k =0}^{2n} \|U(k)\|^2 \right)\right)^{1/2}
   \notag\\
   \leq& \dStep^{n/2}  \exp\left(\frac{\varrho \varpi}{2}  \|U_0\|^2 \right)  \exp ( \varkappa n)  \,,
 \end{align*}
 where $\varrho$ and $\varkappa$ are the constants appearing in \eqref{eq:exp:mom:sums}.
 Setting $\dStep := \exp(-4(\varkappa + \gamma_0) )$, $\varpi := \min \{\eta^\ast, 2\eta/\rho \}$, $\beta = \beta (\eta, \gamma_0)$
 as in Lemma~\ref{lem:rho:bound}, we obtain that for each integer $n$
 \begin{align}
\E (\|\rho_{2n}\|^{4}) \leq C\exp\left( \eta  \|U_0\|^2 -2n \gamma_0\right).
    \label{eq:rho:decay:est}
  \end{align}
In particular we infer \eqref{eq:decay} for $t = 2n$.

Next, observe that, for each $n \in \mathbb{N}$,
 \begin{align*}
 \rho(t) =
 \begin{cases}
	\JJ_{2n, t} \rho_{2n} - \AAA_{2n, t}v_{2n, t}& \textrm{ for } t \in [2n, 2n+1),\\
	\JJ_{2n+1, t} \rho_{2n+1}& \textrm{ for } t \in [2n+1, 2n+2].
 \end{cases}
 \end{align*}
By \eqref{eq:control:def:iterative} and \eqref{eq:op:norm:1}--\eqref{eq:op:norm:3}, for any $t \in [2n, 2n + 2]$
\begin{align}
  \|v_{2n, t}\|_{L^2([2n,t]; \RR^{d})} \leq
  \|v_{2n, 2n+1}\|_{L^2([2n,2n+1]; \RR^{d})} \leq \beta^{-1/2} \|J_{2n, 2n+1}\| \|\rho_{2n}\| \,,
  \label{eq:v:norm:obs}
\end{align}
and consequently for any $t \in [2n, 2n+1)$
 \begin{align*}
 \| \rho(t) \| &\leq \|\JJ_{2n, t} \rho_{2n}\| +  \| \AAA_{2n, t}v_{2n, t}\|
 \leq   \|\JJ_{2n, t} \rho_{2n}\| + \|\AAA_{2n, t}\|_{\mathcal{L}(L^2([2n, t]; \RR^{d}), H)}\|v_{2n, 2n+1}\|_{L^2([2n, 2n+1]; \RR^{d})}
\notag \\
&\leq C\beta^{-1/2} \left(1 +  \sup_{s \in [2n,t]} \|\JJ_{s,t}\|^2  \right)\|\rho_{2n}\|,
\end{align*}
and for any $t \in [2n + 1, 2n+2]$
\begin{align*}
\| \rho(t) \| \leq \sup_{s \in [2n+1,t]} \|\JJ_{s,t}\| \|\rho_{2n+1}\|.
\end{align*}
Combining these observations, with \eqref{eq:lin:growth:est} and \eqref{eq:rho:decay:est} the desired estimate \eqref{eq:decay} now follows.\footnote{Note
that the constant we obtain for \eqref{eq:decay} grows as $\beta^{-1/2}$.  This is inconsequential as we obtain a $\beta^{-2}$
dependence in the constant for \eqref{eq:bound} below; see \eqref{eq:cost:bbd:ito:corr}.}

  We turn next to the proof of \eqref{eq:bound}.
  Although $v$ is not adapted, it follows from Lemma~\ref{lem:mal:est} and \eqref{eq:control:def:iterative} that
  $v_{0,N} \in \Mspc^{1,2}(L^2([0,N]; \RR^{d}))$ for any $N > 0$, and we may thus use the generalized
  It\={o} isometry, \eqref{eq:gen:ito:ineq}.  By \eqref{eq:block:adapted:struct:rho:v},
  $v(t)$ is $\mathcal{F}_{\vartheta(t) }$ measurable, where $\vartheta$ is defined in  \eqref{eq:block:adapt:fn}, and consequently  by \eqref{eq:Fs:t:meas:zero:cond},
  $\DM_s v(r) = 0$ if $s > \vartheta(r)$.  Thus
  \begin{align}
    \E \left( \int_0^{2N} v \cdot dW \right)^2 &= \E \int_0^{2N} |v(s)|^2_{\RR^d} ds + \E \int_0^{2N}\int_0^{2N} \chi_{s \leq \vartheta( r )} \chi_{r \leq \vartheta( s) }  \mbox{Tr}( \DM_s v(r) \DM_r v(s)) ds dr.
    \notag\\
    &\leq \sum_{n = 0}^{N-1} \left( \E \int_{2n}^{2n+1} |v_{2n,2n+1}(s)|^2_{\RR^d} ds + \E  \int_{2n}^{2n+1} \int_{2n}^{2n + 1} | \DM_s v_{2n, 2n+1}(r)|^2_{\RR^{d\times d}} ds dr \right).
    \label{eq:gen:ito:iso:block:diag}
\end{align}
For the first term in \eqref{eq:gen:ito:iso:block:diag}, we make use of \eqref{eq:v:norm:obs}
\begin{align}
   \sum_{n = 0}^{N-1}  \E \int_{2n}^{2n+1} |v_{2n,2n+1}(s)|^2_{\RR^d} ds \leq \beta^{-1/2} \sum_{n = 0}^{N-1} \bigl(\E \|\JJ_{2n, 2n+1}\|^4 \cdot \E \| \rho_{2n}\|^4\bigr)^{1/2}
   \leq C\exp( \eta \|U_0\|^2)
  \label{eq:trad:ito:term:bnd}
\end{align}
for a constant $C = C(\eta, \gamma_0)$ independent of $N$.

To bound the second term in \eqref{eq:gen:ito:iso:block:diag}, we compute an explicit expression for $\DM_s v$.
By Lemma~\ref{lem:mal:est}, each of $\JJ_{2n, 2n+1}$, $\AAA_{2n,2n+1}$, $\AAA_{2n,2n+1}^\ast$,
$\MM_{2n, 2n+1} + \beta I$, and $(\MM_{2n,2n+1} + \beta I)^{-1}$ are differentiable in the Malliavin sense and lie in the space
$\Mspc^{1,p}$ for any $p > 1$ (see \eqref{eq:def:D12}).  It thus follows from \eqref{eq:two:tm:step:cont:err} that $\rho_{2n} \in \Mspc^{1,p}$
for any $p > 1$ and any $n$.
Moreover, recalling that by \eqref{eq:block:adapted:struct:rho:v},  $\rho_{2n}$ is $\mathcal{F}_{2n}$ measurable, \eqref{eq:Fs:t:meas:zero:cond} implies that
$\MD_s \rho_{2n} = 0$ for any $s \geq 2n$.
Then by the Malliavin product rule (see e.g. \cite[Lemma 3.6]{PronkVeraar2013}) we compute
\begin{align}
   \DM_s^j v_{2n, 2n+1} &=  \AAA_{2n,2n+1}^* (\MM_{2n,2n+1}  + I \beta)^{-1}  (\DM_s^j \JJ_{2n,2n+1}) \rho_{2n}
   		+  \AAA_{2n,2n+1}^* (\DM_s^j (\MM_{2n,2n+1}  + I \beta)^{-1} ) \JJ_{2n,2n+1}\rho_{2n}
		\notag\\
		&\qquad + (\DM^j_s \AAA_{2n,2n+1}^*) (\MM_{2n,2n+1}  + I \beta)^{-1}  \JJ_{2n,2n+1}\rho_{2n}
		\label{eq:MD:stoc:integrand}
\end{align}
for any $j \in  \{1, \ldots, d\}$ and $s \in [2n, 2n+1]$.
Moreover, after differentiating the identity $(\MM_{2n,2n+1}  + I \beta)^{-1}(\MM_{2n,2n+1}  + I \beta) = I$ and recalling that
$\MM_{2n,2n+1} = \AAA_{2n,2n+1} \AAA_{2n,2n+1}^*$ we obtain
\begin{multline}
\DM_s^j (\MM_{2n,2n+1}  + I \beta)^{-1} = \\
 (\MM_{2n,2n+1}  + I \beta)^{-1}
((\DM_s^j \AAA_{2n,2n+1}) \AAA_{2n,2n+1}^* + \AAA_{2n,2n+1}(\DM_s^j \AAA_{2n,2n+1}^*))
(\MM_{2n,2n+1}  + I \beta)^{-1} \,.
\label{eq:mal:der:matr}
\end{multline}
By \eqref{eq:MD:stoc:integrand}, \eqref{eq:mal:der:matr}, and the bounds \eqref{eq:op:norm:1}--\eqref{eq:op:norm:3}, one has
for each $s \in [2n, 2n+1]$
\begin{align}
\| \DM_s^j v_{2n, 2n+1} &\|_{L^2([2n, 2n+1];\RR^{d})} \notag\\
    \leq& \beta^{-1/2} \|\DM_s^j \JJ_{2n,2n+1}\| \|\rho_{2n}\|
	+ \beta^{-1} \|\DM_s^j \AAA_{2n,2n+1}\|_{\mathcal{L}(L^2([2n, 2n+1],\RR^d), H)} \|\JJ_{2n,2n+1}\|\|\rho_{2n}\|
	\notag\\
&+ 2 \beta^{-1} \|\DM_s^j \AAA_{2n,2n+1}^\ast\|_{\mathcal{L}(H,L^2([2n, 2n+1],\RR^d))} \|\JJ_{2n,2n+1}\|\|\rho_{2n}\| \,.
\label{eq:mid:step}
\end{align}
Finally, we use \eqref{eq:mid:step}, \eqref{eq:mal:J:est}--\eqref{eq:mal:A:star:est}, \eqref{eq:lin:growth:est}, and \eqref{eq:rho:decay:est}
to conclude
\begin{align}
\E \sum_{n = 0}^{N-1}  \int_{2n}^{2n+1} \int_{2n}^{2n + 1} | \DM_s v_{n, n+2}(r)|^2_{\RR^{d\times d}} ds dr \notag
&= \E \sum_{n = 0}^{N-1} \sum_{j = 0}^{d} \int_{2n}^{2n + 1} \|\DM_s^j v_{2n, 2n+1}\|_{L^2([2n, 2n+1];\RR^{d})} ds \\
&\leq C \beta^{-2} \exp(\eta/2 \|U_0\|^2) \sum_{n = 0}^{\infty} [\E \|\rho_{2n}\|^4]^{1/2} \notag\\
&\leq C \beta^{-2} \exp(\eta \|U_0\|^2) \,,
\label{eq:cost:bbd:ito:corr}
\end{align}
where $C = C(\eta, \gamma_0)$.  Combining \eqref{eq:cost:bbd:ito:corr}, \eqref{eq:trad:ito:term:bnd}
with \eqref{eq:gen:ito:iso:block:diag} we now infer \eqref{eq:bound}, completing the proof of Proposition~\ref{prop:grad:est:MS}.
\end{proof}

\section{Spectral Bounds for the Malliavin Covariance Matrix}
\label{sec:Mal:spec:bnds}

In this section we present the main technical result of this work,
Theorem~\ref{thm:mind:blowing:conclusion}, which yields a
probabilistic spectral bound on the Malliavin matrix $\MM_{0,T}$ (see \eqref{eq:Mal:def}).
Recall that in Section~\ref{sec:math:setting} we established the uniqueness of the invariant measure associated to \eqref{eq:B1:Vort}--\eqref{eq:B3:Vort}
assuming a gradient estimate on the Markov semigroup, \eqref{eq:main:grad:est}.  Then, in Section \ref{sec:grad:control}, we established this gradient estimate \eqref{eq:main:grad:est} modulo
Proposition \ref{prop:mal:cor:1}, which is a corollary to Theorem~\ref{thm:mind:blowing:conclusion}.  Hence,
we have reduced the proof of the uniqueness in Theorem \ref{thm:mainresult} to the proof of Theorem~\ref{thm:mind:blowing:conclusion}.

\begin{Thm}
\label{thm:mind:blowing:conclusion}
Let $U_0 \in H$ and define $\MM_{0,T}$ according to \eqref{eq:Mal:def}, relative to $U(\cdot) = U(\cdot, U_0)$ solving \eqref{eq:BE:abs}.
Fix any $\alpha \in (0, 1]$, $N \geq 1$, and $\eta > 0$. Then, there exists $\epsilon^\ast := \epsilon^\ast(T, \alpha, \eta, N) > 0$
such that for any $0 <\epsilon < \epsilon^\ast$, there is a measurable set $\Omega_\epsilon = \Omega_\epsilon(\alpha, N) \subset \Omega$
satisfying
\begin{align*}
  \Prb(\Omega_\epsilon^c) \leq r( \epsilon) \exp(\eta \|U_0\|^2)\,,
\end{align*}
where $r = r(T, \alpha, \eta,N)$ is a nonnegative decreasing function such that $r( \epsilon) \to 0$ as $\epsilon \to 0$. On this set $\Omega_\epsilon$
\begin{align*}
 \inf_{\phi \in \mathcal{S}_{\alpha, N}}  \frac{\langle \MM_{0,T} \phi, \phi \rangle }{\|\phi\|^2} \geq \epsilon \,,
\end{align*}
where $\mathcal{S}_{\alpha, N} = \{ \phi \in H: \|P_N \phi \|^2 >  \alpha \| \phi\|^2 \}$.
In particular,
\begin{align}
\Prb \left( \inf_{ \phi \in \mathcal{S}_{\alpha, N}} \frac{\langle \MM_{0,T} \phi, \phi \rangle}{\| \phi \|^2 } \geq \epsilon \right)
   \geq 1 - r(\epsilon) \exp(\eta \|U_{0}\|^2).
  \label{eq:mal:matrix:inv:bnd}
\end{align}
\end{Thm}

\begin{Rmk}
\label{rmk:mind:blowing:remarks:for:a:mind:blowing:conclusion}
\mbox{}
\begin{itemize}
\item[(i)] An explicit form for $r(\epsilon)$ is given in \eqref{eq:small:sets:bal:1} below. While the decay rate in $r(\epsilon)$ as $\epsilon \to 0$ is much slower than in previous works, as observed in Section \ref{sec:grad:control},
it is sufficient for the proof of Proposition \ref{prop:grad:est:MS}.
\item[(ii)] Using the Markov property in the general form found in e.g. \cite[Theorem 9.12]{ZabczykDaPrato1992}, Theorem~\ref{thm:mind:blowing:conclusion} immediately implies Proposition~\ref{prop:mal:cor:1}.
\item[(iii)] Note that it is not enough to replace $S_{\alpha, N}$ by $H_N \subset S_{\alpha, N}$ in Theorem \ref{thm:mind:blowing:conclusion}, as we made use of
 \eqref{eq:mal:matrix:inv:bnd} for small $\alpha > 0$ in the
proof of Lemma~\ref{lem:rho:bound}; see e.g. \eqref{eq:partial:inversion:est}, \eqref{eq:alpha:epsilon:beta:juggle} above.
\end{itemize}
\end{Rmk}

Broadly speaking, the proof of Theorem \ref{thm:mind:blowing:conclusion} involves an `iterative proof by contradiction' following a strategy apparent even in e.g. \cite{Norris1986}.\footnote{In \cite{Norris1986} as in
\cite{Malliavin1978} the goal is to provide a probabilistic proof of H\"ormander's hypoellipticity theorem \cite{Hormander1967}.  These works link
H\"ormander's bracket
condition (associated to a hypo-elliptic evolution equation) to the invertibility of the Malliavin covariance matrix.
In addition to \cite{Norris1986} we refer the interested reader to e.g. \cite{Nualart2006, Hairer2011} for further details on the probabilistic approach to hypoellipticity.}
We show that, on sets of large probability,
if the Malliavin matrix $\MM_{0,T}$ has a small eigenvalue, then a certain quadratic form $\mathcal{Q}$ associated to \eqref{eq:BE:abs}
is small when evaluated at the corresponding
eigenfunction.    We then show that $\mathcal{Q}$
has a suitable lower bound on $S_{\alpha,N}$.
This lower bound may be seen as an
infinite-dimensional analogue of the H\"ormander bracket condition \cite{Hormander1967}.   By combining these upper and lower
bounds we conclude that, with large probability, $S_{\alpha,N}$ cannot contain eigenfunctions of $\MM_{0,T}$ corresponding to small eigenvalues.

We refer to proof of Theorem~\ref{thm:mind:blowing:conclusion} as `iterative' because the bounds on $\mathcal{Q}$ are obtained by an inductive argument which
yields a chain of quantitative bounds on certain functionals associated to H\" ormander type Lie brackets.
Although, we  make significant use of a methodology recently developed in \cite{MattinglyPardoux1, BakhtinMattingly2007, HairerMattingly2011}
to carry out this process in infinite dimensions,
new and interesting difficulties emerge in our situation which reflect the interaction between the nonlinear structure of
\eqref{eq:B1:Vort}--\eqref{eq:B3:Vort} and our choice of stochastic forcing (see Remark \ref{rem:forcing} above).  Firstly, with the
stochastic forcing in the temperature equation only, it is non-trivial to determine a sequence of suitable (H\"ormander type) Lie brackets associated to \eqref{eq:BE:abs}; a completely different
methodology must be developed for \eqref{eq:BE:abs} compared to the one used for the stochastic Navier-Stokes equation in \cite{EMattingly2001, Romito2004, HairerMattingly06}.
Secondly, the vector fields we obtain are $U$ dependent. This situation forces us to use an infinite-dimensional analogue
of the H\"ormander bracket condition, which is weaker than the condition appearing in previous works.

The rest of this section is devoted to proof of  Theorem~\ref{thm:mind:blowing:conclusion}
based on the lower and upper bounds on forms $\mathcal{Q}$
associated to H\"ormander type brackets involving \eqref{eq:BE:abs}.  The lower and upper bounds are given below as
Propositions~\ref{prop:large:conclusion} and \ref{prop:small:conclusion}, respectively.  The
detailed computations of the H\"ormander brackets are postponed for Sections~\ref{sec:Hormander:brak}.

\subsection{Quadratic Forms; Upper and Lower Bounds}

Before precisely stating the lower and upper bounds,
we briefly recall the origin of the quadratic forms found in these
propositions. As explained in the introduction, the `admissible H\"ormander
brackets' are elements in the sets
\begin{align}
   \mathcal{V}_m := \mbox{span}\left\{ [E, F], [E, \sigma_k^j], E: k \in \mathcal{Z}, j = \{0,1\}, E \in \mathcal{V}_{m-1}  \right\}
   \label{eq:Mth:iter:admissible:brak}
\end{align}
starting from $\mathcal{V}_{0} =  \mbox{span}\{\sigma_k^j: k \in \mathcal{Z}, j = \{0,1\}\}$, where we recall that the Lie brackets are given by $[E_1 , E_2 ] := \nabla E_2 \cdot E_1 - \nabla E_1 \cdot E_2 $.
We show in Section~\ref{sec:Hormander:brak},
that for each $N < \tilde{N}$ there exists $M = M(\tilde{N})$ such that the set
\begin{align}
  \mathfrak{B}_{N, \tilde{N}}(U) := \{ \sigma_{j}^{m}, \psi_{j}^{m} + J^{\tilde{N}}_{j,m}(U): m \in\{0,1\},  j \in \ZZ^{2}_{+}, |j| \leq N \}
  \label{eq:approx:basis:tails}
\end{align}
is contained in $\mathcal{V}_M$.  Here, recall that $\sigma_{j}^{m}$ and $\psi_{j}^{m}$ are basis elements for $H$ defined in \eqref{eq:def:basis:sig},
\eqref{eq:def:basis:psi} above.  The elements $J^{\tilde{N}}_{j,m}(U)$ are $U$ dependent `error' terms, which reside in $H^{\tilde{N}} := Q_{\tilde{N}} H$ (see \eqref{eq:H:N:proj:ops}) and satisfy the bound \eqref{eq:junk:est};
 the explicit
form for these terms is given in \eqref{eq:junk:express},  \eqref{eq:junk:controlled} below.

The upshot is that for any (finite) $M$ we are only able to identify $U$-dependent subsets of $\mathcal{V}_M$. Hence, we need to introduce a new form of the H\"ormander
described above in \eqref{eq:our:new:Hormander} (in more general terms) satisfied by
\begin{align}
 \langle \mathcal{Q}_{N, \tilde{N}}(U) \phi, \phi \rangle := \sum_{\aproxB \in \mathfrak{B}_{N, \tilde{N}}(U)}& |\langle \phi, \aproxB(U) \rangle|^2
  \label{eq:our:quad:form:Horm}
\end{align}
for any $\tilde{N} > N$.  We are ready to state the lower bound on $\mathcal{Q}_{N, \tilde{N}}$.

\begin{Prop}
\label{prop:large:conclusion}
Fix any any integers $N \leq \tilde{N}$ and define $\mathfrak{B}_{N, \tilde{N}}$ by \eqref{eq:approx:basis:tails}.
Then, for any $U \in H^2$ and any $\alpha \in (0,1]$ it holds that
\begin{align}
  \langle \mathcal{Q}_{N, \tilde{N}}(U) \phi, \phi \rangle
   \geq   \left( \frac{\alpha}{2} - C^\ast \frac{N^{8}}{\tilde{N}} (1 +
   \|U\|_{H^2}^2) \right)  \| \phi \|^2
   \label{eq:upper:bound:quad:form}
\end{align}
for every $\phi \in \mathcal{S}_{\alpha,N}= \{ \phi \in H: \| P_N \phi\|^2 > \alpha \|\phi\|^2\}$, where $C^\ast$
is a universal constant (see Remark \ref{rmk:constants:conven}).
\end{Prop}

\begin{proof} Since $\{ \sigma_{j}^{m}, \psi_{j}^{m}\}_{m \in\{0,1\},  |j| \leq N}$ form an orthonormal basis of
$H_N = P_N H$, we obtain for any $\phi \in \mathcal{S}_{\alpha, N}$
\begin{align*}
   \sum_{\aproxB \in \mathfrak{B}_{N, \tilde{N}}(U)} |\langle  \phi, \aproxB(U) \rangle|^2
   &= \sum_{\substack{|j| \leq N \\ m \in\{0,1\}}} |\langle \phi, \sigma_j^m\rangle|^2 + |\langle \phi, \psi_j^m + J_{j, m}^{\tilde{N}}(U)\rangle|^2
      \notag\\
   &=   \|P_N \phi \|^2 +   \sum_{\substack{|j| \leq N \\ m \in\{0,1\}}}
   \left(2 \langle  \phi, \psi_j^m\rangle \langle  \phi, J_{j, m}^{\tilde{N}}(U)\rangle + \langle   \phi, J_{j, m}^{\tilde{N}}(U)\rangle^2\right)
   \notag\\
   &\geq   \frac{1}{2}\|P_N  \phi \|^2 -   \sum_{\substack{|j| \leq N \\ m \in\{0,1\}}} \langle  \phi, J^{\tilde{N}}_{j,m}(U)\rangle^2
   \geq   \frac{\alpha}{2}\| \phi \|^2 -    \| \phi\|^2  \sum_{\substack{|j| \leq N \\ m \in\{0,1\}}} \|J^{\tilde{N}}_{j,m}(U)\|^2 \,.
\end{align*}
Now observe that  \eqref{eq:junk:est} with $s = 1$ yields
\begin{align*}
   \| J^{\tilde{N}}_{j,m}(U) \|^2 \leq C \frac{N^{8}}{\tilde{N}}( 1 + \|U\|^2_{H^{2}})
\end{align*}
and the desired bound, \eqref{eq:upper:bound:quad:form} follows.
\end{proof}

\noindent Next we will state the `upper bound' on $\mathcal{Q}_{N, \tilde{N}}$, whose proof is long and technical and is postponed to Section~\ref{sec:Mal:spec:bnds}.  It links the chain of Lie brackets presented in Section~\ref{sec:Hormander:brak} (and summarized in Figure~\ref{fig:brak:1}) to
quantitative estimates on $\mathcal{Q}_{N, \tilde{N}}$.

\begin{Prop}
\label{prop:small:conclusion}
Fix $T > 0$. There are positive constants $q_i = q_i(T) > 0$, $i = 0, \ldots 5$ such that the following holds. Fix any $\eta > 0$, any integer $\tilde{N} > 0$ and define
\begin{align}
   \mathcal{E}(\tilde{N}) := \min \left\{ q_0, \left(\frac{q_1}{\tilde{N}}\right)^{q_2^{\tilde{N}}} \right\} \,.
\label{eq:eps:2}
\end{align}
Then for every $\epsilon \in (0,  \mathcal{E}(\tilde{N}))$
there is a set $\Omega_{\epsilon, \tilde{N}}^{*}$  and a constant $C = C(\eta, T)$ such that
\begin{align}
  \Prb((\Omega_{\epsilon,\tilde{N}}^*)^c) \leq C \tilde{N}^{q_3} \exp(\eta \|U_0\|^2)\epsilon^{q_4^{\tilde{N}}}
  \label{eq:good:conclusion:off:bad:set}
\end{align}
and on $\Omega_{\epsilon, \tilde{N}}^{*}$ one has (cf. \eqref{eq:our:quad:form:Horm})
\begin{align}
    \langle \MM_{0,T} \phi, \phi \rangle \leq \epsilon \|\phi\|^2 \quad
  \Rightarrow \quad
     \langle \mathcal{Q}_{N, \tilde{N}}(U(T)) \phi, \phi \rangle  \leq \epsilon^{q_5^{\tilde{N}}} \|\phi\|^2\,.
    \label{eq:good:implication:upper:bnd}
\end{align}
which is valid for $N < \tilde{N}$ and any $\phi \in H$.
\end{Prop}
\begin{Rmk}
  We may suppose without loss of generality that
  \begin{align*}
  	q_2 >1, \quad q_3 > 1, \quad \textrm{ and } \quad q_4 < 1.
  \end{align*}
\end{Rmk}

\subsection{Proof of Theorem~\ref{thm:mind:blowing:conclusion}}
Before turning to precise details, let us outline the proof of Theorem~\ref{thm:mind:blowing:conclusion}
using Propositions~\ref{prop:large:conclusion},~\ref{prop:small:conclusion}.    Observe that, on $\Omega_{\epsilon, \tilde{N}}^{*}$ given in Proposition~\ref{prop:small:conclusion},
we can combine \eqref{eq:upper:bound:quad:form}, \eqref{eq:good:implication:upper:bnd} to infer
\begin{align}
  \langle \MM_{0,T} \phi, \phi \rangle \leq \epsilon \|\phi\|^2 \quad
  \Rightarrow \quad \epsilon^{q_5^{\tilde{N}}} \|\phi\|^2 \geq
  \left( \frac{\alpha}{2} - C^\ast \frac{N^{8}}{\tilde{N}} (1 +
   \|U(T)\|_{H^2}^2) \right)  \| \phi \|^2,
\label{eq:thm41:1}
\end{align}
for any $\phi \in \mathcal{S}_{\alpha, N}$ and for any $\epsilon < \mathcal{E}(\tilde{N})$.
If $\|U(T)\|_{H^2}^2$ was bounded above by a deterministic constant, then we would
prove Theorem~\ref{thm:mind:blowing:conclusion} by
taking $\Omega_{\epsilon}=\Omega_{\epsilon,\tilde{N}}^{\star}$ with $\tilde{N}=\tilde{N}(\alpha, N)$ sufficiently large such that
the conclusion of \eqref{eq:thm41:1} produces a contradiction for $\epsilon$ sufficiently small.  Of course, since the $U$ appearing in
\eqref{eq:thm41:1}
is a solution to \eqref{eq:BE:abs}, such an upper bound is not be to expected.  Instead, we make
a modification of  $\Omega_{\epsilon, \tilde{N}}^{*}$ by intersecting with sets $\Omega_{\epsilon, U, h}$ that
quantify $\|U(T)\|_{H^2}$ in terms of a function $h=h(\epsilon)$, which grows unboundedly as $\epsilon \to 0$.  Returning to \eqref{eq:thm41:1},
$\tilde{N}$ now depends on $\epsilon$ through $h$.   As such, care is needed in the choice of $h$ to avoid a possibly circular argument.

\begin{proof}[Proof of Theorem~\ref{thm:mind:blowing:conclusion}]
For each $\epsilon \in (0, 1/e)$, let
\begin{align}
  h(\epsilon) := \log (\log (\log(\epsilon^{-1}))), \qquad \tilde{N}(\epsilon) :=  \left\lceil \frac{4C^\ast N^{8} h (\epsilon)}{\alpha} \right\rceil \,,
  \label{eq:h:tildeN:eps:def}
\end{align}
where $\lceil x \rceil$ denotes the smallest integer larger or equal to $x$ and $C^\ast$ is the constant from
\eqref{eq:upper:bound:quad:form}.
Then
\begin{align}
\frac{\alpha}{2} -  C^\ast \frac{N^{8}}{\tilde{N}(\epsilon)}  h (\epsilon)  \geq \frac{\alpha}{4}\,.
\label{eq:alp:low}
\end{align}
Observe that, with these definitions, there exists $\epsilon_1^\ast > 0$ such that $\epsilon < \mathcal{E}(\tilde{N}(\epsilon))$
whenever $\epsilon < \epsilon_1^\ast$, where $\mathcal{E}$ is defined by \eqref{eq:eps:2}.\footnote{Indeed, one can begin by supposing $\epsilon_1^\ast< \min\{q_0,1/e\}$
and observe that
\begin{align*}
 \limsup_{\epsilon \to 0^+} \epsilon
 \left(\frac{q_1}{\tilde{N}(\epsilon)}\right)^{-q_2^{\tilde{N}(\epsilon)}}
 &\leq
  \limsup_{\epsilon \to 0^+} \epsilon
  h(\epsilon)^{2q_2^{\tilde{N}(\epsilon)}}
  =  \limsup_{\epsilon \to 0^+} \epsilon \exp \left( 2q_2^{\tilde{N}(\epsilon)} \log( h(\epsilon))\right)\\
  &\leq  \limsup_{\epsilon \to 0^+} \epsilon \exp \left( \exp(  h(\epsilon)^2) \right) =
  \limsup_{h \to \infty } \exp(- \exp (\exp (h))) \exp \left( \exp(  h^2) \right) = 0.
\end{align*}
} For $\epsilon <  \epsilon_1^\ast$, $\Omega_{\epsilon, \tilde{N}(\epsilon)}^*$ (see Proposition \ref{prop:small:conclusion}) is well defined and we set
\begin{align*}
\Omega_{\epsilon} := \hat{\Omega}_{\epsilon, U,h} \cap \Omega_{\epsilon, \tilde{N}(\epsilon)}^*,
\end{align*}
where
\begin{align*}
\hat{\Omega}_{\epsilon, U,h} := \{1 + \|U(T)\|^2_{H^2} \leq h(\epsilon)\} \,.
\end{align*}
We now show that on the sets $\Omega_{\epsilon}$ we obtain the desired conclusion \eqref{eq:mal:matrix:inv:bnd}.
First observe that by \eqref{eq:good:conclusion:off:bad:set}, the Markov inequality, and \eqref{eq:smoothing:est}, there is $C = C(\eta, T)$ such that
\begin{align}
   \Prb(\Omega_{\epsilon}^c) \leq \Prb( \hat{\Omega}_{\epsilon, U,h}^c ) +  \Prb( (\Omega_{\epsilon, \tilde{N}(\epsilon)}^\ast)^c)
      \leq C \left(\frac{1}{h(\epsilon)} + (\tilde{N}(\epsilon))^{q_3}\epsilon^{q_4^{\tilde{N}}}  \right)\exp( \eta \|U_0\|^2)
      =: r(\epsilon)\exp( \eta \|U_0\|^2)
       \label{eq:small:sets:bal:1}
\end{align}
whenever $\epsilon < \epsilon_1^\ast$.
The later quantity $r(\epsilon)$ decays to zero as $\epsilon \to 0^+$; since
$h(\epsilon) \to \infty$ as $\epsilon \to 0^+$
\begin{align}
  \limsup_{\epsilon \to 0^+} \tilde{N}(\epsilon)^{q_3}
  \epsilon^{q_4^{\tilde{N}(\epsilon)}}
    &\leq
  \limsup_{\epsilon \to 0^+} h(\epsilon)^{2q_3}
  \epsilon^{\exp ( -h(\epsilon)^{2})} \notag\\
  &\leq
  \exp\left( \limsup_{\epsilon \to 0^+}( 2 q_3 \log h(\epsilon) +
  \exp (  -h(\epsilon)^{2}) \log \epsilon ) \right)
  \notag \\
  &=
  \exp\left( \limsup_{h \to \infty} \left( 2 q_3 \log( h) -
  \exp (  -h^{2}) \exp( \exp(h)) \right) \right) =
   0.
  \label{eq:key:est}
\end{align}
On the other hand, on $\Omega_\epsilon$,
Propositions~\ref{prop:large:conclusion}, \ref{prop:small:conclusion}, and \eqref{eq:alp:low} yield
\begin{align*}
\langle \MM_{0,T}& \phi, \phi \rangle \leq \epsilon \|\phi\|^2 \notag\\
  &\Rightarrow \quad
  \epsilon^{q_5^{\tilde{N}(\epsilon)}} \|\phi\|^2
  \geq  \left( \frac{\alpha}{2} - C^\ast \frac{N^{8}}{\tilde{N}} (1 +
   \|U(T)\|_{H^2}^2) \right) \|\phi\|^2 \geq
  \left(\frac{\alpha}{2} -  C^\ast \frac{N^{8}}{\tilde{N}}  h (\epsilon) \right)\|\phi\|^2 \geq
  \frac{\alpha}{4} \|\phi\|^2,
 \end{align*}
 for each $\epsilon < \epsilon_1^*$ and any $\phi \in \mathcal{S}_{\alpha, N}$.
 We infer that if $\epsilon < \epsilon_1^*$, then on $\Omega_\epsilon$,
 \begin{align*}
   \inf_{\phi \in \mathcal{S}_{\alpha,N}} \frac{\langle \MM_{0,T} \phi, \phi \rangle}{\|\phi\|^2} > \epsilon
   \quad \textrm{ whenever }  \epsilon^{q_5^{\tilde{N}(\epsilon)}} < \frac{\alpha}{4}.
 \end{align*}
 From the definition of $\tilde{N}(\epsilon)$ in \eqref{eq:h:tildeN:eps:def}, we have $\epsilon^{q_5^{\tilde{N}(\epsilon)}} \to 0$ as $\epsilon \to 0^+$
 (e.g. see \eqref{eq:key:est} above), and the proof of Theorem~\ref{thm:mind:blowing:conclusion} is complete.
\end{proof}

\section{Lie Bracket Computations}
\label{sec:Hormander:brak}

In this section we present, in functional setting, the chain of Lie brackets that
approximately generate  spanning sets for successively larger finite dimensional subspaces
$H_N$ of the phase space $H$.  More precisely, we show that the approximate basis
$\mathfrak{B}_{N, \tilde{N}}(U)$ lie in admissible sets $\mathcal{V}_M$ (see \eqref{eq:Mth:iter:admissible:brak}, \eqref{eq:approx:basis:tails})
for sufficiently large $M = M(\tilde{N})$, which allows us to define the quadratic forms $ \mathcal{Q}_{N, \tilde{N}}(U)$ leading to the upper
and lower bounds, Propositions~\ref{prop:large:conclusion}, ~\ref{prop:small:conclusion}  in Section~\ref{sec:Mal:spec:bnds}.

As we described above, these computations are motivated by the celebrated H\" ormander condition
for the Kolmogorov-Fokker-Planck equations associated to \eqref{eq:BE:abs}, cf. \eqref{eq:fokk:er:plank}.  Our situation is notable in comparison to previous
works in the infinite dimensional setting, \cite{EMattingly2001, Romito2004, HairerMattingly06, HairerMattingly2011} as,
to the best of our knowledge, we are the first to analyze a system, where the chain of vector fields and the associated quadratic forms
$\mathcal{Q}_{N, \tilde{N}} = \mathcal{Q}_{N, \tilde{N}}(U)$ depend on $U$, and are therefore random.

   For contrast consider the 2D stochastic Navier-Stokes equation written in the vorticity formulation
\begin{align}
  d \omega + F_{NS}(\omega) = \sigma dW = \sum_{k,l} \alpha_k \psi_k^l dW^{k,l},
  \quad \textrm{ with } F_{NS}(\omega) := -\nu \Delta \omega + (K\ast \omega)\cdot \nabla\omega,
  \label{eq:SNSE:ex}
\end{align}
where as in \eqref{eq:def:basis:psi} we set $\psi_k^0 := \cos (k \cdot x)$ and $\psi_k^1 := \sin (k \cdot x)$
for $k \in \ZZ^2_{+}$.
Let $\mathcal{V}_M^{NS}$, $M \geq 0$ be the sequences of admissible vector fields corresponding to \eqref{eq:SNSE:ex},
defined analogously to \eqref{eq:Mth:iter:admissible:brak}.
The following bracket structure for \eqref{eq:SNSE:ex} was observed in \cite{EMattingly2001}.
Suppose that $\psi_k^m, \psi_{k'}^{m'} \in \mathcal{V}_{M}^{NS}$ for some $M \geq 0$.
Then, since $(K\ast \omega)\cdot \nabla\omega$ is the only nonlinear (quadratic) term in \eqref{eq:SNSE:ex}, we obtain
\begin{align}
[[F_{NS}(\omega), \psi_k^m], \psi_{k'}^{m'}] = (K\ast \psi_k^m) \cdot \nabla \psi_{k'}^{m'}
+ (K\ast \psi_{k'}^{m'}) \cdot \nabla \psi_k^m \in \mathcal{V}_{M + 2},
\label{eq:lie:nse}
\end{align}
where the Lie brackets $[\cdot, \cdot]$ are defined as in \eqref{eq:Lie:brak:abs}.
Using elementary algebra, we obtain that $\psi_{k + k'}^{m + m'} \in \mathcal{V}_{M + 2}$ if $k$ and  $k'$ are not parallel and $|k| \not = |k'|$.
Under appropriate algebraic assumptions on the set of directly forced
modes, one therefore obtains that for any $N > 0$, $H_N \subset \mathcal{V}_M$ for large enough $M = M(N)$.
As we already noted  in the introduction, this strategy of repeated brackets with constant vector field to generate
exactly $H_N$ has been used in all of the previously known examples; see \cite{HairerMattingly2011}.

Our situation is completely different. Since the random perturbation appears only in the temperature equation in \eqref{eq:B1:Vort}--\eqref{eq:B3:Vort},
we immediately see, recalling the notation \eqref{eq:def:basis:sig}, that for any $k, k' \in \ZZ^2_+$,
$m, m' \in \{0, 1\}$
\begin{equation}\label{eq:lie:bou}
[[F(U), \sigma_k^m], \sigma_{k'}^{m'}] =  B(\sigma_k^m, \sigma_{k'}^{m'})
+ B(\sigma_{k'}^{m'}, \sigma_{k}^{m}) = 0 \,,
\end{equation}
and therefore no new modes are generated.
The observation in \eqref{eq:lie:bou} suggests that we need to make more carefully use of the interaction between the buoyancy term $G$
and the advective structure in $B$.

The strategy which we devised to generate suitable directions is summarized in Figure~\ref{fig:brak:1} below.
Strikingly, a  bracket $[[[F(U), \sigma_k^l], F(U)], \sigma_{k'}^{l'}] = c_1B(\psi_k^{l+1}, \sigma_{k'}^{l'}) +
c_2 B(\psi_{k'}^{l'+1}, \sigma_{k}^{l})$ produces the desirable cancellation, where
$c_1, c_2$ are suitable constant.\footnote{As we already observed in Remark~\ref{rem:forcing}
it is precisely at this point that we are able to avoid the condition that the forcing in \eqref{eq:B1}--\eqref{eq:B2} contain wavenumbers of
different magnitudes as is required for the 2D stochastic Navier-Stokes equation in \cite{HairerMattingly06}.}
While the computations leading to this cancellation are involved this `miracle' is perhaps
anticipated by the advective structure of $B$.

Having  devised a strategy to generate $\sigma$-modes,
we also have to generate the elements $\psi_k^l$;
that is, suitable directions in the $\omega$ variable.
In this case no additional cancellation is evident as we found for the
$\sigma$-modes.  Instead, we produce
functions of the form $\psi^m_l + J_{m,l}(U)$, where $J_{m,l}(U)$ is an `error term' with a complicated dependence on $U$.
Again as an artifact of the advective structure in $B$, these $J_{m,l}(U)$ are concentrated
in the $\theta$ component only, and we can push these errors entirely into large wave-numbers
by generating additional directions in $\sigma$.  This in turn allows us to make use of the generalized
Poincar\'e inequality to obtain bounds leading to our form of the H\"ormander condition \eqref{eq:our:new:Hormander}.

\begin{figure}[tb]
 \centering
 \vspace{-.5in}
 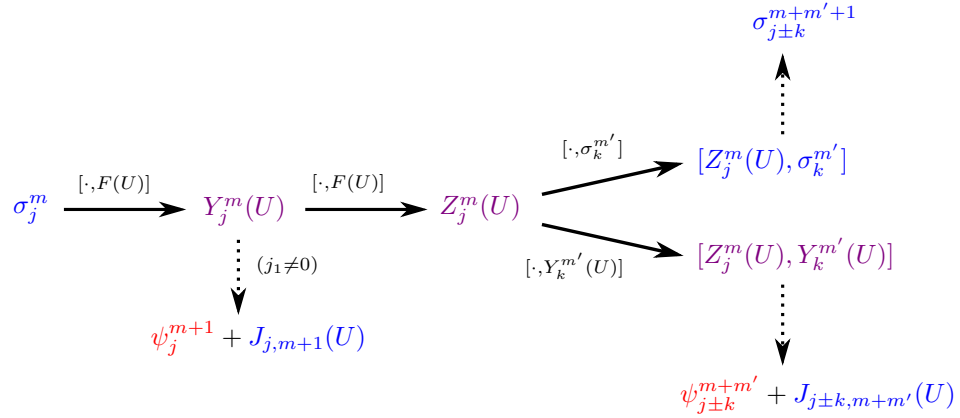
 \vspace{-.7in}
 \caption{The strategy for (approximately) generating $H_N$ from Lie brackets.
 \rd{Red} and \bl{blue} text indicate that the corresponding functions reside exclusively in the $\omega$
 and $\theta$ components, respectively.  \pp{Purple} text indicates that the element has non-trivial content
 in both components.  Solid arrows mean that the new function is generated from a Lie bracket, with the type
 of bracket indicated above the arrow.  Dotted arrows signify that the new element is
 generated as a linear combination of elements from the previous position.}
 \label{fig:brak:1}
 \end{figure}

 \subsection{Detailed Bracket Computations}

Let us provide more details. In what follows $U$ is arbitrary sufficiently smooth
function.\footnote{At this stage in the argument, $U$ is not necessarily a solution of \eqref{eq:BE:abs}.
Of course we will use that $U$ solves \eqref{eq:BE:abs} later on in Section~\ref{sec:Braket:Est:Mal:Mat}. This causes
no problems since $U$ is smooth; cf. Proposition~\ref{prop:wellposed}.}
In the forthcoming computations we make use
of the following simple observations.  Recalling that the Lie bracket between $C^1$ vector fields $E_1, E_2: H \to H$
is given by $[E_1, E_2](U) := \nabla E_2(U) E_1(U) - \nabla E_1(U)  E_2(U)$ we have the antisymmetry and
Jacobi identities
\begin{align}
  [E_1, E_2] = -[E_2, E_1], \quad [[E_1, E_2], E_3] + [[E_2, E_3 ],E_1 ] + [[E_3, E_1], E_2] = 0,
  \label{eq:jacobi:and:his:many:colored:identity}
\end{align}
valid for any $E_1, E_2, E_3$.
 From \eqref{eq:convection:term} and \eqref{eq:boy:lin:term}
we have that for any $\tilde{U} \in H^1,$
\begin{align}
  B(U, \tilde{U}) &= 0 \quad \textrm{if } U= (0, \de)\,,
   \label{eq:B:extra:canel} \\
  GU &= 0 \quad \textrm{if  } U= (\om, 0),
  \label{eq:G:extra:canel} \\
  G \sigma_j^m &= (-1)^{m+1} g j_1 \psi_j^{m+1}.
\label{eq:G:sigma:psi}
\end{align}
Note also that in what follows the superscripts $m$ appearing in the basis elements $\sigma_{k}^{m}$, $\psi_{k}^{m}$
are understood modulo 2, for example by $\sigma_{k}^{m + m'}$ we mean $\sigma_{k}^{m + m'(\textrm{mod } 2)}$.
For any $j \in \ZZ^2$ we define $j^\perp := (-j_2,  j_1)$.

We first show how the directions $\sigma^m_k, k\in \ZZ_+^2, m \in \{0,1\}$ can be obtained. Define
$Y_j^m(U) := [F(U),\sigma_j^m]$ and by
(\ref{eq:B:extra:canel}), (\ref{eq:G:sigma:psi}),
\begin{align}
Y_j^m(U)
	&=  A \sigma_j^m + B(\sigma_j^m, U) + B(U, \sigma_j^m) - G \sigma_j^m
\notag\\	
	&= \nu_2 |j|^2 \sigma_j^m + (-1)^{m}g j_1 \psi_j^{m+1} + B(U, \sigma_j^m).
\label{eq:YjmU}
\end{align}
Now set
\begin{align}
Z_j^m(U) := [F(U), Y_j^m(U)] =  \nabla Y_j^m(U) F(U) - \nabla F(U) Y_j^m(U)
\label{eq:zjm}
\end{align}
and after some computations we derive the explicit formula
\begin{align}
Z_j^m(U) &= B(F(U), \sigma_j^m) +
\nu_2^2 |j|^4 \sigma_j^m
  +(-1)^{m}  (\nu_1+\nu_2) g j_1 |j|^2  \psi_j^{m+1}  + A(B(U, \sigma_j^m))
  +(-1)^{m} g j_1 B(\psi_j^{m+1} , U) \notag\\
   &- B(U,  -\nu_2 |j|^2 \sigma_j^m + (-1)^{m+1} g j_1 \psi_j^{m+1}) +B(U,B(U, \sigma_j^m))
        - G B(U, \sigma_j^m). \label{eq:def:zujm}
\end{align}
Remarkably, as alluded to  above, the higher order bracket $[Z_{j}^m(U), \sigma_{k}^{m'}]$ is
independent of $U$.   To see this and to derive a formula for $[Z_{j}^m(U), \sigma_{k}^{m'}]$ one may proceed by explicit computations.
Instead, we argue as follows: by \eqref{eq:B:extra:canel}, \eqref{eq:YjmU} we obtain
$[Y^m_j(U), \tilde{U}] = -B(\tilde{U}, \sigma_{j}^m) = 0$ for any $\tilde{U} = (0, \tilde{U}_2)$, and consequently by the Jacobi identity,
\eqref{eq:jacobi:and:his:many:colored:identity}
\begin{align*}
[Z_{j}^m(U), \sigma_{k}^{m'}] = [[F(U), Y^m_j(U)], \sigma_{k}^{m'}] = - [Y^m_j(U), [F(U), \sigma_{k}^{m'}]] = -[Y^m_j(U), Y_{k}^{m'}(U)]\,.
\end{align*}
Moreover, since $(-1)^{m+1}g j_1 \psi_j^{m+1}$ is the only term with non-zero first component in $Y^m_j(U)$, from \eqref{eq:YjmU} we obtain
\begin{align*}
[Y^m_j(U), Y_{k}^{m'}(U)] = B((-1)^{m}g j_1 \psi_j^{m+1}, \sigma_{k}^{m'}) - B((-1)^{m'}g k_1 \psi_k^{m'+1}, \sigma_{j}^{m}),
\end{align*}
and therefore
\begin{align}
 [Z_{j}^m(U), \sigma_{k}^{m'}]
 = g \left((-1)^{m+1} j_1 B(\psi_j^{m+1}  , \sigma_k^{m'}) + (-1)^{m'} k_1B(\psi_k^{m'+1}, \sigma_j^m)\right) \,.
\label{eq:the:totalmiracle}
\end{align}
Next, note the following observation which
is a consequence of simple trigonometric identities.
\begin{Lem}
\label{Lem:B:psi:sigma}
For any $j, k \in \ZZ^2_+$ and $m, m' \in\{ 0,1\}$
\begin{align*}
B(\psi_j^{m}  , \sigma_k^{m'})
=& \frac{(-1)^{1+m m'}}{2}\frac{(j^\perp \cdot k) }{|j|^{2}}
\left[
\sigma^{m+m'}_{j+k} + (-1)^{m'+1} \sigma^{m+m'}_{j-k}
\right] \,, \\
B(\psi_j^{m}  , \psi_k^{m'})
=& \frac{(-1)^{1+m m'}}{2}\frac{(j^\perp \cdot k) }{|j|^{2}}
\left[
\psi^{m+m'}_{j+k} + (-1)^{m'+1} \psi^{m+m'}_{j-k}
\right]\,.
\end{align*}
\end{Lem}

Using Lemma \ref{Lem:B:psi:sigma} and \eqref{eq:the:totalmiracle} we have
\begin{align}
[Z_{j}^m(U), \sigma_{k}^{m'}]
 = &
g(-1)^{(m+1)(m'+1)}
 \frac{(j^\perp \cdot k) }{2} \left[(-1)^{m'}
 b(j,k)\sigma_{j-k}^{m+m'+1} -
 a(j,k)\sigma_{j+k}^{m+m'+1}
 \right] \,,
\label{eq:Z:sig:form}
\end{align}
where
\begin{align}\label{eq:sig:coeff:a}
a(j,k) := \frac{j_1}{|j|^2} + \frac{k_1}{|k|^2}
\quad \textrm{and}\quad
b(j,k) := \frac{j_1}{|j|^2} - \frac{k_1}{|k|^2}.
\end{align}
From these relations, the following proposition follows easily.

\begin{Prop}
\label{prop:generatesigma}
Let $j, k \in \ZZ_+^2$, $a(j,k), b(j,k)$ be as in
\eqref{eq:sig:coeff:a}.
Then,  with $Z_j^m$ given by
\eqref{eq:zjm},
\begin{align}
g (j^\perp\cdot k) a(j,k) \sigma_{j+k}^0 &=  -[Z_j^0(U), \sigma_k^1]
-[Z_j^1(U), \sigma_k^0],
\label{eq:new:sig:dir:1}\\
g (j^\perp\cdot k) a(j,k) \sigma_{j+k}^1 &= [Z_j^0(U), \sigma_k^0] - [Z_j^1(U), \sigma_k^1],
\label{eq:new:sig:dir:2}\\
g (j^\perp\cdot k) b(j,k) \sigma_{j-k}^0 &= [Z_j^1(U), \sigma_k^0] - [Z_j^0(U), \sigma_k^1],
\label{eq:new:sig:dir:3}\\
g (j^\perp\cdot k) b(j,k) \sigma_{j-k}^1 &= - [Z_j^1(U), \sigma_k^1] - [Z_j^0(U), \sigma_k^0].
\label{eq:new:sig:dir:4}
\end{align}
\end{Prop}

\begin{figure}[tb]
  \centering
 \scalebox{.65}{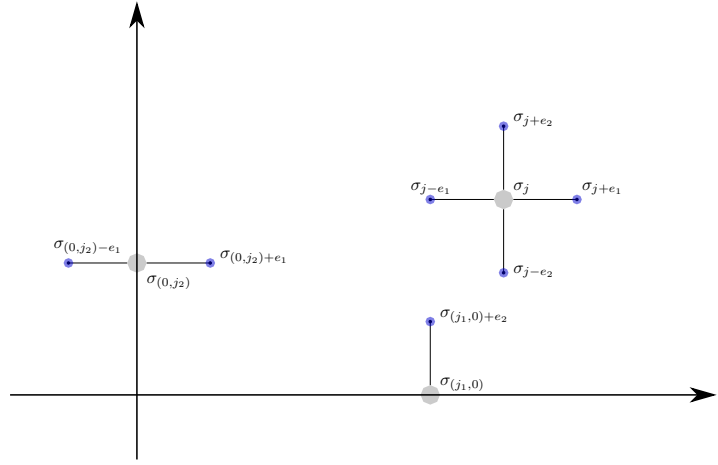}
 \vspace{-.1in}
 \caption{New basis elements $\sigma_j^m$ in the $\theta$ components of the phase space (see \eqref{eq:def:basis:sig}) that can be generated via
 the chain of brackets leading to \eqref{eq:Z:sig:form} with Proposition~\ref{prop:generatesigma} .
  Recall that we  index these basis elements by $\ZZ^2_{+}$ so that the points in figure lie on lattice points.  Note
  furthermore that elements $\sigma_{l}$ correspond to both $\sigma_{l}^0$, $\sigma_{l}^1$.
 Grey points represent existing directions and purple points represent new directions generated via one iteration
 of the chain of brackets illustrated in the upper part of Figure~\ref{fig:brak:1}.
  Perusing \eqref{eq:new:sig:dir:1}--\eqref{eq:new:sig:dir:4}
 and the definition of $a$ and $b$ in \eqref{eq:sig:coeff:a} we see why new directions are restricted along the $x$ and $y$ coordinate axises
 respectively.}
 \label{fig:new:sig:directions}
 \end{figure}

\begin{Rmk}
\label{rm:spanningsigma}
The diagram in Figure \ref{fig:new:sig:directions} and an induction argument detailed in Section~\ref{sec:it:proof:by:con} and illustrated
in Figure~\ref{fig:span:arg:1} show that starting
with the forced directions $\sigma_{(1,0)}^m, \sigma_{(0,1)}^m$ for each $m \in \{0,1\}$,
it is possible to reach $\sigma_k^m$ for any $k\in \ZZ_+^2$ and $m \in \{0,1\}$. If we replaced the
vectors $(1,0), (0, 1)$ in the definition of $\ZZZ$ by other elements in $\ZZ^2_+$, Figure~\ref{fig:new:sig:directions}
would change, namely the segments parallel to axes would be changed to segments parallel to the new directions in
$\ZZZ$.  In this case more a complicated algebraic condition as in \cite{HairerMattingly06} is needed to demonstrate that $\ZZZ$ generates a spanning
set for $\mathbb{Z}^2_+$.
\end{Rmk}

Directions $\psi_k^m$ are different
and they include an error term with a
component in the $\de$-direction. First note that  \eqref{eq:YjmU} can be rewritten as
\begin{equation}
\label{eq:psij1nonzero}
 (-1)^{m}g j_1 \psi_j^{m+1} =
Y_j^m(U) - (-1)^{m} g j_1 J^*_{j,m+1}(U) \,,
\end{equation}
where $J^*_{j,m+1}(U):= \frac{(-1)^{m} }{g j_1 }(\nu_2 |j|^2 \sigma_j^m + B(U,\sigma_j^m))$.  Note that,
by \eqref{eq:convection:term}, we see that $J_{j,m+1}^*$ is concentrated only in its $\theta$ component.
Since we can generate $Y_j^m(U)$, by \eqref{eq:psij1nonzero}, we can also generate
 $\psi_j^{m+1}$ (with an error term), whenever $j_1\neq 0$.  This constitutes the first downward branch
 in the lower portion of Figure~\ref{fig:brak:1}.

To reach the basis function $\psi_j^m$ along the $j_2$ axis ($j_1 = 0$) we can mimic the approach in \cite{EMattingly2001}
by considering brackets of the form $[[F(U), \psi_j^m], \psi_{k}^{m'}]$; cf. \eqref{eq:lie:nse}. Since we did not generate $\psi_j^m$, $\psi_{k}^{m'}$, we
instead use elements $Y_j^m(U)$, that is, $\psi_j^m$ with error terms and calculate
\begin{align}\label{eq:Z:Jac}
[[F(U), Y_j^m(U)], Y_{k}^{m'}(U)] = [Z_j^m(U), Y_{k}^{m'}(U)] =
\big[[Z_j^{m}(U),F(U)],\sigma_{k}^{m'}\big]- \big[[Z_j^{m}(U),\sigma_{k}^{m'}],F(U)\big] \,.
\end{align}
Notice, that the second identity follows from \eqref{eq:jacobi:and:his:many:colored:identity} and
shows that $[[F(U), Y_j^{m}(U)],Y_{k}^{m'}(U)]$ can be obtained from admissible bracket operations.
On the other hand, with \eqref{eq:zjm}
\begin{align*}
[[F(U), Y_j^{m}(U)],Y_{k}^{m'}(U)]
&=
\nabla^2 F(U)\{Y_j^{m}(U), Y_{k}^{m'}(U)\} + \nabla F(U) \{\nabla Y_j^{m}(U) \{Y_{k}^{m'}(U)\} \} \\
&\quad -
\nabla(\nabla Y_j^m (U) \{F(U)\}) \{Y_{k}^{m'}(U)\} + \nabla Y_{k}^{m'}(U) \{[F(U), Y_j^{m}(U)]\} \,,
\end{align*}
where $\nabla^k G(U) \{X_1, \cdots, X_k\}$ denotes the $k$th derivative of $G$ in the directions $X_1, \cdots, X_k$. With \eqref{eq:B:extra:canel}, \eqref{eq:psij1nonzero}, and \eqref{eq:drift:part:BE} we obtain
\begin{align*}
\nabla^2 F(U)\{Y_j^{m}(U), Y_{k}^{m'}(U)\} &= -B(Y_j^{m}(U), Y_{k}^{m'}(U)) - B(Y_{k}^{m'}(U), Y_j^{m}(U)) \\
&= (-1)^{m+m'+1}g^2 j_1 k_1 \Big(B(\psi_j^{m+1},\psi_k^{m'+1}) +  B(\psi_k^{m'+1},\psi_j^{m+1}) \\
&\quad + B(\psi_j^{m}, J^*_{k, m'}(U)) + B(\psi_{k}^{m'}, J^*_{j, m}(U)) \Big) \,.
\end{align*}
Also,
since $\nabla Y_j^{m}(U) \{Y_{k}^{m'}(U)\} = -B(Y_{k}^{m'}(U), \sigma_j^{m}) = (-1)^{m'}gk_1 B(\psi_{k}^{m'}, \sigma_j^{m})$ has a zero
$\omega$ component, we have
\begin{align*}
\nabla F(U) \{\nabla Y_j^{m}(U) \{Y_{k}^{m'}(U)\} \} = (-1)^{m'}gk_1 GB(\psi_{k}^{m'}, \sigma_j^{m}) +
D_{j, k}^{m, m'}(U) \,,
\end{align*}
where $U \mapsto D_{j, k}^{m, m'}(U)$ is affine and has a zero $\omega$ component. Finally, since $\nabla Y(U) \{X\}$ has a zero $\omega$ component
for any $X$ (of course the same is true for a derivative of $\nabla Y(U) \{X\}$), we obtain
\begin{align}
[Z_j^{m}(U),Y_{k}^{m'}(U)] =& (-1)^{m+m'+1}g^2 j_1 k_1 \Big( B(\psi_j^{m+1},\psi_k^{m'+1}) + B(\psi_k^{m'+1},\psi_j^{m+1})\Big) \notag\\
& +  (-1)^{m'}gk_1 GB(\psi_k^{m'+1},\sigma_{j}^{m})  +
 H^{m,m'}_{j,k}(U),
\label{eq:new:de:bad:tails}
\end{align}
where $U \mapsto H^{m,m'}_{j,k}(U)$ is affine and it is concentrated entirely in the $\theta$ component.

Using \eqref{eq:new:de:bad:tails} we are able to reach basis function that are not
accessible by the brackets leading to \eqref{eq:psij1nonzero}.  Note that this second
case $(j_1 = 0)$ is represented graphically the last lower branch of Figure~\ref{fig:brak:1}.

\begin{Prop}\label{prop:generatepsi}
Let $\ell_2>0, \ell=(0,\ell_2)$, $\ell'=(1, \ell_2)$, and ${\vec{e}_{1}}=(1,0).$ Then
\begin{align*}
g^{2}\frac{|\ell|^3}{|\ell'|^2} \psi^{0}_{\ell} =
- [Z_{\ell'}^{0}(U),Y_{\vec{e}_{1}}^0(U)]
-
[Z_{\ell'}^{1}(U),Y_{\vec{e}_{1}}^1(U)]
+H^{0,0}_{\ell',\vec{e}_{1},}(U)
+H^{1,1}_{\ell',\vec{e}_{1},}(U),
\end{align*}
and
\begin{align*}
g^{2}\frac{|\ell|^3}{|\ell'|^2}\psi^{1}_{\ell} =
[Z_{\ell'}^{1}(U),Y_{\vec{e}_{1}}^0(U)] -
[Z_{\ell'}^{0}(U),Y_{\vec{e}_{1}}^1(U)] +
H^{0,1}_{\ell',\vec{e}_{1},}(U) - H^{1,0}_{\ell',\vec{e}_{1}}(U).
\end{align*}
\end{Prop}

\begin{proof}
From Lemma \ref{Lem:B:psi:sigma} and the fact $G\sigma_\ell^m = 0$ (see \eqref{eq:G:extra:canel}) one has
\begin{align}
GB(\psi_{e_1}^{m'+1},\sigma_{\ell'}^m)
=& (-1)^{mm' + m'+ 1}g \ell_2\psi_{e_1+\ell'}^{m+m'}.
\label{eq:new:de:bad:tails:int}
\end{align}
Combining \eqref{eq:new:de:bad:tails}, \eqref{eq:new:de:bad:tails:int}, and Lemma \ref{Lem:B:psi:sigma}, we have
\begin{align}
   [Z_{\ell'}^m(U),Y_{e_1}^{m'}(U)] &=(-1)^{mm'+1} \frac{g^2 \ell_2}{2}\Big[\frac{2 + \ell_2^2}{1 + \ell_2^2}\psi_{\ell'+e_1}^{m+m'}
   + \frac{\ell_2^2}{1 + \ell_2^2}(-1)^{m'}\psi_{\ell'-e_1}^{m+m'}\Big]
	+ H_{\ell',e_1}^{m,m'}(U).
	\label{eq:new:de:bad:tails:2}
\end{align}
The proposition follows after eliminating the term involving $\psi_{\ell'+e_1}^{m+m'}$ by considering first
the cases $(m, m') = (0, 0)$, $(m, m') = (1, 1)$ to determine $\psi_{\ell}^{0}$ and $(m, m') = (0, 1)$, $(m, m') = (1, 0)$
to determine $\psi_{\ell}^{1}$.
\end{proof}

Combining \eqref{eq:psij1nonzero} and Proposition~\ref{prop:generatepsi} we now define the error term in the $\psi$-directions
\begin{align}
  J_{j,m}(U) =
  \begin{cases}
  	(-1)^{m} \dfrac{\mu |j|^2}{g j_1} \sigma^{m+1}_j + (-1)^{m} \dfrac{1}{g j_1} B(U, \sigma_j^{m+1})&
	\textrm{ if }  j_1 \not = 0,\\
	\dfrac{1 + |j|^2}{g^2 |j|^3 }
	(-H_{j + e_1, e_1}^{0, 0}(U) - H_{j + e_1, e_1}^{1, 1}(U)) & 	\textrm{ if }  j_1  = 0, m = 0,\\
	\dfrac{1 + |j|^2}{g^2 |j|^3 }
	(-H_{j + e_1, e_1}^{0, 1}(U) + H_{j + e_1, e_1}^{1, 0}(U)) & 	\textrm{ if }  j_1  = 0, m = 1 \,.
  \end{cases}
  \label{eq:junk:express}
\end{align}
By combining \eqref{eq:YjmU}, Proposition \ref{prop:generatepsi}, and \eqref{eq:junk:express},
we have for each $j\in\ZZ_{+}^2$, $m\in\{0,1\}$,
\begin{align}
  \psi_{j}^{m}+ J_{j,m}(U) =
  \begin{cases}
  	\dfrac{(-1)^{m}}{g j_1} Y_{j}^{m+1}(U) &
	\textrm{ if }  j_1 \not = 0,\\
	\dfrac{1 + |j|^2}{g^2 |j|^3 }
	(- [Z_{j+\vec{e_1}}^{0}(U),Y_{\vec{e}_{1}}^0(U)]- [Z_{j+\vec{e_1}}^{1}(U),Y_{\vec{e}_{1}}^1(U)]) & 	\textrm{ if }  j_1  = 0, m = 0,\\
	\dfrac{1 + |j|^2}{g^2 |j|^3 }
	([Z_{j+\vec{e_1}}^{1}(U),Y_{\vec{e}_{1}}^0(U)] - [Z_{j+\vec{e_1}}^{0}(U),Y_{\vec{e}_{1}}^1(U)]) & 	\textrm{ if }  j_1  = 0, m = 1 \,.
  \end{cases}
  \label{eq:junk:express:2}
\end{align}

\subsection{Estimates and Related Properties for the Error Terms}
We next summarize some basic properties of $J_{j,m}(U)$ in the following lemmata.
\begin{Lem}
\label{lem:junk:est:basic}
Fix any $j \in \ZZ_+^2$ with $|j| \leq N$, $m \in\{0,1\}$, and  any $U \in H^1$.  Then,
$(J_{j,m}(U))_\om = 0$ (the $\om$ component of $J_{j,m}(U)$
is zero) and
\begin{align*}
   \| J_{j,m}(U)\| \leq CN^3( 1 +  \|U\|_{H^1}) \,,
\end{align*}
where the constant $C$ is independent of $N$. Moreover $U \mapsto J_{j,m}(U) - J_{j,m}(0)$ is linear, that is, $U \mapsto J_{j,m}(U)$ is affine.
\end{Lem}

\begin{proof}
The proof is a straightforward consequence of definitions. For example by careful inspection we have an estimate
\begin{align*}
\|H_{j + e_1, e_1}^{m, m'}(U)\| \leq C|j|^4( 1 +  \|U\|_{H^1})\,,
\end{align*}
of any $m, m' \in \{0, 1\}$.
\end{proof}

Lemma \ref{lem:junk:est:basic} does not provide us with sufficient estimate for $J_{j, m}$ as
it grows both in $N$ and $U$. However, we crucially use the fact  $(J_{j, m})_\om = 0$ as follows.
We can generate sufficiently many, and consequently subtract from  $J_{j, m}$, pure modes $\sigma_k^m$. Hence,
we generate all modes $\sigma_k^m$ with $|k| \leq \tilde{N}$ but for approximation of $H_N$ we use only those with
$|k| \leq N \ll \tilde{N}$, the rest we use for controlling the size of the error $J_{j, m}$
(for details see the proof of Lemma \ref{lem:junk:e}).

To this end we derive estimates for projections of $J_{j,m}(U)$ into high Fourier modes.
Recall that $Q_N$ is the orthogonal projection on complement of $H_N$ and denote
\begin{align}
J^{\tilde{N}}_{j,m}(U) := Q_{\tilde{N}} J_{j,m}(U) \,.
\label{eq:junk:controlled}
\end{align}

\begin{Lem}
\label{lem:junk:est}
For every integers $N$, $\tilde{N}$ with $\tilde{N} \geq N > 0$,
and every integer $s \geq 1$ and $U \in H^{s+1}$
\begin{align}
   \| J^{\tilde{N}}_{j,m}(U)\| \leq& C \frac{N^{s+3}}{\tilde{N}^{s/2}} ( 1 + \| U\|_{H^{s+1}}) \qquad (|j| \leq N, m \in\{0,1\})\,,
   \label{eq:junk:est}
\end{align}
where $C = C(s)$ is independent of $N, \tilde{N}$ and $U$.
\end{Lem}

\begin{proof}
Since the $\tilde{N}^\textrm{th}$ eigenvalue $\lambda_{\tilde{N}} \sim \tilde{N}$, cf. \cite{ConstantinFoias88}, one has by the generalized Poincar\'e inequality that
\begin{align*}
\| J^{\tilde{N}}_{j,m}(U)\| \leq C \frac{1}{\lambda_{\tilde{N}}^{s/2}} \| J_{j,m}(U)\|_{H^s}
\leq  \frac{1}{\tilde{N}^{s/2}} \| J^{\tilde{N}}_{j,m}(U)\|_{H^s} \,.
\end{align*}
By careful inspection of \eqref{eq:junk:express}, noting that $J_{j,m}$ is affine in $U$, we obtain
\begin{align*}
\| J^{\tilde{N}}_{j,m}(U)\|_{H^s} \leq
C(1 + \|U\|_{H^{s+1}}) \,,
\end{align*}
where $C = C(N, s)$. The exact dependence of the right hand side on $N$ can be inferred from the fact that each derivative of $J_{j,m}(U)$
can produce at most one factor of $|j| \leq N$.
\end{proof}

\section{Noise Propagation in the Phase Space: Quantitative Estimates}
\label{sec:Braket:Est:Mal:Mat}

This section is devoted to the proof of Proposition \ref{prop:small:conclusion}.  To establish this
``upper bound" on the quadratic forms $\mathcal{Q}_{N,\tilde{N}}$ defined in Section~\ref{sec:Mal:spec:bnds},
recall that in Section~\ref{sec:Hormander:brak} we showed that $\mathfrak{B}_{N, \tilde{N}}(U) \subset \mathcal{V}_M$
 (see \eqref{eq:Mth:iter:admissible:brak}, \eqref{eq:approx:basis:tails}) for sufficiently large $M = M(\tilde{N})$.
We thus to translate each of the Lie bracket computations
in Section~\ref{sec:Hormander:brak} into quantitative bounds.  Roughly speaking, we would like to show that
\begin{align}
 \langle \MM_{0,T}\phi, \phi \rangle  \textrm{ is `small' implies that } \langle \phi, \sigma_k^l \rangle  \textrm{ are all `small' for all } k \in \ZZZ, l \in \{0,1\}
 \label{eq:rough:quant:bbd:stat1}
\end{align}
and that, starting from any admissible vector field $E \in \mathcal{V}_M$, cf. \eqref{eq:Mth:iter:admissible:brak}
\begin{align}
 \langle \phi, E \rangle  \textrm{ is `small' implies that } \langle \phi, [E, \sigma_k^l]\rangle, \langle \phi, [E, F]\rangle  \textrm{ are `small' for all } k \in \ZZZ, l \in \{0,1\}.
 \label{eq:rough:quant:bbd:stat2}
\end{align}

To achieve \eqref{eq:rough:quant:bbd:stat1}, \eqref{eq:rough:quant:bbd:stat2} we broadly follow an approach recently
developed in \cite{MattinglyPardoux1,BakhtinMattingly2007,HairerMattingly2011}.\footnote{As in these works, the
more classical methods using the Norris lemma do not apply, since it requires the inversion of the operators $\JJ_{0,t}$.
See \cite{Norris1986} and also e.g.
\cite{Nualart2006, Hairer2011} for further details.}
Notice that
\begin{align}
  \langle \MM_{s,t} \phi, \phi \rangle = \sum_{\substack{k \in \ZZZ\\ l \in\{0,1\}}}(\alpha_{k}^{l})^2 \int_s^t \langle \sigma_{k}^{l}, \KK_{r,t} \phi \rangle^2 dr
	\label{eq:MJ:matrix:back:form}
\end{align}
and define
$\gphi(t) := \langle \KK_{t,T} \phi, E(U(t))\rangle$ over test functions $\phi$ and admissible vector fields $E$.
To address the first case in  \eqref{eq:rough:quant:bbd:stat2} we make use of a change of variable
$\bar{U} := U - \sigma W$.  Expanding $E(U)$ in this new variable we
obtain a Wiener polynomial with coefficient similar to $[E,\sigma_k^l]$ and we infer the `smallness'
from time regularity results for Wiener polynomial derived in \cite{HairerMattingly2011} and recalled here
 as Theorem~\ref{thm:all:about:my:wiener}.  For the second case in \eqref{eq:rough:quant:bbd:stat2} we differentiate $\gphi$,
and find that $\gphi' = \langle \KK_{t,T} \phi ,[E, F] \rangle$, at least up to a change of variable.  We then make use of the
fact that we can bound the maximum of $\gphi'$ in terms of, for example, $\gphi$ and $C^{\alpha}$ norms of $\gphi'$ to
deduce the desired implication.

Observe that our quadratic forms $\mathcal{Q}_{N,\tilde{N}}$ depend on $U(T)$ and thus have a strong probabilistic dependence.
Indeed the existence of these `error' terms in \eqref{eq:our:quad:form:Horm}
means that we have to carefully track the growth of constants as a function of the number of Lie brackets we
take.  We also need to explain, at a quantitative level, how we are able to push error terms to entirely to high wavenumbers.
Neither of these concerns can be addressed from an  `obvious inspection' of the methods in \cite{HairerMattingly2011}.  In addition to these
mathematical concerns, we have developed several Lemmata~\ref{lem:suit:reg}, \ref{lem:aux:aux} which
we believe streamline the presentation of some of the arguments in comparison to previous works.

The rest of the section is organized as follows:  We begin with some generalities
introducing or recalling some general lemmata that will be used repeatedly in the
course of arguments leading to the rigorous form of \eqref{eq:rough:quant:bbd:stat1}--\eqref{eq:rough:quant:bbd:stat2}.
In Subsection~\ref{sec:implications:eigenvals} we present the series of Lemmas~\ref{lem:zeroth:step}--\ref{lem:de:to:de:creation:step}
each of which corresponds to one (or more) of the Lie
brackets computed in Section~\ref{sec:Hormander:brak}.    As we proceed we refer
to Figure~\ref{fig:brak:2} to help guide the reader through some admittedly involved computations.
In Subsection \ref{sec:it:proof:by:con} we piece together the proved implications in an inductive
argument to complete the proof of Proposition \ref{prop:small:conclusion}
\begin{figure}[tb]
  \centering
 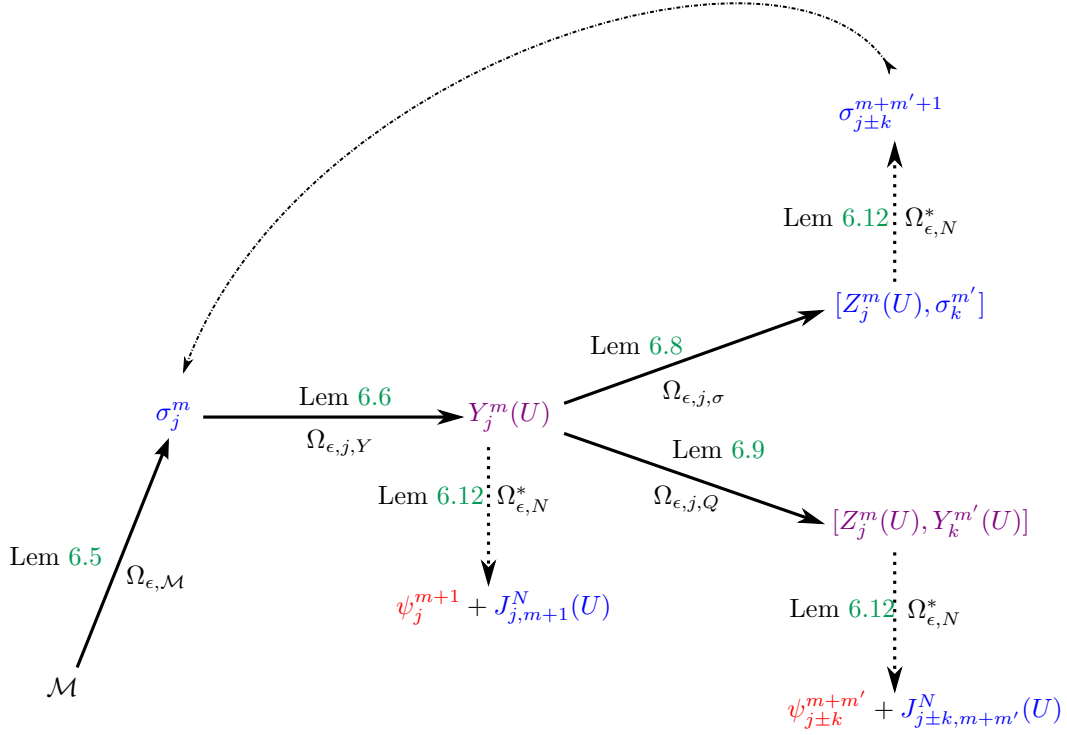
 \vspace{.1in}
 \caption{An illustration of the structure of the lemmata in Section~\ref{sec:Braket:Est:Mal:Mat} that leads to the proof of
 Proposition~\ref{prop:small:conclusion}.
 The arrows indicate that if one term is `small' then the other one is `small' on a set of large measure (displayed below the arrow),
 where the meaning of 'smallness' is made precise in each lemma.
 Above the arrow we indicate in which lemma is this implication proved.
  The arrow at the top of the
 diagram shows that the process is iterative.
 }
\label{fig:brak:2}
\end{figure}
In all that follows we maintain the convention from Remark \ref{rmk:constants:conven}, that is, all constants are implicitly dependent on given
parameters $\nu_1, \nu_2, g, \alpha, \ldots$ of the problem.
Note also that we carry out our arguments on a general time interval $[T/2, T]$ for some $T > 0$, which makes all constants $T$ dependent;
we apply the conclusion only for $T = 1$ above for the proof of Theorem~\ref{thm:mainresult}.

\subsection{Preliminaries}
We begin by introducing some further notational conventions and some general Lemmata~\ref{lem:suit:reg},~\ref{lem:aux:aux},~\ref{thm:all:about:my:wiener} which are be used frequently in the course of the analysis.

For any $a < b$, $\beta \in \RR$ and $\alpha \in (0,1]$ define the semi-norms
\begin{align*}
\|U\|_{C^{\alpha}([a,b],H^{\beta})} :=
\sup_{\substack{t_1 \neq t_2\\ t_1,t_2\in[a,b]}}
\frac{\| U(t_1)- U(t_2)\|_{H^{\beta}}}{|t_1-t_2|^{\alpha}}\,.
\end{align*}
If $a = T/2$ and $b = T$ we will write $\|\cdot \|_{C^\alpha H^\beta}$ instead of $\|\cdot\|_{C^{\alpha}([T/2,T],H^{\beta})}$
and denote $\|U\|_{C^{0} H^\beta} := \sup_{t \in [T/2,T]} \|U\|_{H^\beta}$.    Similar notations will be employed
for the H\"older spaces  $C^\alpha([a,b])$, $C^{1,\alpha}([a,b])$ etc.
Recalling the notation in \eqref{eq:Lie:brak:abs} we will define the `generalized Lie bracket'
\begin{align*}
  [E_1(U), E_2(\tilde{U})] := \nabla E_2(\tilde{U}) E_1(U) - \nabla E_1(U)  E_2(\tilde{U}),
\end{align*}
for all suitably regular $E_1, E_2 : H \to H$ and $U, \tilde{U} \in H$.
Below we often consider  $\Ushft = U - \sigma_\theta W$ which satisfies the
shifted equation (cf. \eqref{eq:BE:abs})
\begin{align}
  \pd_t \Ushft = F(U) = F(\bar{U} + \sigma_\theta W), \quad   \Ushft(0) = U_0.
\label{eq:ubar:def:2}
\end{align}
Note that, in contrast to $U$, $\bar{U}$ is $C^{1,\alpha}$ in time for any $\alpha < 1/2$.

We next prove two auxiliary lemmata which encapsulates the process of obtaining $[E,F]$ type brackets from time differentiation.

\begin{Lem}
\label{lem:suit:reg}
Suppose $E:H\rightarrow H$ is Fr\'{e}chet differentiable, $\phi\in H$, $U$ solves \eqref{eq:BE:abs},
and $\bar{U}$ is defined by \eqref{eq:ubar:def:2}.  Then for any $p\geq 1$ and any $\eta > 0$, we have that
\begin{align}
    \E& \sup_{t\in[T/2,T]}|\partial_{t}\langle \KK_{t,T}\phi,E(\bar{U})\rangle|^{p}
    \leq  C\|\phi\|^{p}\exp(\eta \|U_0\|^2)
    \Big(\E \sup_{t\in[T/2,T]}\|[E(\bar{U}),F(U)]\|^{2p} \Big)^{1/2},
    \label{eq:suit:reg1}
 \end{align}
  where $C = C(\eta, p, T)$.   Moreover, for any $\alpha \in (0,1]$,
\begin{align}
\E&\left(\|\partial_{t}\langle \KK_{t,T}\phi,E(\bar{U})\rangle\|^{p}_{C^{\alpha}}\right) \notag\\
    & \quad \quad \leq  C\|\phi\|^{p}\exp(\eta \|U_0\|^2)
    \cdot \Big[\Big(\E \sup_{t\in[T/2,T]}\|[E(\bar{U}),F(U)]\|_{H^2}^{2p}\Big)^{1/2}
    + \Big(\E \|[E(\bar{U}),F(U)]\|_{C^{\alpha} H}^{2p}\Big)^{1/2}\Big]\,,
    \label{eq:suit:reg2}
\end{align}
with $C = C(\eta, p, T)$.
\end{Lem}

\begin{proof}
Since $\KK_{t, T}\phi$ solves \eqref{eq:def:K:backwards:eq} and $\bar{U}$ satisfies \eqref{eq:ubar:def:2} we have
\begin{align}
    \pd_{t}\langle \KK_{t,T}\phi,E(\bar{U})\rangle &= \langle \pd_t  \KK_{t, T} \phi, E(\bar{U})\rangle
    +  \langle  \KK_{t, T} \phi,  \nabla E(\bar{U})\cdot\pd_{t}\Ushft \rangle
    \notag\\
         &= -\langle  \KK_{t, T} \phi, \nabla F(U) E(\bar{U})\rangle +  \langle  \KK_{t, T} \phi,  \nabla E(\bar{U})F(U) \rangle
    \notag\\
         &= \langle  \KK_{t, T} \phi, [F(U),E(\bar{U})]\rangle.
         \label{eq:gen:der}
\end{align}
Now, \eqref{eq:suit:reg1} immediately follows from
H\"{o}lder inequality and \eqref{eq:back:lin:est:1}. To prove \eqref{eq:suit:reg2},
 we use that for any $\alpha \in (0,1)$, $s, s' \in \RR$, and any suitably regular $A$, $B$
 one has
    \begin{align}
       \|\langle A, B \rangle\|_{C^{\alpha}} &:= \! \! \!
       \sup_{\substack{t \neq s\\ s, t \in [T/2, T]}} \left| \frac{\langle A(t), B(t) \rangle
       - \langle A(s), B(s) \rangle}{|s-t|^{\alpha}} \right|
       = \! \! \! \sup_{\substack{t \neq s\\ s, t \in [T/2, T]}} \left| \frac{\langle A(t) - A(s), B(t) \rangle
       + \langle A(s), B(t) - B(s) \rangle}{|s-t|^{\alpha}} \right|
       \notag \\
       &\leq \|A\|_{L^\infty  H^{-s}} \|B \|_{C^{\alpha}H^{s}}  + \|A\|_{C^{\alpha}H^{-s'}} \|B\|_{L^\infty H^{s'}}.
       \label{eq:holder:space:time:algebra:props}
    \end{align}
Combining \eqref{eq:holder:space:time:algebra:props} with \eqref{eq:gen:der} and using H\"{o}lder's inequality,
\begin{align}
    \E\big(\|\partial_{t}\langle \KK_{t,T}\phi,E(\bar{U})\rangle\|_{C^{\alpha}}^p\big)
    &\leq  C   \Big(\E\big(\sup_{t\in[T/2,T]}\|\KK_{t,T}\phi\|^{2p}\big)\Big)^{1/2}
    \Big(\E\big(\|[E(\bar{U}),F(U)]\|_{C^{\alpha}H}^{2p}\big)\Big)^{1/2}
    \notag \\ &\qquad  + C
 \Big(\E\big(\|\KK_{t,T}\phi\|_{C^{\alpha} H^{-2}}^{2p}\big)\Big)^{1/2}
    \Big(\E\big(\sup_{t\in[T/2,T]}\|[E(\bar{U}),F(U)]\|_{H^2}^{2p}\big)\Big)^{1/2}
    \notag
\end{align}
and \eqref{eq:suit:reg2} follows from \eqref{eq:back:lin:est:1} and  \eqref{eq:lin:smooth}.
\end{proof}

\begin{Lem}\label{lem:aux:aux}
Fix $T > 0$, $\alpha \in (0,1]$ and an index set $\mathcal{I}$.
Consider a collection of random functions $\gphi$ taking values in $C^{1, \alpha}([T/2, T])$ and indexed by
$\phi \in \mathcal{I}$. Define, for each $\epsilon > 0$,
 \begin{align}\label{eq:def:Lambda}
\Lambda_{\epsilon, \alpha}: = \bigcup_{\phi \in I} \Lambda_{\epsilon, \alpha}^\phi,
\quad \textrm{ where } \quad
\Lambda_{\epsilon, \alpha}^\phi := \left\{ \sup_{t \in [T/2, T]} |\gphi (t)| \leq \epsilon
\textrm{  and  }  \sup_{t \in [T/2, T]} |\gphi'(t)|  > \epsilon^{\frac{\alpha}{2(1 + \alpha)}} \right\}.
 \end{align}
Then, there is $\epsilon_0 = \epsilon_0(\alpha, T)$ such that for each
$\epsilon \in (0, \epsilon_0)$
\begin{align}
\Prb (\Lambda_{\epsilon, \alpha}) \leq C  \epsilon \E \left(\sup_{\phi \in I} \|\gphi\|_{C^{1, \alpha}([T/2, T])}^{2/\alpha}\right)\,.
\label{eq:HM:wrapper:thm:conclusion}
\end{align}
\end{Lem}

\begin{proof}
As observed in \cite[Lemma 6.14]{HairerMattingly2011} we have the elementary bound
 \begin{align}
    \| f '\|_{L^\infty} \leq 4  \|f\|_{L^\infty} \max \left\{\frac{2}{T},
    \| f\|_{L^\infty}^{-1/(1 + \alpha)}
    \| f' \|_{C^{\alpha}}^{1/(1 + \alpha)}  \right\} \,,
    \label{eq:holder:interp:0}
 \end{align}
 which is valid for any $f \in C^{1,\alpha}([T/2,T])$.
Fix any $\phi \in \mathcal{I}$.  On the set $\Lambda_{\epsilon, \alpha}^\phi$, if $\gphi$ attains the maximum in \eqref{eq:holder:interp:0} in the first term, then
\begin{align*}
\epsilon^{\frac{\alpha}{2(1 + \alpha)}} <   \|\gphi'\|_{L^{\infty}} \leq
4 \|\gphi\|_{L^{\infty}} \frac{2}{T}\leq \epsilon \frac{8}{T} \,.
\end{align*}
Clearly this cannot happen if $\epsilon < \epsilon_0(\alpha, T) \,(:= (T/8)^{2(1+\alpha)/(2+ \alpha)})$.  Thus for any $\epsilon < \epsilon_0$, on $\Lambda_{\epsilon, \alpha}^\phi$, one has
\begin{align*}
\epsilon^{\frac{\alpha}{2(1 + \alpha)}} < 4 \epsilon^{\alpha/(1 + \alpha)} \|\gphi'\|_{C^{\alpha}}^{1/(1 + \alpha)}  \,,  \quad \textrm{ i.e. }
\quad \|\gphi'\|_{C^{\alpha}}^{2/\alpha} \geq 4^{-2/\alpha(1 + \alpha)} \epsilon^{-1} := C(\alpha) \epsilon^{-1}\,.
\end{align*}
Since this lower bound is independent of $\phi \in \mathcal{I}$ we infer
\begin{align*}
  \Lambda_{\epsilon, \alpha} \subset \left\{  \sup_{\phi \in \mathcal{I}}  \|\gphi'\|_{C^{\alpha}}^{2/\alpha} \geq C(\alpha) \epsilon^{-1}  \right\}.
\end{align*}
With this observation and the Markov inequality we infer \eqref{eq:HM:wrapper:thm:conclusion}, completing the proof.
\end{proof}

\begin{Rmk}\label{rmk:sets:give:monster:a:headache}
  Observe that
  \begin{align*}
  \Lambda_{\epsilon, \alpha}^c = \bigcap_{\phi \in I} \left\{ \sup_{t \in [T/2, T]} |\gphi (t)| > \epsilon
\textrm{  or  }  \sup_{t \in [T/2, T]} |\gphi'(t)|  \leq \epsilon^{\frac{\alpha}{2(1 + \alpha)}} \right\}.
  \end{align*}
  Thus, on $\Lambda_{\epsilon, \alpha}^c$,
  \begin{align}
     \sup_{t \in [T/2, T]} |\gphi (t)| < \epsilon \quad \Rightarrow  \quad \sup_{t \in [T/2, T]} |\gphi'(t)|  \leq \epsilon^{\frac{\alpha}{2(1 + \alpha)}}
     \label{eq:trivial:obs:237}
  \end{align}
  for every $\phi \in \mathcal{I}$.
\end{Rmk}

Finally, we recall in our notations, a crucial quantitative
bound on Wiener polynomials established in \cite{HairerMattingly2011}.
In particular this restatement avoids the language of `almost implication' introduced in \cite{HairerMattingly2011, Hairer2011}.

Given any multi-index $\alpha := (\alpha_1, \ldots, \alpha_d) \in \mathbb{N}^d$
recall the standard notation $W^\alpha := W^{\alpha_1}_1\cdots W^{\alpha_d}_d$.
\begin{Thm}[Hairer-Mattingly, \cite{HairerMattingly2011}]
\label{thm:all:about:my:wiener}
Fix $M, T > 0$. Consider the collection $\mathfrak{P}_M$ of $M^{th}$ degree  of  `Wiener polynomials' of the form
\begin{align*}
  F = A_{0} + \sum_{|\alpha| \leq M} A_{\alpha} W^\alpha,
\end{align*}
where for each multi-index $\alpha$, with $|\alpha| \leq M$,
$A_{\alpha}: \Omega \times [0,T] \to \RR$ is an arbitrary stochastic process.  Then, for all $\epsilon \in (0,1)$
and $\beta > 0$, there exists a measurable set $\Omega_{\epsilon,M, \beta}$ with
\begin{align*}
  \Prb(\Omega_{\epsilon,M,\beta}^c) \leq C \epsilon,
\end{align*}
such that on $\Omega_{\epsilon,M,\beta}$ and for every $F \in \mathfrak{P}_M$
\begin{align*}
  \sup_{t \in [0,T]} |F(t)| < \epsilon^{\beta}
  \quad \Rightarrow  \quad
  \begin{cases}
   \textrm{ either } &  \sup\limits_{|\alpha| \leq M} \sup\limits_{ t \in [0,T]} |A_{\alpha}(t)| \leq \epsilon^{\beta 3^{-M}},\\
   \textrm{ or } & \sup\limits_{|\alpha| \leq M} \sup\limits_{s \not = t \in [0,T]} \frac{ |A_{\alpha}(t) - A_{\alpha}(s)|}{|t -s|}\geq \epsilon^{-\beta 3^{-(M+1)}}.
  \end{cases}
\end{align*}
\end{Thm}

\subsection{Implications Starting from Small Eigenvalues}
\label{sec:implications:eigenvals}

We now start proving the implications depicted in Figure~\ref{fig:brak:2}.  Note that throughout what
follows we fix a small constant $\epsilon_0 = \epsilon_0(T)$ which gives the range of $\epsilon$ values
for which Lemmas~\ref{lem:zeroth:step}--\ref{lem:de:to:de:creation:step} hold.  The first lemma
explains how a lower bounds bound on the eigenvalues of $\MM_{0,T}$
initiates the iteration.

\begin{Lem}
\label{lem:zeroth:step}
For every $0<\epsilon < \epsilon_0(T)$ and every $\eta > 0$ there exist a set
$\Omega_{\epsilon, \mathcal{M}}$ and $C = C(\eta, T)$ with
\begin{align*}
  \Prb({\Omega_{\epsilon,\mathcal{M}}^{c}})
  \leq C \exp (\eta \|U_0\|^2)\epsilon
\end{align*}
such that on the set $\Omega_{\epsilon, \mathcal{M}}$
\begin{align*}
  \langle \MM_{0,T} \phi, \phi \rangle \leq \epsilon \|\phi\|^{2}
  \quad \Rightarrow \quad
  \sup_{t \in [T/2, T]} | \langle \KK_{t, T} \phi, \sigma_{k}^{l}\rangle |
  \leq \epsilon^{1/8} \|\phi\|,
\end{align*}
for each $k \in \ZZZ$, $l \in\{0,1\}$, and every $\phi \in H$. We recall that $\ZZZ \subset \ZZ^2_{+}$ is the set of directly forced
modes as in \eqref{eq:stochastic:forcing:exact:form} and the elements $\sigma_k^l$
are given by \eqref{eq:def:basis:sig}.
\end{Lem}

\begin{proof}
For any $\phi \in H$ with $\|\phi\| = 1$, define
 \begin{align*}
     \gphi (t) := \sum_{\substack{k \in \ZZZ\\ l \in\{0,1\}}} (\alpha_{k}^l)^2 \int_0^t \langle \sigma_{k}^{l}, \KK_{r,T} \phi \rangle^2 dr
     \leq \sum_{\substack{k \in \ZZZ\\ l \in\{0,1\}}} (\alpha_{k}^l)^2 \int_0^T \langle \sigma_{k}^{l}, \KK_{r,T} \phi \rangle^2 dr =
     \langle \MM_{0,T} \phi, \phi \rangle,
 \end{align*}
see \eqref{eq:MJ:matrix:back:form}.  Note that
  \begin{align*}
      \gphi'(t) = \sum_{\substack{k \in \ZZZ\\ l \in\{0,1\}}}  (\alpha_{k}^l)^2 \langle \sigma_{k}^{l}, \KK_{t,T} \phi \rangle^2, \quad
      \gphi''(t) = 2 \sum_{\substack{k \in \ZZZ\\ l \in\{0,1\}}} (\alpha_{k}^l)^2 \langle \sigma_{k}^{l}, \KK_{t,T} \phi \rangle \langle \sigma_{k}^{l}, \pd_{t}\KK_{t,T} \phi \rangle.
 \end{align*}
Let $\Omega_{\epsilon, \MM} := \Lambda_{\epsilon, 1}^c$, where $\Lambda_{\epsilon, \alpha}$ is as in \eqref{eq:def:Lambda}
with $\mathcal{I} := \{ \phi \in H: \| \phi\| =1\}$.
 Then by Lemma \ref{lem:aux:aux} with $\alpha = 1$, \eqref{eq:back:lin:est:1}, and \eqref{eq:lin:smooth}  one has
\begin{align*}
\Prb (\Omega_{\epsilon, \MM}^c) &\leq
C \epsilon \sum_{\substack{k \in \ZZZ \\ l \in\{0,1\}}} (\alpha_k^l)^4
\E \left(\sup_{\substack{t \in [T/2, T]\\ \|\phi\| = 1} } \left| \langle \sigma_{k}^{l}, \KK_{t,T} \phi \rangle
\langle \sigma_{k}^{l}, \pd_{t}\KK_{t,T} \phi \rangle \right|^{2} \right)
 \leq C \exp (\eta \|U_0\|^2) \epsilon
\end{align*}
for any $\epsilon < \epsilon_0 = \epsilon_0 (T)$, where $C = C(\eta, T)$.
Finally, on $\Omega_{\epsilon, \MM}$ we have, cf. \eqref{eq:trivial:obs:237}, that
\begin{align*}
\langle \MM_{0,T} \phi, \phi \rangle \leq \epsilon \|\phi\|^{2}
  \quad \Rightarrow \quad
  \sup_{t \in [T/2, T]} |\alpha_{k, l}| | \langle \KK_{t, T} \phi, \sigma_{j}^{l}\rangle |
  \leq \epsilon^{1/2} \|\phi\|,
\end{align*}
for each $k\in\ZZZ$, $l\in\{0,1\}$ and any $\phi \in H$. Since $\alpha_{k}^l \neq 0$, the assertion of the lemma
follows for $\epsilon \leq \epsilon_0(T)$.
\end{proof}

We next turn to implications of the form $\sigma \to [F, \sigma] =Y$; see Figure~\ref{fig:brak:2}.

\begin{Lem}
\label{lem:de:to:om:creation:step}
Fix any $j\in\ZZ_{+}^{2}$.  For each $0<\epsilon < \epsilon_0(T)$ and $\eta > 0$ there exist a set
${\Omega_{\epsilon,j,Y}}$  and $C = C(\eta, T)$ with
\begin{align*}
  \Prb({\Omega_{\epsilon,j,Y}^{c}}) \leq C|j|^{8}\exp(\eta \|U_0\|^2)\epsilon,
\end{align*}
such that on the set $\Omega_{\epsilon,j, Y}$, for each $m\in\{0,1\}$, it holds that
\begin{align}\label{eq:si:imp:Y}
 \sup_{t \in [T/2, T]} | \langle \KK_{t, T} \phi, \sigma_{j}^{m}\rangle | \leq \epsilon \|\phi\|
  \quad \Rightarrow \quad
  \sup_{t \in [T/2, T]} | \langle \KK_{t, T} \phi,Y_j^m(U)\rangle | \leq \epsilon^{1/4} \|\phi\|.
\end{align}
\end{Lem}

\begin{proof}
By expanding $U = \Ushft + \sigma W$, and using \eqref{eq:B:extra:canel}, we observe that
\begin{align}
Y_j^m(U) = Y_j^m(\Ushft).
\label{eq:noise:can:eq:1}
\end{align}
Then for fixed $m\in\{0,1\}$ and any $\phi \in \mathcal{I} := \{ \phi \in H : \|\phi\| = 1\}$ define
$\gphi (t) :=  \langle \KK_{t, T} \phi, \sigma_{j}^{m}\rangle$
and observe by \eqref{eq:gen:der} and \eqref{eq:noise:can:eq:1} that
$\gphi'(t) =  \langle \KK_{t, T} \phi, [F(U),\sigma_{j}^{m}]\rangle
 =  \langle \KK_{t, T} \phi, Y_{j}^{m}(U)\rangle
 = \langle \KK_{t, T} \phi, Y_{j}^{m}(\bar{U})\rangle$.
Let $\Omega_{\epsilon,j, Y} := \Lambda_{\epsilon, 1}^c$ with $\Lambda_{\epsilon, 1}$ as in \eqref{eq:def:Lambda}.
Once again, with \eqref{eq:trivial:obs:237}, we see that \eqref{eq:si:imp:Y} holds on $\Omega_{\epsilon,j, Y}$.
On the other hand, by \eqref{eq:HM:wrapper:thm:conclusion}, \eqref{eq:suit:reg1}, \eqref{eq:noise:can:eq:1}, \eqref{eq:zjm}, and
\eqref{eq:smoothing:est}, we have
  \begin{align*}
      \Prb( \Omega_{\epsilon,j, Y}^c) &\leq  C \epsilon
      \E \left( \sup_{\phi \in \mathcal{I}} \sup_{t\in[T/2,T]} |\partial_t \langle \KK_{t, T} \phi, Y_j^m (\bar{U}) \rangle|^2 \right)
	 \leq  C\epsilon \exp\left( \frac{\eta}{2}  \|U_0\|^2\right)
    \E \sup_{t\in [T/2,T]}\|Z_{j}^{m}(U)\|^{4}
    \notag\\
    &\leq  C\epsilon |j|^{8}\exp\left( \frac{\eta}{2} \|U_0\|^2\right)
    \E \left(1+\sup_{t\in[T/2,T]}\|U\|_{H^{2}}^8\right)
   \leq  C\epsilon |j|^{8}\exp\left( \eta \|U_0\|^2\right)
   \end{align*}
for any $\epsilon < \epsilon^\ast (T)$,   where $C = C(\eta, T)$.
For the third inequality above we have also used the estimate
 \begin{align}
 \sup_{t\in[T/2,T]}\|Z_{j}^{m}(U)\|_{H^s}\leq C|j|^{4+s} \left(1+\sup_{t\in[T/2,T]}\|U\|_{H^{s+2}}^2\right),
 \label{eq:Z:newbound:1}
 \end{align}
which follows from \eqref{eq:def:zujm} by counting derivatives and applying the H\"{o}lder and Poincar\'{e} inequalities.
\end{proof}

\begin{Rmk}
The constants in the exponents of $|j|$ and $\epsilon$ in the forthcoming Lemmas~\ref{lem:om:to:de:creation:step}, \ref{lem:de:to:de:creation:step}
rapidly become; however, there is nothing special about these numbers.  We simply need to track that
in the bounds $|j|$ and $\epsilon$ grow like $|j|^\tau$ and $\epsilon^\kappa$ respectively for some $\kappa, \tau > 0$.
\end{Rmk}

We next establish implications corresponding the chain of brackets $Y \to Z \to [Z, \sigma]$. We refer again to
the  Figure~\ref{fig:brak:2}.

\begin{Lem}
\label{lem:om:to:de:creation:step}
Fix $j\in\ZZ_{+}^{2}$.  For each
$0<\epsilon < \epsilon_0(T)$, and $\eta > 0$
there exist a set ${\Omega_{\epsilon,j,\sigma}}$ and $C = C(\eta, T)$ with
\begin{align}
  \Prb(\Omega_{\epsilon,j,\sigma}^{c}) \leq C |j|^{90\times 6}\exp( \eta \|U_0\|^2) \epsilon,
    \label{eq:good:set:3}
\end{align}
such that on the set $\Omega_{\epsilon,j, \sigma}$, for each $m\in\{0,1\}$, it holds that
\begin{equation*}
 \sup_{t \in [T/2,T]} | \langle \KK_{t, T} \phi, Y_j^m(U)\rangle | \leq \epsilon \|\phi\|
 \quad \Rightarrow \quad
 \begin{cases}
  \sup \limits_{t \in [T/2,T]} |  \langle \KK_{t, T} \phi, Z_j^m(\bar{U})\rangle | \leq \epsilon^{1/30}\|\phi\|,  \notag \\
  \sup\limits_{\substack{k \in \ZZZ\\l \in\{0,1\}}} \sup_{t \in [T/2,T]} |  \langle \KK_{t, T} \phi, [Z_j^m(U),\sigma_{k}^{l}] \rangle |
    \leq \epsilon^{1/60} \|\phi\|.
 \end{cases}
\end{equation*}
\end{Lem}

\begin{proof}
In the course of the proof we suppress, for the sake of brevity,
the subscript $\sigma$ in the definition of
various sets leading to $\Omega_{\epsilon,j, \sigma}$.
For fixed $m\in\{0,1\}$ and $\phi \in H$ let $\gphi(t):=  \langle \KK_{t, T} \phi, Y_j^m({U})\rangle = \langle \KK_{t, T} \phi, Y_j^m({\bar{U}})\rangle$ (cf. \eqref{eq:noise:can:eq:1})
so that $\gphi'(t) = \langle \KK_{t,T}\phi, [Y_j^m({\bar{U}}), F(U)]\rangle = -\langle \KK_{t,T}\phi, Z_j^m(U)\rangle$ (see \eqref{eq:gen:der}, \eqref{eq:zjm}).
Let
$\Omega^{1}_{\epsilon,j} = \Lambda_{\epsilon, 1/4}^c$, where $\Lambda_{\epsilon, \alpha}^c$ is as in \eqref{eq:def:Lambda} over
with $\mathcal{I} := \{ \phi \in H: \|\phi\| =1\}$. Then, on
$\Omega^{1}_{\epsilon,j}$  one has, in view of \eqref{eq:trivial:obs:237},
\begin{equation} \label{eq:name2}
 \sup_{t \in [T/2,T]} | \langle \KK_{t, T} \phi, Y_j^m(U)\rangle | \leq \epsilon \|\phi\| \quad
 \Rightarrow \quad \sup_{t \in [T/2,T]} |  \langle \KK_{t, T} \phi, Z_j^m(U)\rangle |
  \leq \epsilon^{1/10}\|\phi\| \,.
\end{equation}
By Lemma \ref{lem:aux:aux} with $\alpha = 1/4$ and \eqref{eq:suit:reg2}, \eqref{eq:smoothing:est},
\eqref{eq:smoothing:est:2} we have
\begin{align}
\Prb ((\Omega^{1}_{\epsilon,j})^c) &\leq C \epsilon \E \left(\sup_{\phi \in \mathcal{I}}\|g'\|_{C^{1/4}}^8 \right)\notag \\
&\leq C\epsilon\exp(\frac{\eta}{2} \|U_0\|^2)
\Big[ \Big(\E \sup_{t\in[T/2,T]}\|Z_j^m(U)\|_{H^2}^{16} \Big)^{1/2}
     + \Big(\E \|Z^m_j(U)\|_{C^{1/4}H}^{16}\Big)^{1/2}\Big] \notag \\
        & \leq C\epsilon|j|^{48}\exp(\frac{\eta}{2} \|U_0\|^2)
    \bigg[\Big(\E\big(1 + \sup_{t\in[T/2,T]}\|U\|_{H^{4}}^{32} \big)\Big)^{1/2}
    \notag \\ &\quad\quad\quad\quad\quad\quad\quad\quad\quad \quad
    + \Big(\E\Big(\|U\|_{C^{1/4} H^{2}}^{16}
    \big(1 + \sup_{t\in[T/2,T]}\|U\|_{H^{2}}^{16} \big)
    \Big)\Big)^{1/2}\bigg]
    \notag \\
    & \leq C\epsilon|j|^{48}\exp(\eta \|U_0\|^2) \,,
    \label{eq:est:Om:1}
\end{align}
where $C = C(\eta, T)$ and we used
 the bilinearity of $Z$ with estimates like those leading to \eqref{eq:Z:newbound:1}.
Next, by expanding $U=\bar{U} + \sigma W$ we find
\begin{align}
  Z_j^{m}(U) = Z_j^{m}(\bar{U}) -
  \sum_{\substack{k\in \ZZZ\\ l \in\{0,1\}}}  {\alpha_{k}^l} [Z_j^m(U), \sigma_k^l] W^{k,l}.
  \label{eq:ZtoZbar}
\end{align}
In view of \eqref{eq:the:totalmiracle},  all of the second order terms in \eqref{eq:ZtoZbar}
of the form $[[Z(U),\sigma_k^l], \sigma_{k'}^{l'}]W^{k,l}W^{k',l'}$ are zero and
$[Z_j^m(U), \sigma_k^l] = [Z_j^m(\bar{U}), \sigma_k^l]$.

To estimate each of the terms in \eqref{eq:ZtoZbar}, we introduce for $s \in \{0,1\}$, $\phi \in H$
\begin{equation*}
{\mathcal{N}}_s(\phi) :=
\max_{k\in\ZZZ,l\in\{0,1\}} \left\{
\left\|\langle \KK_{t, T} \phi, Z_j^m(\bar{U})\rangle
\right\|_{C^s},
|\alpha_{k,l}|\left\|\langle \KK_{t, T} \phi, [Z_j^m(U),\sigma_{k}^{l}] \rangle \right\|_{C^s}
 \right\}.
\end{equation*}
By Theorem~\ref{thm:all:about:my:wiener}, there exists a set $\Omega_{\epsilon}^\sharp$ such that
$\Prb((\Omega_{\epsilon}^\sharp)^c) < C \epsilon,$
and on $\Omega_{\epsilon}^\sharp$
\begin{align}\label{eq:dich}
\sup_{t \in [T/2,T]} |  \langle \KK_{t, T} \phi, Z_j^m(U)\rangle| \leq \epsilon^{1/10}
\quad
\implies
\quad
\left\{
\begin{array}{rl}
\it{either} &
{\mathcal{N}}_0(\phi) \leq \epsilon^{1/30}, \\
 \it{or} &
{\mathcal{N}}_1(\phi) \geq \epsilon^{-1/90}.
\end{array}
\right.
\end{align}
Recalling that $\mathcal{I} = \{ \phi \in H: \| \phi\| =1\}$, let
\begin{align*}
\Omega_{\epsilon,j}^2 :=
 \bigcap_{\phi \in \mathcal{I}}\{ {\mathcal{N}}_0(\phi) < \epsilon^{1/30}\}\cap \Omega_{\epsilon, j}^\sharp.
\end{align*}
By \eqref{eq:name2}
on the set $\Omega_{\epsilon,j,\sigma} := \Omega_{\epsilon,j}^1 \cap \Omega_{\epsilon,j}^2$ we obtain
the desired conclusion for each $\epsilon < \epsilon_0(T)$.
Thus it remains to estimate the size of $\Omega_{\epsilon,j,\sigma}^c$. By \eqref{eq:dich}, \eqref{eq:est:Om:1},
and the Markov inequality we have
\begin{align}
\Prb (\Omega_{\epsilon,j,\sigma}^c) &\leq \Prb((\Omega_{\epsilon,j}^1)^c) + \Prb((\Omega_{\epsilon, j}^\sharp)^c)
+ \Prb\left( \sup_{\phi \in \mathcal{I}} \mathcal{N}_1(\phi) \geq \epsilon^{-1/90}\right) \notag \\
&\leq C|j|^{48}\exp(\eta \|U_0\|^2) \epsilon + C \epsilon \E\left(\sup_{\phi \in \mathcal{I}} (\mathcal{N}_1(\phi))^{90}\right) \,.
\label{eq:ana:est:2}
\end{align}
However, by \eqref{eq:suit:reg1} and \eqref{eq:smoothing:est} along with further estimates along the lines leading to  \eqref{eq:Z:newbound:1} we have
\begin{align}
\E \left\|\langle \KK_{t, T} \phi, Z_j^m(\bar{U})\rangle
\right\|_{C^1([T/2,T];\mathbb{R})}^{90} &\leq C \exp(\eta/2 \|U_0\|^2)
\left(\E \sup_{t \in[T/2, T]} \|[Z_j^m(\bar{U}), F(U)]\|^{180} \right)^{1/2} \notag \\
&\leq C \exp(\eta/2 \|U_0\|^2) |j|^{90\times 6} \left(\E(1 + \|U\|_{H^{4}}^{3\times 180})\right)^{1/2} \notag \\
&\leq C \exp(\eta \|U_0\|^2) |j|^{90\times 6}\,,
\label{eq:ana:est}
\end{align}
where $C = C(\eta, T)$.
Finally, due to \eqref{eq:the:totalmiracle} and similar applications of  \eqref{eq:suit:reg1} and \eqref{eq:smoothing:est} the estimate
\begin{equation}
\E \left\|\langle \KK_{t, T} \phi, [Z_j^m(U),\sigma_{k}^{l}] \rangle \right\|_{C^1([T/2,T];\mathbb{R})}^{90}
\leq C \exp(\eta \|U_0\|^2) |j|^{90\times 2}
\label{eq:ana:est:3}
\end{equation}
follows. By combining \eqref{eq:ana:est:2}-\eqref{eq:ana:est:3} we obtain \eqref{eq:good:set:3}, and the proof is complete.
\end{proof}

The final lemma of this section corresponds to brackets of the form $Y \to Z \to [Z, Y]$.
For fixed $j\in\ZZ_{+}^{2}$, define $\ZZZ_{j}$ as the union of $j$ with the set of points in $\ZZ_{+}^{2}$ adjacent to $j$,  that is,
$$\ZZZ_{j}:= \{k\in \ZZ_{+}^{2}:k=j\pm m \ \text{for some}
\ m\in\{0\}\cup \ZZZ\}.$$

\begin{Lem}
\label{lem:de:to:de:creation:step}
Fix $j\in\ZZ_{+}^{2}$.  For each $0<\epsilon < \epsilon_0(T)|j|^{-2}$ and $\eta > 0$
there exists $C = C(\eta, T)$ and a measurable set $\Omega_{\epsilon,j,Q}$ with
\begin{align}
  \Prb((\Omega_{\epsilon,j,Q})^{c}) \leq C |j|^{14\times 5400}\exp(\eta \|U_0\|^2)
  \epsilon,
    \label{eq:good:set:4}
\end{align}
such that on the set $\Omega_{\epsilon,j,Q}$ it holds that, for every $\phi \in H$,
\begin{align}
&\sum_{\substack{i\in\ZZZ_{j}\\m\in\{0, 1\}}}\sup_{t \in [T/2,T]} | \langle \KK_{t, T} \phi, Y_i^m(U)\rangle | \leq \epsilon \|\phi\| \label{eq:assu} \\
  &\quad \quad \quad \Rightarrow
  \sum_{\substack{k\in\ZZZ\\ m,l\in\{0,1\}}}\sup_{t \in [T/2,T]} | \langle \KK_{t, T} \phi,[Z_j^m(U),Y_{k}^{l}(U)]\rangle | \leq \epsilon^{1/3600} \|\phi\|.
  \notag
\end{align}
\end{Lem}

\begin{proof}
By \eqref{eq:Z:Jac} it suffices to find $\Omega_{\epsilon,j,Q}$
satisfying \eqref{eq:good:set:4} such that, on $\Omega_{\epsilon,j,Q}$, assuming \eqref{eq:assu}, it follows that
\begin{align} \label{eq:Z:bound:31}
\sum_{\substack{k\in\ZZZ\\ m,l\in\{0,1\}}}\sup_{t \in [T/2,T]} | \langle \KK_{t, T} \phi,[[Z_j^m(U),\sigma_{k}^l],F(U)]\rangle |
   & \leq \epsilon^{1/2}  \,, \\
  \sum_{\substack{k\in\ZZZ\\ m,l\in\{0,1\}}}\sup_{t \in [T/2,T]} | \langle \KK_{t, T} \phi,[[Z_j^m(U),F(U)],\sigma_{k}^l]\rangle |
  & \leq \epsilon^{1/1800}  \,.
   \label{eq:Z:bound:3}
\end{align}
To obtain \eqref{eq:Z:bound:31}, we have by \eqref{eq:Z:sig:form} and the definition of $Y_j^m(U)$, that if \eqref{eq:assu} holds
(note $j \pm k \in \ZZZ_j$ for $k\in\ZZZ$), then
\begin{align}
| \langle \KK_{t, T} \phi,[[Z_j^m(U),\sigma_{k}^l],F(U)]\rangle |
&= |g(j^\perp \cdot k)|| \langle \KK_{t, T} \phi, [a(j, k)\sigma_{j + k}^{m + l + 1} +
(-1)^{l + 1}b(j, k)\sigma_{j - k}^{m + l + 1}, F(U)]\rangle | \notag \\
&\leq C|j| (| \langle \KK_{t, T} \phi, Y_{j + k}^{m + l + 1}(U) \rangle | +
| \langle \KK_{t, T} \phi, Y_{j - k}^{m + l + 1}(U)\rangle| ) \notag \\
&\leq C|j| \epsilon \|\phi\|
\label{eq:Z:bound:5}
\end{align}
and \eqref{eq:Z:bound:31} follows for any $\epsilon < (C|j|)^{-2} $.

It remains to prove that \eqref{eq:assu} implies \eqref{eq:Z:bound:3} on an appropriate set.  By Lemma \ref{lem:om:to:de:creation:step},
 there exists a set $\Omega^{1}_{\epsilon,j}$ satisfying \eqref{eq:good:set:3}
such that on $\Omega^{1}_{\epsilon,j}$, for each $m\in\{0,1\}$, and each $\phi \in \mathcal{I} = \{ \phi \in H : \| \phi\| =1 \}$
\begin{align}
\sup_{t \in [T/2,T]} | \langle \KK_{t, T} \phi, Y_j^m(U)\rangle | \leq \epsilon
  \quad  \Rightarrow \quad
  \sup_{t \in [T/2,T]} |  \langle \KK_{t, T} \phi, Z_j^m(\bar{U})\rangle |  \leq \epsilon^{1/30}\,.
\label{eq:ZbarF:set1}
\end{align}
For fixed $m \in \{0,1\}$, and each $\phi \in \mathcal{I}$, let $\gphi(t):=  \langle \KK_{t, T} \phi, Z_j^m({\bar{U}})\rangle$
so that $\gphi'(t) = \langle \KK_{t,T}\phi , [Z_j^m({\bar{U}}), F(U)]\rangle$ (see \eqref{eq:gen:der}).
Let
$\Omega^{2}_{\epsilon,j} := \Lambda_{\epsilon^{1/30}, 1/4}^c$, where $\Lambda_{\epsilon, \alpha}$ is defined in \eqref{eq:def:Lambda}.
Thus on $\Omega^{3}_{\epsilon,j}:=\Omega^{1}_{\epsilon,j}\cap \Omega^{2}_{\epsilon,j}$ we have (invoking \eqref{eq:ZbarF:set1})
for each $m\in\{0,1\}$, and $\phi \in \mathcal{I}$
\begin{align}
\sup_{t \in [T/2,T]} |  \langle \KK_{t, T} \phi, Y_j^m(U)\rangle |  \leq \epsilon
 \quad
 \Rightarrow
 \quad
 \sup_{t \in [T/2,T]} |  \langle \KK_{t, T} \phi, [Z_j^m(\bar{U}),F(U)]\rangle | \leq \epsilon^{1/300}.
 \label{eq:ZbarU:F}
\end{align}
Similarly as in the proof of Lemma \ref{lem:om:to:de:creation:step},
\begin{align*}
\Prb ((\Omega^{2}_{\epsilon,j})^c) \leq C \epsilon \E \left( \sup_{\phi \in \mathcal{I}} \|g'\|_{C^{1/4}([T/2, T])}^{8\times 30}\right)
\leq C \epsilon |j|^{240 \times 6} \exp (\eta \|U_0\|^2)\,,
\end{align*}
where the last inequality is analogous to estimates in \eqref{eq:ana:est}.
Next, we establish \eqref{eq:ZbarU:F} with $Z_{j}^{m}(\bar{U})$ replaced by $Z_{j}^{m}(U)$, for which we use the expansion \eqref{eq:ZtoZbar}.  Specifically,
as in \eqref{eq:Z:bound:5}
\begin{align}\label{eq:ZbarU:F:2}
\sup_{t \in [T/2, T]}|\langle \KK_{t, T} \phi, [[Z_j^m(\bar{U}), \sigma_k^l],F(U)]\rangle| |W^{k, l}(t)|\leq
C |j| \epsilon \|\phi\| \sup_{t \in [T/2, T]} |W^{k, l}(t)| \,.
\end{align}
Since $\E \|W^{k, l}\|_{L^\infty} < \infty$,
 Markov inequality yields $\Prb ((\Omega^{4}_{\epsilon,k,l})^c) \leq C \epsilon^{1/2}$, with
 $\Omega^{4}_{\epsilon,k,l} := \{\sup_{t \in [T/2, T]} |W^{k, l}(t)| \leq \epsilon^{-1/2}\}$.
By combining \eqref{eq:ZtoZbar}, \eqref{eq:ZbarU:F}, and \eqref{eq:ZbarU:F:2} on
$\Omega^{5}_{\epsilon,j}:=\Omega^{3}_{\epsilon,j}\cap\Omega^{4}_{\epsilon,j}$ for any $\epsilon < \epsilon_0(T)|j|^{-2}$ it holds that
\begin{align*}
\sum_{\substack{i\in\ZZZ_{j}\\m\in\{0,1\}}}&\sup_{t \in [T/2,T]} | \langle \KK_{t, T} \phi, Y_i^m(U)\rangle | \leq \epsilon
 \notag\\
  &\Rightarrow
\sum_{m\in\{0,1\}}\sup_{t \in [T/2,T]} | \langle \KK_{t, T} \phi, [Z_j^m(U),F(U)]\rangle | \leq \epsilon^{1/600}.
\end{align*}
Similarly as in Lemma \ref{lem:de:to:de:creation:step}, we expand $[Z_j^m(U),F(U)]$
with respect to $U=\bar{U}+\sigma W$ and again we use Theorem 7.1 of \cite{HairerMattingly2011}
to establish
\begin{align*}
&\sup_{t \in [T/2,T]} | \langle \KK_{t, T} \phi, [Z_j^m(U),F(U)]\rangle | \leq \epsilon^{1/600} \|\phi\| \notag \\
&\Rightarrow \sup_{l\in\{0,1\}, k\in\ZZZ}\sup_{t\in[T/2,T]}
| \langle \KK_{t, T} \phi, [[Z_j^m(\bar{U}),F(\bar{U})],\sigma_k^l]\rangle | \leq \epsilon^{1/1800} \|\phi\|
\end{align*}
on a set $\Omega_{\epsilon}^6$ satisfying
\begin{align}
   \Prb( (\Omega_\epsilon^6)^{c}) \leq
   C|j|^{14\times 5400}\exp(\eta \|U_0\|^2)\epsilon.
\label{eq:MH:lem:2}
\end{align}
The proof is finished if we set $\Omega_{\epsilon,j, Q}:= \Omega_{\epsilon,j}^{5}\cap \Omega_{\epsilon,j}^{6}$, and note
by a \eqref{eq:lie:bou} and \eqref{eq:Z:sig:form} that
\begin{align*}
[[Z_j^m(\bar{U}),F(\bar{U})],\sigma_{k}^{l}]= [[Z_j^m(U),F(U)],\sigma_{k}^{l}].
\end{align*}
\end{proof}

\subsection{Spanning Sets for $H^N$ from Brackets and Associated Tails}
\label{sec:it:proof:by:con}

With all elements in Figure~\ref{fig:brak:2} now established, we explain how the lemmata
are pieced together to conclude the proof of Proposition~\ref{prop:small:conclusion}.
To simplify the forthcoming calculations we denote by $\kappa$ the power of $\epsilon$, and $\tau$ the power of $|j|$, appearing in the statement of Lemma
\ref{lem:de:to:de:creation:step}, that is, $\kappa = 1/3600$ and $\tau = 14\times 5400$.
Then assertions of Lemmata \ref{lem:zeroth:step} -- \ref{lem:de:to:de:creation:step} are of the form: for each $\phi \in H$ with $\|\phi\| \leq 1$
and for any sufficiently small $\epsilon$ one has
\begin{align*}
A_j(\phi) \leq \epsilon \quad \Rightarrow \quad B_j(\phi) \leq \epsilon^{\kappa}
\end{align*}
on a set $\Omega_\epsilon$ with $\Prb (\Omega_\epsilon^c) \leq C \epsilon |j|^\tau \exp(\eta \|U_0\|^2)$, where $C =C(T)$ and $A_j$, $B_j$ are appropriate functionals.

Denote
\begin{align*}
\mathcal{I}_N := \{j \in \ZZ^{2}_+ : |j_1| + |j_2| \leq N+1\} \setminus \{(0, N+1), (0, N), (N+1, 0), (N,0)\}
\end{align*}
see Figure~\ref{fig:span:arg:1}.  Note that the choice to `delete' the
corners of the triangular set  $\mathcal{I}_{N}$ is to assure that points in $\mathcal{I}_N \setminus \mathcal{I}_{N-1}$
can be reached from points in $\mathcal{I}_{N-1}$ using only moves depicted in Figure~\ref{fig:new:sig:directions}.
\begin{figure}
  \centering
\vspace{-.55in}
 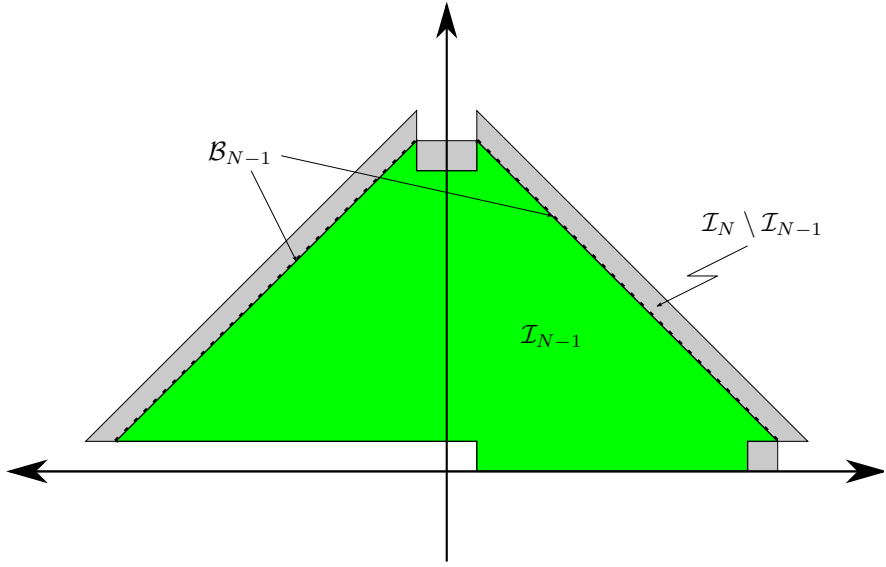
 \vspace{-.35in}
 \caption{An illustration of the sets $I_{N}$ and the associated induction procedure in Lemma~\eqref{lem:small:sum}.}
\label{fig:span:arg:1}
\end{figure}

\begin{Lem}
\label{lem:small:sum}
Let $\kappa$ be as above, and for every $N \geq 0$ denote
\begin{align*}
p_N  := \kappa^{2N + 2} \,.
\end{align*}
Then there exists
\begin{align*}
\epsilon_0 = \epsilon_0(N, T) := C(T) \min \left\{1, \left(\frac{g}{2N^2} \right)^{2/(\kappa p_{N-1})} \right\} \,,
\end{align*}
such that for each
$\epsilon \in (0, \epsilon_0)$, there is a set $\Omega_{\epsilon,N}$ and $C = C(T)$ with
\begin{align*}
  \Prb(\Omega_{\epsilon,N}^c) \leq C N^{\tau + 2} \exp(\eta \|U_0\|^2)\epsilon^{p_N} \,,
\end{align*}
such that on $\Omega_{\epsilon, N}$ for any $\phi \in H$, $j \in \mathcal{I}_{N}$, $m \in\{0,1\}$, one has
\begin{align}
    \langle \MM_{0,T} \phi, \phi \rangle \leq \epsilon \|\phi\|^2& \quad
  \notag\\
 \Rightarrow \quad \sup_{t\in[T/2,T]} &| \langle \KK_{t, T} \phi, \sigma_j^m\rangle| \leq \epsilon^{p_{N} \kappa} \|\phi\|
   \quad \textrm{and}
  \quad
  \sup_{t\in[T/2,T]} | \langle \KK_{t, T} \phi,   Y_j^m(U)\rangle| \leq \epsilon^{p_{N} \kappa} \|\phi\| \,.
  \label{eq:ind}
\end{align}
\end{Lem}

\begin{Rmk}
 Lemma~\ref{lem:small:sum} establishes the smallness of $| \langle \KK_{t, T} \phi, \sigma_j^m\rangle|$
 for all $j \in \mathcal{I}_{N}$.  The smallness of $| \langle \KK_{t, T} \phi, \psi_j^m\rangle|$ requires more work and it is discussed below in Lemma \ref{lem:junk:e}.
\end{Rmk}

\begin{proof}
We first remark that the function $n \mapsto \left(\frac{g}{n^2} \right)^{2/(\kappa p_{n-1})}$ decreases, and therefore
$\epsilon_0 \leq \left(\frac{g}{n^2} \right)^{2/(\kappa p_{n-1})}$ for each $n \leq N$.

We proceed by induction in $N \geq 1$.
For the first step, $N =1$, we show that the result holds on the set $\mathcal{I}_1 = \{(1, 1), (-1, 1)\}$.
To this end we first establish \eqref{eq:ind} for $j \in \ZZZ = \{(0, 1), (1, 0)\}$, which
are the directly forced modes in \eqref{eq:BE:abs}.
Indeed, Lemma \ref{lem:zeroth:step} and Lemma \ref{lem:de:to:om:creation:step} imply that \eqref{eq:ind} holds for each $j \in \ZZZ$, $m \in\{0,1\}$,
with $p_N$ replaced by
$\kappa$, on the set
$$
\Omega^{1}_{\epsilon,1} = \Omega_{\epsilon, \MM} \cap \Omega_{\epsilon^{\kappa},(0,1), Y} \cap \Omega_{\epsilon^{\kappa},(1,0), Y},
$$
with $\Prb((\Omega^{1}_{\epsilon,1})^c) \leq C \exp(\eta \|U_0\|^2) \epsilon^{\kappa}$.

We now establish \eqref{eq:ind} for $j \in \mathcal{I}_1$. By \eqref{eq:ind} with $j' \in \ZZZ$, and Lemma \ref{lem:om:to:de:creation:step}, one has
\begin{align*}
  \langle \MM_{0,T} \phi, \phi \rangle \leq \epsilon \|\phi\|^2 \quad \Rightarrow   \quad
\sup_{\substack{m,l \in\{0,1\} \\ j',k \in \ZZZ}} \sup_{t \in [T/2,T]} |  \langle \KK_{t, T} \phi, [Z_{j'}^m(U),\sigma_{k}^{l}] \rangle | \leq \epsilon^{\kappa^3} \|\phi\|,
\end{align*}
on a set
$$
   \Omega_{\epsilon, 1}^{2} := \Omega^{1}_{\epsilon,1} \cap \Omega_{\epsilon^{\kappa^2},(0,1), \sigma} \cap \Omega_{\epsilon^{\kappa^2},(1,0), \sigma}
$$
with
$$
\Prb((\Omega^{2}_{\epsilon, 1})^c) \leq C \exp(\eta \|U_0\|^2) \epsilon^{\kappa^2}.
$$
Using Proposition \ref{prop:generatesigma} with $j' =  (0,1)$, $k = (1, 0)$, and all combinations of $m, l \in\{0,1\}$, we obtain (since $({j'}^{\perp}\cdot k) ,a(j',k),b(j',k)\neq 0$) that on the set $\Omega_{\epsilon, 1}^{2}$, for each $j \in \mathcal{I}_1$, $m\in\{0,1\}$,
\begin{align*}
  \langle \MM_{0,T} \phi, \phi \rangle \leq \epsilon \|\phi\|^2 \quad \Rightarrow   \quad
\sup_{t\in[T/2,T]} | \langle \KK_{t, T} \phi, \sigma_{j}^m\rangle| \leq \frac{1}{2g}\epsilon^{\kappa^3} \|\phi\|
\leq \epsilon^{\frac{1}{2}\kappa^3}
\|\phi\| \leq \epsilon^{\kappa^4}\|\phi\|.
\end{align*}
Notice we have used the inequality $\epsilon \leq (2g)^{2/\kappa^3}$.
The first part of \eqref{eq:ind} follows with $p_1 = \kappa^3$. The second part of \eqref{eq:ind} with $p_1 = \kappa^4$ follows from the first part and
Lemma \ref{lem:de:to:om:creation:step} on
$$
\Omega_{\epsilon, 1} :=  \Omega_{\epsilon, 1}^{2} \cap \Omega_{\epsilon^{\kappa^4},(-1,1), Y} \cap \Omega_{\epsilon^{\kappa^4},(1,1), Y},
$$
so that $\Prb(\Omega_{\epsilon, 1}^{c}) \leq C \exp(\eta \|U_0\|^2) \epsilon^{\kappa^4}$.
This completes the proof of the base case $N = 1$.

Next we establish the inductive step.  That is, assuming \eqref{eq:ind} holds for each $j\in \mathcal{I}_{N-1}$ (with $N - 1 \geq 1$) on a set $\Omega_{\epsilon,N-1}$, we will show that
\eqref{eq:ind} holds true for $j\in \mathcal{I}_{N}$ on a set $\Omega_{\epsilon,N}$.
We introduce the set
\begin{align*}
  \mathcal{B}_{N-1} := \{ j' = (j_{1}', j_{2}') \in \mathcal{I}_{N-1} : |j_{1}|' + |j_{2}|' = N\},
\end{align*}
which is the `boundary of $\mathcal{I}_{N-1}$ excluding the $x$ and $y$ axes' as illustrated by the broken line segments in Figure~\ref{fig:span:arg:1}.
Denote
\begin{align*}
\Omega_{\epsilon, N}^{1} := \Omega_{\epsilon, N-1} \cap \bigcap_{j' \in \mathcal{B}_{N-1} } \Omega_{\epsilon^{p_{N-1}}, j', \sigma},
\end{align*}
then the  inductive hypothesis and Lemma \ref{lem:om:to:de:creation:step} imply
\begin{align}
\Prb((\Omega_{\epsilon, N}^{1})^c) \leq C\exp(\eta \|U_0\|^2) ((N - 1)^{\tau + 2}  \epsilon^{p_{N - 1}} + |\mathcal{B}_{N-1}|N^{\tau} \epsilon^{p_{N - 1}})
   \leq C N^{\tau+2} \exp(\eta \|U_0\|^2) \epsilon^{p_{N - 1}}.
   \label{eq:prelim:ind:step:large:set}
\end{align}
Then on the set $\Omega_{\epsilon, N}^{1}$ we have
\begin{align}
 \langle \MM_{0,T} \phi, \phi \rangle \leq \epsilon \|\phi\|^{2}  \quad \Rightarrow \quad
\sup_{\substack{l,m \in\{0,1\}\\ j' \in \mathcal{B}_{N-1}, k\in\ZZZ}} \sup_{t \in [T/2,T]} |
\langle \KK_{t, T} \phi, [Z_{j'}^m(U),\sigma_{k}^{l}] \rangle | \leq \epsilon^{p_{N-1}\kappa} \|\phi\|.
\label{eq:name1}
\end{align}

To complete the inductive step it is enough to establish \eqref{eq:ind} for any fixed $j \in \mathcal{I}_N \setminus \mathcal{I}_{N-1}$ and $m \in\{0,1\}$.
We observe that for each $j \in \mathcal{I}_N \setminus \mathcal{I}_{N-1}$, there exists $j' \in \mathcal{B}_{N-1} $ such that $k := j - j' \in \ZZZ \cup (- \ZZZ)$.
In other words, any point in
$\mathcal{I}_{N} \setminus \mathcal{I}_{N-1}$ can be reached from $\mathcal{B}_{N-1}$ via `allowable directions' as shown in Figure~\ref{fig:new:sig:directions}.
Since $k$ is parallel to one of the axes and $j'$ is not, we have $k^\perp \cdot j' \neq 0$ and $a(j', k) \neq 0$, $b(j', k) \neq 0$, where $a$, $b$ are defined by \eqref{eq:sig:coeff:a}.
Using \eqref{eq:name1} and Proposition \ref{prop:generatesigma}, we infer that on the set $\Omega_{\epsilon, N}^{1}$, for each fixed $j \in \mathcal{I}_N \setminus \mathcal{I}_{N-1}$, $m \in\{0,1\}$
(for our choice of $\epsilon \leq \left(\frac{g}{2 N^2} \right)^{2/(\kappa p_{N-1})}$)
\begin{equation}
 \langle \MM_{0,T} \phi, \phi \rangle \leq \epsilon \|\phi\|^{2} \quad \Rightarrow \quad
\sup_{t\in[T/2,T]} | \langle \KK_{t, T} \phi, \sigma_j^m\rangle| \leq \frac{2N^{2}}{g}\epsilon^{p_{N-1}\kappa} \|\phi\| \leq \epsilon^{p_{N-1}\kappa/2} \|\phi\| \,,
\label{eq:prelim:ind:step:imp}
\end{equation}
where we used that $|(k\perp j')| \geq 1$, $|a(k, j')| \geq N^{-2}$, and $|b(k, j')| \geq N^{-2}$.
\footnote{
To see the estimate for $b$ (estimates for $a$ are analogous) we observe that
\begin{align*}
  |b((1, 0), j')| = 1 - \frac{j_{1}^{2}}{j_{1}(j_{1}^{2} + j_{2}^{2})}\geq 1 - \frac{1}{|j_{1}|}, \qquad
  |b((0, 1), j')| =  \frac{|j_{1}|}{(j_{1}^{2} + j_{2}^{2})} \geq \frac{1}{1 + (N-1)^2} \geq \frac{1}{N^2} \,,
\end{align*}
and the desired bound follows when $|j_{1}| \geq 2$.  If $|j_1| = 1$, then $b((1, 0), (\pm 1, N-1)) = 1 - \frac{1}{1 + (N-1)^2} \geq \frac{1}{2}$.
}

To complete the induction it remains to establish the second part of
\eqref{eq:ind}. Define
\begin{align*}
	\Omega_{\epsilon, N} = \Omega_{\epsilon, N}^{1} \cap  \bigcap_{j' \in \mathcal{B}_{N-1} } \Omega_{\epsilon^{p_{N-1}\kappa/2}, j', Y} \,.
\end{align*}
We then obtain (analogously to \eqref{eq:prelim:ind:step:large:set})
that
\begin{align*}
\Prb(\Omega_{\epsilon, N}^c) \leq  C\exp(\eta \|U_0\|^2)(N^{\tau+2} \epsilon^{p_{N - 1}\kappa} +  N^{\tau+1} \epsilon^{p_{N-1}\kappa/2})
\leq CN^{\tau+2} \exp(\eta \|U_0\|^2)\epsilon^{p_{N}} \,.
\end{align*}
On $\Omega_{\epsilon, N}$, \eqref{eq:prelim:ind:step:imp} and Lemma~\ref{lem:de:to:om:creation:step} yield
that \eqref{eq:ind} holds with the desired $p_{N}$.
\end{proof}

\begin{Lem}
\label{lem:junk:e}
Fix $N \geq 2$ and let $p_N$ and $\epsilon_0$ be as in Lemma \ref{lem:small:sum}.
There is
\begin{align*}
\epsilon_1 = \epsilon_1(N, T) := C(T) \min \left\{ 1, \left(\frac{C(T)}{1 + N^4} \right)^{2/(p_{2N}\kappa)} \right\} \,.
\end{align*}
Then
for every $\epsilon \in (0, \epsilon_1)$ and $\eta > 0$ there exists a set $\Omega_{\epsilon,N}^*$ with
\begin{align*}
  \Prb((\Omega_{\epsilon,N}^*)^c) \leq C N^{\tau+3} \exp(\eta \|U_0\|^2)\epsilon^{p_{2N}\kappa}\,,
\end{align*}
where $C = C(\eta, T)$,
such that on the set $\Omega_{\epsilon, N}^*$ for all $\phi \in H$, $j$ with $|j| \leq N$, and $m \in\{0,1\}$, one has
\begin{numcases}{\langle \MM_{0,T} \phi, \phi \rangle \leq \epsilon \|\phi\|^2 \quad \Rightarrow \quad}
       | \langle \phi, \sigma_j^m\rangle| \leq \epsilon^{p_{2N}} \|\phi\| \,, \label{eq:ind:3}  \\
    | \langle \phi,   \psi_j^m + J^N_{j, m}(U(T)) \rangle| \leq \epsilon^{p_{2N}\kappa/2} \|\phi\| \,. \label{eq:ind:4}
\end{numcases}
For the definition of $J^N_{j, m}$ see \eqref{eq:junk:controlled} and \eqref{eq:junk:express}.
\end{Lem}

\begin{proof}
Below, without further notice, we use that $\epsilon_0(2N, T) \geq \epsilon_1(N, T)$ for appropriate $C(T)$, where $\epsilon_0$ is as in Lemma \ref{lem:small:sum}.
Observe that $|j| \leq N$ implies $j \in  \mathcal{I}_{2N}$, because
\begin{align*}
N^2 \geq j_1^2 + j_2^2 \geq \frac{1}{2}(|j_1| + |j_2|)^2 \,.
\end{align*}
Then by Lemma \ref{lem:small:sum}, for each $j$ such that $|j| \leq N$, on the set $\Omega_{\epsilon, 2N}$, \eqref{eq:ind} with $t = T$ implies
\begin{equation}
   | \langle \phi, \sigma_j^m\rangle| \leq \epsilon^{p_{2N}} \|\phi\|  \label{eq:ind:2}
\end{equation}
and \eqref{eq:ind:3} follows.

To establish \eqref{eq:ind:4} we first fix
$|j| \leq N$ with $j_1 \neq 0$ and $m \in\{0, 1\}$. Then, by Lemma \ref{lem:small:sum}, \eqref{eq:junk:express:2},
and \eqref{eq:junk:controlled}, on the set $\Omega_{\epsilon, 2N}$,
\begin{align}
\epsilon^{p_{2N}\kappa} \|\phi\| &\geq |\langle \phi,   Y_j^{m+1}(U (T))\rangle| =
g |j_1||\langle \phi, \psi_j^m + J^N_{j, m}(U(T)) + P_N J_{j, m}(U(T)) \rangle| \notag \\
&\geq g|j_1||\langle \phi, \psi_j^m + J^N_{j, m}(U(T)) \rangle| - g |j_1||\langle \phi, P_N J_{j, m}(U(T)) \rangle|\,.
\label{eq:calc:part:1}
\end{align}
Next, fix $|j| \leq N$ with $j_1 = 0$ and set $j' := \vec{e}_{1} + j$.  It is easy to check that $j'$, $j' \pm  \vec{e}_{2} $, $j' \pm  \vec{e}_{1} $ belong to
$\mathcal{I}_{2N}$ whenever $N \geq 2$, so that, by the second part of \eqref{eq:ind}, \eqref{eq:assu} is satisfied (with $j$ replaced by $j'$) on the set $\Omega_{\epsilon, 2N}$.
Then by Lemma \ref{lem:de:to:de:creation:step} (the smallness conditions on $\epsilon$ required
by Lemma \ref{lem:de:to:de:creation:step} are satisfied if $\epsilon < \epsilon_1$ for appropriate $C(T)$),
\begin{equation}
  \sum_{\substack{k\in\ZZZ\\ m,l\in\{0,1\}}}\sup_{t \in [T/2,T]} | \langle \KK_{t, T} \phi,[Z_{j'}^m(U),Y_{k}^{l}(U)]\rangle | \leq \epsilon^{p_{2N}\kappa} \|\phi\|,
  \label{eq:calc:part:3}
\end{equation}
on the set
\begin{align*}
\Omega_{\epsilon, N}^{*,1} := \Omega_{\epsilon, 2N} \cap \bigcap_{|j'|\leq N, j_1 = 1} \Omega_{\epsilon^{p_{2N}}, j', Q},
\end{align*}
with $\Prb ((\Omega_{\epsilon, N}^{*,1})^c) \leq C (2N)^{\tau +3} \epsilon^{p_{2N}} \exp(\eta \|U_0\|^2)$.
  Then by \eqref{eq:junk:express:2} and \eqref{eq:calc:part:3} (with $t=T$), on $\Omega_{\epsilon, N}^{*,1}$ one has
\begin{equation}
2\epsilon^{p_{2N}} \|\phi\| \geq \frac{g^2 |j|^3}{1 + |j|^2} \left(|\langle \phi, \psi_{j}^m + J^N_{j, m}(U(T))  \rangle|
- |\langle \phi, P_N J_{j, m}(U(T)) \rangle|\right) \,.
\label{eq:calc:part:2}
\end{equation}
Combining both cases, \eqref{eq:calc:part:1} and \eqref{eq:calc:part:2}, one has
for any $|j| \leq N$, on the set $\Omega_{\epsilon, N}^{*,1}$, that
\begin{equation}
|\langle \phi, \psi_{j}^m + J^N_{j, m}(U(T))  \rangle| \leq C (\epsilon^{p_{2N}} \|\phi\| + |\langle \phi, P_N J_{j, m}(U(T)) \rangle|) \,.
\label{eq:first:junk}
\end{equation}
Since by Lemma \ref{lem:junk:est:basic} the first component of $J_{j, m}(U(T))$ vanishes, there exists $(\beta_k^l)_{|k| \leq N, l \in\{0,1 \}}$
such that
\begin{align*}
P_N J_{j, m}(U(T)) = \sum_{\substack{|k| \leq N \\ l \in\{0,1\}}} \beta_k^l \sigma^l_k \,.
\end{align*}
Consequently, by \eqref{eq:ind:2} and Lemma \ref{lem:junk:est:basic}, on the set $\Omega_{\epsilon,2N}$,
\begin{align}
| \langle \phi, P_N J_{j, m}(U(T)) \rangle|
&= | \sum_{\substack{|k| \leq N \\ l \in\{0,1\}}} \beta_k^l \langle \phi,  \sigma^l_k \rangle|
 \leq C \sup_{\substack{|k| \leq N \\ l \in\{0,1\}}}|\langle \phi, \sigma^l_k\rangle| \sum_{\substack{|k| \leq N \\ l \in\{0,1\}}} |\beta_k^l|
  \notag  \\
&\leq C N \|J_{j, m} (U(T))\| \epsilon^{p_{2N}} \|\phi\|
\leq C N^{4} (1 + \|U(T)\|_{H^1}) \epsilon^{p_{2N}} \|\phi\|
 \,.
  \label{eq:second:junk}
\end{align}
If we set $\hat{\Omega}_{\epsilon, N} := \{ 1 + \|U(T)\|_{H^1} \leq \epsilon^{-p_{2N}/2}\}$, then
\eqref{eq:smoothing:est} and the Markov inequality imply
\begin{align*}
\Prb (\hat{\Omega}_{\epsilon, N}^c)  \leq C \epsilon^{p_{2N}/2} \exp(\eta \|U_0\|^2) \,,
\end{align*}
where $C = C(T)$.
Combining \eqref{eq:first:junk} and \eqref{eq:second:junk}, on
$\Omega_{\epsilon, N}^\ast := \hat{\Omega}_{\epsilon, N} \cap \Omega_{\epsilon, N}^{\ast,1}$, it holds that
\begin{align*}
|\langle \phi, \psi_j^m + J^N_{j, m}(U(T))  \rangle| \leq
C  \|\phi\| (\epsilon^{p_{2N}\kappa}  + N^{4} \epsilon^{p_{2N}/2})
\leq
C  \|\phi\| (1  + N^{3}) \epsilon^{p_{2N}\kappa}
\leq \epsilon^{p_{2N} \kappa/2} \|\phi\| \,,
\end{align*}
provided $\epsilon \leq \epsilon_1$. Finally, notice that
\begin{align*}
\Prb ((\Omega_{\epsilon, N}^*)^c) \leq \Prb (\hat{\Omega}_{\epsilon, N}^c) + \Prb(\Omega_{\epsilon, N}^c)
\leq C N^{\tau+3} \exp(\eta \|U_0\|^2) \epsilon^{p_{2N}\kappa} \,.
\end{align*}
\end{proof}

\begin{proof}[Proof of Proposition \ref{prop:small:conclusion}]
By Lemma \ref{lem:junk:e} on a set $\Omega_{\epsilon, \tilde{N}}$  one has
\begin{align*}
\sum_{\tilde{b}(U) \in \mathfrak{B}_{N,\tilde{N}}(U)} |\langle \phi, \tilde{b}(U(T))\rangle|^2
\leq \sum_{\substack{|j| \leq \tilde{N} \\ m \in\{0,1\}}} |\langle \phi, \sigma_j^m\rangle|^2 + |\langle \phi, \psi_j^m + J_{j, m}^{\tilde{N}}(U(T))\rangle|^2
\leq  \tilde{N}^2 \epsilon^{p_{2\tilde{N}}\kappa} \|\phi\|^2
\leq \epsilon^{p_{2\tilde{N}}\kappa /2} \|\phi\|^2 \,,
\end{align*}
whenever $\epsilon \in (0, \epsilon_2)$.
\end{proof}

\section{Mixing and Other Convergence Properties}
\label{sec:mixing}

In this final section we show how the abstract results developed in \cite{HairerMattingly2008, KomorowskiWalczuk2012}
(and cf. \cite{Shirikyan2006,HairerMattinglyScheutzow2011, KuksinShirikian12})
can be applied in our setting to establish mixing and pathwise attraction properties for the unique invariant measure associated to
\eqref{eq:B1:Vort}--\eqref{eq:B3:Vort} and to thus complete the proof of the main result Theorem~\ref{thm:mainresult}.

We begin by introducing some notations.   For any $r \in (0,1]$ and any $\ObsGC> 0$ define
\begin{align}
   \rho_r(U_1, U_2)
   	:= \inf_{\substack{\gamma \in C^1([0,1], H),\\ \gamma(0) = U_1, \gamma(1) = U_2}}
    		\int_0^1 \exp(\ObsGC r \|\gamma(\varrho)\|^2) \| \gamma'(\varrho)\| d\varrho.	
	\label{eq:rho:eta:dist}			
\end{align}
The fixed value of $\ObsGC$ is determined in the course of the proof of Theorem~\ref{thm:mainresult}, (i).
As shown in \cite{HairerMattingly2008},  $\rho_r$ is metric on $H$ for any $r > 0$ and
for any $U_1, U_2 \in H$ one has
\begin{align}\label{eq:rho:metr:est}
  \|U_1 -U_2 \| \leq
   \rho_r(U_1, U_2) \leq
      	\exp(\ObsGC r \max\{\|U_1\|, \|U_2\|\}) (\|U_1- U_2\|).
\end{align}
For brevity of notation we set $\rho := \rho_1$ and for any $\Obs : H \to \RR$ define
\begin{align*}
\|\Obs\|_{Lip} := \sup_{U_1 \not= U_2} \frac{|\Obs(U_1) -\Obs(U_2)|}{\rho(U_1, U_2)}.
\end{align*}
On the other hand, for any $\mu_1, \mu_2 \in Pr(H)$, the set of Borealian probability measures on $H$, denote
\begin{align}
  \mathcal{C}(\mu_1, \mu_2) := \{ \Gamma \in Pr(H\times H): \Gamma(A \times H) = \mu_1(A), \Gamma(H \times A) = \mu_2(A), \textrm{ for any } A \in \mathcal{B}(H) \}\,,
  \label{eq:coupling:def}
\end{align}
where $\mathcal{B}(H)$ is a collection of all Borel subsets of $H$.  Elements $\Gamma \in \mathcal{C}(\mu_1, \mu_2)$
are typically referred to as a \emph{coupling} of $\mu_1$, $\mu_2$.
The distance $\rho$ defined by \eqref{eq:rho:eta:dist} can be used to induce a Wasserstein-Kantorovich distance
(see e.g. \cite{GibbsSu2007, HairerMattingly2008, KuksinShirikian12} for further details) on the set
\begin{align*}
   \mbox{Pr}_1(H) := \left\{ \mu \in \mbox{Pr}(H): \int_H \rho(0, u) d\mu(u) < \infty \right\}.
\end{align*}
This distance is defined by two equivalent formulas
\begin{align}
  \rho( \mu_1, \mu_2) := \sup_{\|\Obs\|_{Lip} \leq 1} \left| \int_H \Obs(U) d \mu_1(U) - \int_H  \Obs(U) d \mu_2(U)\right|
  = \inf_{\Gamma \in \mathcal{C}(\mu_1,\mu_2)}  \int_H \rho(U_1, U_2) d \Gamma(U_1, U_2).
  \label{eq:rho:eta:dist:measures}
\end{align}
While the metric defined in \eqref{eq:rho:eta:dist:measures} is a useful for proving contraction properties of $\{P_t^*\}_{t \geq 0}$, it is less transparent for applications involving
observables on $H$.   As such we consider the class of observables $\ObsSet_\ObsGC$ defined in \eqref{eq:W1infty:space:norm}.
By \cite[Proposition 4.1]{HairerMattingly2008},
$\|\Obs\|_{Lip} \leq C \|\Obs\|_{\ObsGC}$, and therefore $\|\Obs\|_{Lip} < \infty$ for $\Obs \in \ObsSet_\ObsGC$.

Next, we  recall in our setting abstract results from \cite[Theorem 3.4, Theorem 4.5]{HairerMattingly2008}
and from \cite{KomorowskiWalczuk2012} (and cf. \cite{Shirikyan2006,KuksinShirikian12}).

\begin{Thm}[Hairer-Mattingly, \cite{HairerMattingly2008}]
\label{thm:HM:mixing:thm}
Suppose $U = U(t,U_0)$ is a stochastic (semi)flow on a Hilbert space $H$ with a $C^1$
dependence on $U_0 \in  H$.  Define the Markov semigroups $\{P_t\}_{t \geq 0}$, $\{P_t^*\}_{t \geq 0}$
associated to $U(t,U_0)$ as in \eqref{eq:Markov:semigroup}, \eqref{eq:Markov:dual} and  assume that there exists $\ObsGC > 0$
such that
  \begin{itemize}
  \item[(a)] there exists $C > 0$ and a decreasing function $\xi: [0,1] \to [0,1]$ with $\xi(1) < 1$
  such that\footnote{The statement of this result in \cite{HairerMattingly2008} is slightly more general and
  involves the use of Lyapunov functions $V$.  Here we simply set $V(x) = \exp(\ObsGC |x|^2)$.}
  \begin{align*}
    \E  \left( \exp( r \ObsGC \|U(t, U_0)\|^2) ( 1 + \|\nabla_{U_0} U(t, U_0) \|)\right) \leq C \exp( r \ObsGC \xi(t) \|U_0\|^2)
  \end{align*}
  for every $U_0 \in H$, $r \in [1/4, 3]$ and $t \in [0,1]$.
  \item[(b)] a gradient estimate on the Markov semigroup, \eqref{eq:main:grad:est}, holds for $\eta = \ObsGC/2$.
  \item[(c)] given any $\gimel >0$, $r \in (0,1)$, and $\epsilon > 0$, there exists $T^* = T^*(\gimel, r, \epsilon, \ObsGC)$ such that
  for any $T > T*$,
  \begin{align}
     \inf_{ \|U_1\|, \|U_2\|  \leq \gimel} \;
     \sup_{\Gamma \in \mathcal{C}( P_T^* \delta_{U_1},P_T^* \delta_{U_2})}
     \Gamma\{ (U', U'') \in H\times H: \rho_{ r}(U',U'') < \epsilon \} > 0.
     \label{eq:weak:irr:mixing:thm:HM2008}
  \end{align}
  Here, $\delta_{U}$ is the Dirac measure concentrated at $U$ and $\mathcal{C}(\delta_{U_1}, \delta_{U_2})$
  is defined in \eqref{eq:coupling:def}.
  \end{itemize}
  Then there exist positive constants $C, \gamma > 0$  such that
    \begin{align} \label{eq:contr:msr}
     \rho( P_t^* \mu_1, P_t^* \mu_2) \leq C \exp(- \gamma t)  \rho(\mu_1, \mu_2),
  \end{align}
  for every $\mu_1, \mu_2 \in \mbox{Pr}_1(H)$ and every $t > 0$. Moreover, there exists a unique invariant measure $\mu_*$
  (that is $P_t^* \mu_* = \mu_*$ for every $t \geq 0$) and
  \begin{align}
       \left| \E \Obs(U(t,U_0)) - \int_H \Obs(\bar{U}) d \mu_\ast(\bar{U}) \right|
       \leq C \exp(- \gamma t + \ObsGC \|U_0\|^2) \| \Obs - \smallint \Obs d \mu_\ast \|_\eta
       \label{eq:mixing:V:abs:obs}
  \end{align}
  which holds for every $U_0 \in H$ and every $\Obs \in \ObsSet_\ObsGC$ (cf. \eqref{eq:W1infty:space:norm}).
\end{Thm}

\begin{Thm}[Komorowski-Walczuk  \cite{KomorowskiWalczuk2012}]
\label{thm:Pathwise:conv:abs}
  Let $\{P_t\}_{t \geq 0}$ be a Feller Markov semigroup
  on a metric space $(H, \rho)$ with the continuity property:
  $\lim_{t \to 0} P_t \Obs(U_0) = \Obs(U_0)$ for all
  $\Obs \in C_b(H)$, $U_0 \in H$.  Let $P_t(U_0, A)$ be the associated transition
  functions (cf. \eqref{eq:Markov:trans:fn}).  Suppose that for some $C, \gamma > 0$ the contraction property
  \eqref{eq:contr:msr} holds for every $\mu_1, \mu_2 \in \mbox{Pr}_1(H)$.
  Assume moreover that, for every $R > 0$\footnote{The condition \eqref{eq:Polish:style:moment:bounds} given here is
  slightly stronger than the conditions (which appear as $H2$, $H3$) given for the results appearing in \cite{KomorowskiWalczuk2012}.}
    \begin{align}
     \sup_{t \geq 0} \sup_{U_0 \in B_R} \int_H [\rho(0, U)]^3 P_t(U_0, dU)   < \infty\,,
     \label{eq:Polish:style:moment:bounds}
  \end{align}
  where $B_R := \{U_0 \in H : \rho(0, U_0) < R\}$.
  Then, there exists a unique invariant probability measure $\mu_\ast \in \mbox{Pr}_1(H)$ such that, for any $\Obs \in C^1(H)$ and any $U_0 \in H$
 \begin{align*}
 \frac{1}{T}\int_0^T \Obs(U(t,U_0)) \to \int_H \Obs(U) d \mu_*(U) \quad \textrm{ in probability}.
 \end{align*}
 Moreover, the limit $\sigma^2 = \lim_{T \to \infty} \frac{1}{T}\E \left(\int_0^T \left(\Obs(U(t,U_0))\,dt  - \int_H \Obs(U) d \mu_*(U)\right)dt\right)^2$ exists and
 \begin{align*}
   \lim_{T \to \infty} \Prb\left( \frac{1}{\sqrt{T}} \int_0^T \left(\Obs(U(t,U_0))  - \int_H \Obs(U) d \mu_*(U)\right)dt < \xi \right) = \mathcal{X}_\sigma (\xi)\,,
 \end{align*}
 where $\mathcal{X}_\sigma$ is the distribution function of a normal random variable with mean zero and variance $\sigma^2$.
\end{Thm}

Using Theorem~\ref{thm:HM:mixing:thm},\ref{thm:Pathwise:conv:abs}
we now establish the attraction properties (i)--(iii) to complete the proof of Theorem~\ref{thm:mainresult}.

\begin{proof}[Proof of Theorem~\ref{thm:mainresult}]
We begin by establishing the conditions for (a)--(c) of Theorem~\ref{thm:HM:mixing:thm}.
To prove (a), note that for any $\ObsGC, r > 0$, \eqref{eq:lin:path:est} with $\eta = r \ObsGC \frac{\kappa}{4}e^{-\kappa /4}$ implies
we have for any $t \in [0, 1]$
\begin{equation*}
\|\nabla_{U_0} U(t, U_0)\| = \|\JJ_{0, t}\| \leq C\exp \left(r \ObsGC \frac{\kappa}{4}e^{-\kappa /4} \int_0^t \|U(s)\|_{H^1}^2 \, ds \right)
\leq C\exp \left(r \ObsGC \frac{\kappa}{4}e^{-\kappa t /4} \int_0^t \|U(s)\|_{H^1}^2 \, ds \right) \,,
\end{equation*}
where $C = C(r, \ObsGC)$ and $\kappa = \min\{\nu_1, \nu_2 \}$. As such we have established (a) with $\xi(t) := \exp(- t \kappa/2)$ follows  from \eqref{eq:exp:no:time:growth}
and for any  $\ObsGC < \eta^\ast /3$. Set $\ObsGC := \eta^\ast/6$.
Since Proposition~\ref{prop:grad:est:MS} holds for any $\eta > 0$ we
infer the second condition (b) for the given $\ObsGC$.

To establish (c), \eqref{eq:weak:irr:mixing:thm:HM2008},
observe that, for any $\gimel >0$ and $\epsilon > 0$ there exists
$T_* = T_*(\gimel, \epsilon)  \geq 0$ such that
\begin{align}
   \inf_{\|U_{0}\| \leq \gimel}  P_{T}(U_{0}, \{ U \in H; \|U \| \leq \epsilon\}) > 0,
   \label{eq:non:zero:return}
\end{align}
for every $T > T_*$. A detailed proof of \eqref{eq:non:zero:return} which applies in our setting
can be found in \cite{EMattingly2001,ConstantinGlattHoltzVicol2013}, here we just briefly sketch the essential ideas.
If there is no forcing, that is, if there is no Brownian motion, then by the dissipativity of \eqref{eq:B1:Vort}--\eqref{eq:B3:Vort},
there is a $T^\ast = T^\ast(\|U_0\|, \epsilon)$ such that $\|U(T, U_0)\|< \epsilon/2$. However, for any open
ball $B(\delta) := \{x \in \RR^{2|\ZZZ|}: |x| < \delta\}$ there is a non-zero probability that the Brownian motion remains
in $B(\delta)$ over the whole interval $[0,T]$. Then by the continuous dependence of solutions with respect to external forcing we
conclude \eqref{eq:non:zero:return} for sufficiently small $\delta > 0$. More precisely
we can use the change of variable $\bar{U} = U - \sigma W$ and standard estimates to show that $\|U(T, U_0)\|< \epsilon$.

We now establish \eqref{eq:weak:irr:mixing:thm:HM2008} from \eqref{eq:non:zero:return}
as follows.  For $U_1, U_2 \in H$ and $T>0$ define $\tilde{\Gamma}_{U_1, U_2} \in Pr(H \times H)$ by
\begin{align*}
   \tilde{\Gamma}_{U_1, U_2}(A _1\times A_2) := P_T(U_1, A_1)P_T(U_2, A_2) \textrm{ for any } A_1, A_2 \in \mathcal{B}(H).
\end{align*}
From \eqref{eq:coupling:def} it
follows that $\tilde{\Gamma}_{U_1, U_2} \in \mathcal{C}(P_T^*\delta_{U_1}, P_T^*\delta_{U_2})$, that is,  $\tilde{\Gamma}_{U_1, U_2}$
couples $P_T^*\delta_{U_1}$ and $P_T^*\delta_{U_2}$.
As such, using \eqref{eq:rho:metr:est} and \eqref{eq:non:zero:return},  we infer that, for any $\gimel, \epsilon > 0$, $r \in (0,1)$, and any
$T > T^\ast(\gimel, \min \{\epsilon\exp(- 3\ObsGC)/2, 1 \})$
\begin{align}
     \inf_{ \|U_1\|, \|U_2\|  \leq \gimel} \;&
     \sup_{\Gamma \in \mathcal{C}( P_T^* \delta_{U_1},P_T^* \delta_{U_2})}
     \Gamma\{ (U', U'') \in H\times H: \rho_{ r}(U',U'') < \epsilon \}  \notag\\
     &\geq     \inf_{ \|U_1\|, \|U_2\|  \leq \gimel} \;
       \tilde{\Gamma}_{U_1, U_2} \left\{ (U', U'') \in B_1\times B_1: \| U'\| + \|U''\| < \epsilon \exp(-\ObsGC r)  \right\}  \notag\\
      &\geq
    \left(\inf_{ \|U_1\|  \leq \gimel}  P_T(U_1,  \left\{ U' \in H:  \|U'\|  < \min(\epsilon/2 \cdot \exp(-3\eta), 1 ) \right\}) \right)^2
    > 0 \,,
    \notag
\end{align}
where $B_1 := \{U \in H: \|U\| < 1\}$ and \eqref{eq:weak:irr:mixing:thm:HM2008} follows.

Having established the conditions (a)--(c) in Theorem~\ref{thm:HM:mixing:thm}  we infer
\eqref{eq:contr:msr} and \eqref{eq:mixing:V:abs:obs} for \eqref{eq:B1:Vort}--\eqref{eq:B3:Vort}.
To prove \eqref{eq:main:thm:mixing:cond}
it suffices to show that $| \smallint \Obs d\mu_\ast | \leq C \|\Obs \|_{\ObsGC}$, where $\mu_\ast$
is the unique invariant measure of \eqref{eq:B1:Vort}--\eqref{eq:B3:Vort}.
However,
\begin{align*}
   \left|  \int \Obs(U) d\mu_\ast(U) \right|
      \leq&   \int \exp(\ObsGC \|U\|^2) \exp(- \ObsGC \| U\|^2) |\Obs(U)| d\mu_\ast(U) \\
   \leq&  \|\Obs\|_{\ObsGC} \int_{H} \exp(\ObsGC \|U\|^{2} )d \mu_\ast(U)\,,
\end{align*}
and therefore it suffices to show
\begin{align}
   \int_{H} \exp(\ObsGC \|U\|^{2} )d \mu_\ast(U) \leq C< \infty,
   \label{eq:exp:mom:bnd:invar:measure}
\end{align}
where $\ObsGC = \eta^*/6$. For any $R >0$, define
\begin{align*}
  \Obs_R(U) =
  \begin{cases}
  \exp( \ObsGC \|U\|^2)& \textrm{ for } \|U\| < R, \\
  \exp( \ObsGC R^2)&  \textrm{ for } \|U\| \geq R, \\
  \end{cases}
\end{align*}
and note that $\Obs_R \in C_b(H)$.  Now using that $P_t^* \mu_\ast = \mu_\ast$ and \eqref{eq:exp:no:time:growth}
we have for any $T > 0$, $\gimel > 0$
\begin{align*}
   \left| \int \Obs_R (U) d \mu_\ast(U) \right|
	\leq&
	\left| \int_{B_\gimel(H)} P_T\Obs_R (U) d \mu_\ast(U)\right|
	+ \left| \int_{B_\gimel(H)^c} P_T \Obs_R (U) d \mu_\ast(U)\right|\\
	\leq& \left| \int_{B_\gimel (H)} \E \exp(\ObsGC \|U(t, U_0)\|^2) d \mu_\ast(U_0)\right|
	+\exp( \ObsGC R^2) \mu_\ast(H \setminus B_\gimel(H))\\
	\leq& C\left| \int_{B_\gimel(H)} \exp\left( {\ObsGC e^{-\frac{T}{2}\min \{\nu_1, \nu_2\}}}  \|U_0\|^2 \right) d \mu_\ast(U_0)\right|
	+\exp( \ObsGC R^2) \mu_\ast(H \setminus B_\gimel(H))\\
		\leq& C \exp\left( {\ObsGC e^{-\frac{T}{2}\min \{\nu_1, \nu_2\}}}  \gimel^2 \right)
	+\exp( \ObsGC R^2) \mu_\ast(H \setminus B_\gimel(H))\,.
\end{align*}
Now since $T \geq 0$ is arbitrary we infer that for $\gimel, R > 0$
that
\begin{align*}
    \left| \int \Obs_R (U) d \mu_\ast(U) \right| \leq C  + \exp( \ObsGC R^2) \mu_\ast(H \setminus B_\gimel(H)) \,,
\end{align*}
where $C$ is independent of $\gimel$.  Passing $\gimel \to \infty$ and then $R \to \infty$ and using the monotone convergence theorem we conclude
\eqref{eq:exp:mom:bnd:invar:measure} and  \eqref{eq:main:thm:mixing:cond} follows.

The remaining convergence properties \eqref{eq:WLLN}, \eqref{eq:CLT} follow once we show that the conditions for
Theorem~\ref{thm:Pathwise:conv:abs} are met.  The Feller property and stochastic continuity of $P_t$
follow immediately from the well-posedness properties of \eqref{eq:B1:Vort}--\eqref{eq:B3:Vort} as recalled above in Proposition~\ref{prop:wellposed}.
It remains to verify the bound in \eqref{eq:Polish:style:moment:bounds}.

By \eqref{eq:rho:metr:est} and \eqref{eq:exp:no:time:growth} we have for already fixed $\ObsGC = \eta^\ast/6$ and
for any $U_0 \in H$ and any $t \geq 0$ that
\begin{align*}
\int [\rho (0,U)]^3 P_t(U_0, dU)
&\leq \int \exp(3 \ObsGC  \|U\|^2) \|U\|^3 P_t(U_0, dU) \leq
C \int \exp(\eta^\ast  \|U\|^2) P_t(U_0, dU) \\
&= C \E (\exp (\eta^\ast \|U(t, U_0)\|^2)) \leq C \exp (\eta^* \|U_0\|^2)\,,
\end{align*}
where the constant $C$ is independent of $U_0$ and $t \geq 0$.
Thus, in view of \eqref{eq:rho:metr:est}, we infer the bound \eqref{eq:Polish:style:moment:bounds} and hence the convergences \eqref{eq:WLLN}, \eqref{eq:CLT}.  This
completes the proof of Theorem~\ref{thm:mainresult}.
\end{proof}

\appendix

\section{Appendix}

Section~\ref{sec:moment:est} collects various moment estimates for \eqref{eq:B1:Vort}--\eqref{eq:B3:Vort}.
In Section~\ref{subsect:malliavin} we briefly review of some aspects of the  Malliavin calculus relevant to our
analysis above.

\subsection{Moment Estimates}	
\label{sec:moment:est}
In this section we provide details for the moments bound used throughout the manuscript.
As above the dependence on physical parameter in constants is suppressed in what follows, see Remark~\ref{rmk:constants:conven}.
Denote
\begin{align*}
\zeta^\ast := \frac{\nu_1 \nu_2}{g^2} \,.
\end{align*}
 Then, cf. \eqref{eq:norm:def},
\begin{equation}\label{eq:def:U:2}
\|U\|^2 = \zeta^\ast \|\omega\|_{L^2}^2 + \|\de\|_{L^2}^2, \qquad
\|U\|^2_{H^1} = \zeta^\ast \|\nabla \omega\|_{L^2}^2 + \|\nabla \de\|_{L^2}^2,
\,.
\end{equation}
Also recall, that our domain is $\TT = \RR^2/(2\pi \ZZ^2)$, and therefore
the Poincar\` e inequality takes the form
$\|U\|_{H} \leq \|U\|_{H^1}$.

Most of the forthcoming bounds have previously been obtained in the context of the stochastic
Navier-Stokes equations and some other nonlinear SPDEs with a dissipative (parabolic) structure.
In order to modify them for the Boussinesq system, we need to
compensate for the `buoyancy' term $g \pd_x\de$ when carrying out energy estimates.
This is accomplished by differently weighting the temperature and momentum equations.\footnote{As noted in the introduction,
we are considering the so called `HRB approximation' in which
the Boussinesq equation is considered
with periodic boundary conditions after subtracting off the temperature differential profile, see \cite{Calzavarini2006}.
In our setup the temperature differential is zero, and therefore the
dissipativity properties we derived here do not contradict the situation illuminated in \cite{Calzavarini2006};
we can exclude the possibility of `grow-up' solutions.}
We illustrate this strategy in the proof of \eqref{eq:exp:no:time:growth}; proofs of other estimates
use the same approach in combination with a straightforward modification of existing methods for the (stochastic)
Navier-Stokes equation (see e.g. \cite{HairerMattingly06, Debussche2011a, KuksinShirikian12})
and they are omitted.

In the first lemma we state a priori bounds on $U$.  These estimates reflect
parabolic type regularization properties of \eqref{eq:B1:Vort}--\eqref{eq:B3:Vort}, and
are particularly useful for obtaining spectral bounds
on the Malliavin matrix carried out in Section~\ref{sec:Mal:spec:bnds}--\ref{sec:Braket:Est:Mal:Mat}.

\begin{Lem}
\label{lem:exp:moments}
Fix any $U_{0} \in H$ and let $U(\cdot) = U(\cdot, U_{0})$ be the unique
solution of \eqref{eq:B1:Vort}--\eqref{eq:B3:Vort} with $U(0) = U_{0}$.
Denote $\kappa := \min \{\nu_1, \nu_2 \}$. There exists $\eta^\ast > 0$ such that:
\begin{itemize}
\item[(i)] For any $T>0$ and $\eta \in (0, \eta^\ast]$,
  \begin{align}
     \E  \exp\left(  \eta \|U(T)\|^2 + \eta \frac{\kappa}{4} e^{-\kappa T/4} \int_0^T \|U(t)\|^2_{H^1} dt \right) \leq
     C \exp\left( {\eta e^{-\kappa T/2}}  \|U_0\|^2 \right) \,,
     \label{eq:exp:no:time:growth} \\
     \E  \exp\left( \sup_{\tau \in [0, T]} \eta \|U(\tau)\|^2 + \eta\frac{\kappa}{2} \int_0^T  \|U(t)\|^2_{H^1 }dt  \right)
     \leq  C\exp\left( \eta \|U_0\|^2 + \eta \|\sigma_\de\|^2 T\right)\,,
     \label{eq:exp:time:int}
\end{align}
for a constant $C$ independent of $T$.
\item[(ii)]  For any $N > 0$ and  $\eta \in (0, \eta^*]$,
\begin{align}
\E \exp \left( \eta \sum_{k = 0}^N  \|U(k)\|^2 \right) \leq
\exp \left(\varrho \eta \|U_0\|^2 + \varkappa N \right) \,,
\label{eq:exp:mom:sums}
\end{align}
where $\varrho, \varkappa > 0$ are positive constants independent of $N$ and $U_0$.
\item[(iii)] For any $s \geq 0$, $p \geq 2$, and $\eta \in (0, \eta^*]$ there exists $C =
C(\eta, s, T, p)$ such that
\begin{align}
  \E  \left(\sup_{t \in [T/2, T]}  \|U(t)\|_{H^s}^p  \right)  \leq C
  \exp( \eta \|U_0\|^2)\,.
  \label{eq:smoothing:est}
\end{align}
\item[(iv)]For
any $p \geq 2$, $s \geq 0$, $\eta > 0$, and $T > 0$ there is
$C = C(s, T, p, \eta)$ such that
\begin{align}
  \E  \left(\|U\|_{C^{\frac{1}{4}}([T/2,T], H^s)}^p  \right)  \leq C \exp( \eta \|U_0\|^2)\,.
  \label{eq:smoothing:est:2}
\end{align}
\end{itemize}

\end{Lem}

\begin{proof}
For any $f \in H^1$ denote
\begin{align*}
   \|f\|^{2}_{D(A^{1/2})} :=  \zeta^\ast \nu \|\nabla f_1\|^2_{L^2} +
\mu \|\nabla f_2\|^{2}_{L^2} \,.
\end{align*}
From \eqref{eq:B1:Vort}--\eqref{eq:B3:Vort} we have
\begin{align}\label{eq:est:11}
  & d \|\om\|^{2}_{L^2} + 2 \nu_1 \| \nabla \om \|^{2}_{L^2}dt = 2 g\langle \pd_{x} \de, \om \rangle dt ,
  \\\label{eq:est:12}
  &d \|\de\|^{2}_{L^2} + 2 \nu_2 \| \nabla \de \|^{2}_{L^2} dt= \|\sigma_{\theta}\|^{2}dt
  + 2 \langle \sigma_{\theta}, \de \rangle dW \,.
\end{align}
Now we weight differently the equations (note the difference to the Navier-Stokes equation).
Multiplying \eqref{eq:est:11} by $\zeta^\ast$ and adding to \eqref{eq:est:12} we obtain (recall \eqref{eq:def:U:2})
for any $\eta > 0$
\begin{align}\label{eq:est:h11}
d (\eta \|U\|^2) = \eta(2 g\zeta^\ast\langle \pd_{x} \de, \om \rangle +  \|\sigma_{\theta}\|^{2} -
2 \|U\|^{2}_{D(A^{1/2})})dt + 2\eta \langle \sigma_{\theta}, \de \rangle dW \,.
\end{align}
Since by Poincar\` e inequality
\begin{align}\label{eq:est:h12}
2 \zeta^\ast g|\langle \pd_{x} \de, \om \rangle|
&\leq  \nu_1 \zeta^\ast \|\omega\|^2_{L^2} + \nu_2 \|\nabla \de \|^2_{L^2}\leq \nu_1 \zeta^\ast \|\nabla \omega\|^2_{L^2} +
\nu_2 \|\nabla \de \|^2_{L^2}
= \|U\|_{D(A^{1/2})}^2\,, \\
\kappa \|U\|^2 &\leq \|U\|^{2}_{D(A^{1/2})}\,,\notag
\end{align}
we have for $Z(t) := \frac{\eta}{\kappa} \|U(t)\|^2_{D(A^{1/2})}$ and $V(t) := \eta \|U(t)\|^2$ that $V \leq Z$ and
\begin{align*}
 \eta(2 g\zeta^\ast\langle \pd_{x} \de, \om \rangle +  \|\sigma_{\theta}\|^{2} -
2 \|U\|^{2}_{D(A^{1/2})}) &\leq \eta \|\sigma_{\theta}\|^{2} - \kappa Z \,, \\
4 \eta^2 |\langle \sigma_{\theta}, \de \rangle|^2\leq 4 \eta^2 \|\sigma_\de\|^2 \|\de\|^2_{L^2}
\leq 4 \eta^2 \|\sigma_\de\|^2 \|U\|^2
&\leq 4\eta \|\sigma_\de\|^2 Z \,.
\end{align*}
Thus, by \cite[Lemma 5.1]{HairerMattingly2008}, for any $\eta \in (0, \kappa/(4\|\sigma_\de\|^2))$ one has
\begin{equation*}
\E \exp \left(\eta \|U(T)\|^2 + \frac{\kappa e^{-\kappa T/4}}{4} \int_0^T Z(s)\, ds\right) \leq C \exp \left(\eta e^{-\kappa T/2}\|U_0\|^2 \right)\,.
\end{equation*}
Now, \eqref{eq:exp:no:time:growth} follows from
$\kappa Z(s) = \eta \|U(s)\|^2_{D(A^{1/2})} \geq \eta \kappa \|U(s)\|^2_{H^1} $.

The estimates \eqref{eq:exp:time:int}, \eqref{eq:exp:mom:sums} follow similarly as
in \cite[proof of Lemma 4.10]{HairerMattingly06} and \eqref{eq:smoothing:est} follows
as in \cite[Proposition 2.4.12]{KuksinShirikian12}. Finally, \eqref{eq:smoothing:est:2} follows from
\eqref{eq:smoothing:est} and the fact that $\|W_{k}^l\|_{C^{1/4}_{[T/2,T]}}$
has finite $p$-th moment for any $p\geq 1$, see \cite{BaldiBenArousKerkyacharian1992}.
\end{proof}

\begin{Rmk}\label{rmk:stupid:rho}
In the following estimates $\eta$ appears only on the right hand side, and therefore they remain valid if $\eta$
is increased, thus we do not assume any upper bound on $\eta$.
\end{Rmk}

Next lemmata collect estimates on linearizations of \eqref{eq:B1:Vort}--\eqref{eq:B3:Vort}.
Recall
the definitions of the operators $\JJ_{0,t}$ and its adjoint $\KK_{0,t}$ given
in \eqref{eq:def:J} and \eqref{eq:def:K:backwards:eq} respectively. Moreover, for any $t \geq s \ge 0$
let $\JJ^{(2)}_{s,t} : H \to \mathcal{L}(H, \mathcal{L}(H))$ be the second derivative of $U$ with respect to an initial condition $U_0$.
Observe that for fixed $U_0$ and any $\xi, \xi' \in H$ the function
$\rho(t) := \JJ^{(2)}_{s,t}(\xi, \xi')$ is the solution of
\begin{align}
\partial_t \rho + A \rho + \nabla B(U) \rho + \nabla B(\JJ_{s, t}\xi) \JJ_{s, t}\xi'
= G \rho, \qquad \rho(s) &= 0. \label{eq:def:j1}
\end{align}

\begin{Lem}\label{lem:lin:moments}
For each $\eta > 0$ and $0 < s < t$, we have the pathwise estimate
\begin{align}
   \| \JJ_{s,t} \| \leq  \exp\left( \eta \int_s^t
   \|U(s)\|_{H^1}^2 ds  + C (t-s)\right),
   \label{eq:lin:path:est}
\end{align}
where $C = C(\eta)$ is independent of $s, t$.
Moreover, for each $\tau \leq T$, $p \geq 1$, and $\eta > 0$
there is $C = C(\eta, T - \tau, p)$ such that
\begin{align}
     \E \sup_{s < t \in [\tau,T]} \| \JJ_{s,t} \|^p  &\leq
     C \exp \left( \eta \|U_0\|^2\right),
     \label{eq:lin:growth:est}   \\
          \E \sup_{s< t \in [\tau, T]} \|\KK_{s, t} \|^p  &\leq C
      \exp( \eta \|U_0\|^2),
   \label{eq:back:lin:est:1}\\
        \E \sup_{s<t \in [\tau, T]} \| \JJ^{(2)}_{s ,t} \|^p &\leq  C \exp \left( \eta \|U_0\|^2 \right)\,.
     \label{eq:second:esti}
\end{align}
\end{Lem}

\begin{proof}
Proof of \eqref{eq:lin:path:est} follows along the same lines as \cite[Lemma 4.10.3]{HairerMattingly06}.
By taking expectation, \eqref{eq:lin:growth:est} follows from \eqref{eq:lin:path:est},
\eqref{eq:exp:no:time:growth}, and \eqref{eq:exp:time:int}.
Finally, \eqref{eq:back:lin:est:1} follows from \eqref{eq:lin:growth:est} by duality and \eqref{eq:second:esti}
is similar to \cite[Lemma 4.10.4]{HairerMattingly06}.
\end{proof}

The next lemma provides us with estimates to initial time in a weak norm, which allows us to avoid
some
technical arguments in Section \ref{sec:implications:eigenvals} (cf.  \cite{HairerMattingly2011}).

\begin{Lem}
For any $p \geq 2$, $T > 0$, and $\eta > 0$ there is
$C = C(p, T, \eta)$ such that
\begin{align}
   \E \sup_{t \in [T/2,T]} \|\pd_{t} \KK_{t, T} \xi \|^p_{H^{-2}}
   \leq C \exp \left(  \eta  \|U_0\|^2 \right) \|\xi\|^{p} \,.
   \label{eq:lin:smooth}
\end{align}
\end{Lem}

\begin{proof}
Recall that $\rho^* = \KK_{t,T}\xi$ solves \eqref{eq:def:K:backwards:eq} and notice that
$\| B(U', U'') \|  \leq \| U'\|_{H^1} \|U''\|_{H^1}$
for any $U', U'' \in H^1$.
Since $\|A \rho^\ast\|_{H^{-2}} \leq \|\rho^\ast\|$ and
\begin{align*}
   \| (\nabla B(U))^* \rho^* - (\nabla G(U))^* \rho^*\|_{H^{-2}} \leq& \sup_{\|\psi\|_{H^2} \leq 1}(|\langle(\nabla B(U))^* \rho^*, \psi \rangle|
   + |\langle(\nabla G(U))^* \rho^*, \psi\rangle|)
   \notag\\
	\leq& \sup_{\|\psi\|_{H^2} \leq 1}\left( |\langle \rho^*, B(U,\psi) \rangle|  + | \langle \rho^*,B(\psi, U) \rangle| +
	|\langle \rho^*, \nabla G(U)\psi\rangle|\right)
	\notag\\
	\leq& \| \rho^*\|\sup_{\|\psi\|_{H^2} \leq 1}\left(2 \|U\|_{H^1} \| \psi \|_{H^1} + |g| \|\psi\|_{H^1}\right)\notag \\
	\leq& C \| \rho^*\| ( \|U \|_{H^1} + 1) \,,
\end{align*}
then \eqref{eq:lin:smooth} follows from \eqref{eq:back:lin:est:1} and  \eqref{eq:smoothing:est}.
\end{proof}

The next lemma is a version of the Foias-Prodi estimate, \cite{FoiasProdi1967},  used in this work.  Specifically
the estimate \eqref{eq:lin:growth:high:mode:est} is employed in the decay estimate
\eqref{eq:stepwise:cond:decay:est}.

\begin{Lem}
\label{lem:parabolic:cont:high:modes:Jop}
For every $p \geq 1$, $T > 0$, $\delta, \gamma > 0$ there exists
$N_* = N_*(p, T, \delta, \gamma)$, such that for any $N \geq N_*$ one has
(recall that $Q_N$ was defined in \eqref{eq:H:N:proj:ops})
\begin{align}
 	 \E \| Q_N\JJ_{0,T} \|^p  &\leq \gamma
 	 \exp \left(  \delta \|U_0\|^2 \right) \,,
\notag
 	 \\
     \E \| \JJ_{0,T} Q_N\|^p  &\leq \gamma \exp \left(  \delta \|U_0\|^2 \right) \,.
     \label{eq:lin:growth:high:mode:est}
\end{align}
\end{Lem}

\begin{proof}
The proof is analogous to \cite[Lemma 4.17]{HairerMattingly06} and it essentially follows from the fact that $A$
has compact resolvent (see \cite[proof of Theorem 8.1]{HairerMattingly2011}).
\end{proof}

We next present estimates on the operators $\AAA_{s,t}$ and the
inverse of the regularized Malliavin matrix that are primarily used in Section~\ref{sec:proof:main:prop}.

\begin{Lem}\label{lem:claim}
For $0 < s < t$ define $\AAA_{s,t}$, $\AAA_{s,t}^*$, and $\MM_{s,t}$ by to \eqref{eq:A:op:def}, \eqref{eq:A:star:op:def},
and \eqref{eq:Mal:def} respectively.  Then
\begin{align}
   \|\AAA_{s,t}\|_{\mathcal{L}(L^2([s, t]; \RR^d), H)}  \leq C \left( \int_s^t \| \JJ_{r,t} \|^2 dr\right)^{1/2}
   \label{eq:gen:AAA:bnd}
\end{align}
for a constant $C$ independent of $s, t$.
Moreover,
\begin{align}
\|\AAA_{s,t}^\ast (\MM_{s,t}  + I \beta)^{-1/2}\|_{\mathcal{L}(H, L^2([s, t]; \RR^d))} &\leq 1 \,, \label{eq:op:norm:1} \\
\|(\MM_{s,t}  + I \beta)^{-1/2} \AAA_{s,t} \|_{\mathcal{L}(L^2([s, t]; \RR^d), H)}
&\leq 1 \,, \label{eq:op:norm:2}\\
\|(\MM_{s,t}  + I \beta)^{-1/2} \|_{\mathcal{L}(H, H)} &\leq \beta^{-1/2} \,.
\label{eq:op:norm:3}
\end{align}
Here, $\mathcal{L}(X, Y)$ denotes the operator norm of the linear map between the given Hilbert spaces
$X$ and $Y$.
\end{Lem}
\begin{proof}
The first bound \eqref{eq:gen:AAA:bnd} follows from the definition of $\AAA_{s,t}$ and H\"older's inequality.
Next, since $\MM_{s, t}$ is self-adjoint
\begin{align*}
\|(\MM_{s,t}  + I \beta)^{1/2}U\|_{H}^2 &=
\langle (\MM_{s,t}  + I \beta)^{1/2}U, (\MM_{s,t}  + I \beta)^{1/2}U \rangle
=\langle (\MM_{s,t}  + I \beta)U, U \rangle \\
&=
\beta \|U\|_H^2 + \langle \AAA_{s,t} \AAA_{s,t}^\ast U, U \rangle
=
\beta \|U\|_H^2 + \|\AAA_{s,t}^\ast U \|^2_{L^2([s, t]; \RR^d)}
\end{align*}
for any $U \in H$. Setting $U = (\MM_{s,t}  + I \beta)^{-1/2} V$ ($(\MM_{s,t}  + I \beta)^{-1/2}$ is invertible)
we immediately obtain \eqref{eq:op:norm:1} and \eqref{eq:op:norm:3}.
Finally,  \eqref{eq:op:norm:2} follows from \eqref{eq:op:norm:1} by duality.
\end{proof}

For the `cost of control' bounds \eqref{eq:bound} on $v$ is Section~\ref{sec:proof:main:prop} we also made
use of bounds on the Malliavin derivative of the random operators $\JJ_{s, t}$, $\AAA_{s, t}$, and $\AAA^*_{s, t}$
for $0 < s < t$.  Observe that for $\tau \leq t$ (see \cite{HairerMattingly2011})
\begin{align}
  \MD_\tau^j \JJ_{s,t} \xi =
  \begin{cases}
    \JJ^{(2)}_{\tau, t}( \sigma_\theta e_j, \JJ_{s,\tau} \xi)&   \textrm{ if } s \leq \tau \,,\\
    \JJ^{(2)}_{s, t}( \JJ_{\tau, s} \sigma_\theta e_j,  \xi) & \textrm{ if } s > \tau \,.
  \end{cases}
  \label{eq:Mal:Div:Snd}
\end{align}
We refer the reader to Appendix~\ref{subsect:malliavin} for further details on the Malliavin derivative operator
$\MD$ and the associated spaces $\Mspc^{1,p}$ on which it acts.
\begin{Lem}\label{lem:mal:est}
For any $0 \leq s < t$ the random operators $\JJ_{s,t}$, $\AAA_{s,t}$, $\AAA_{s,t}^\ast$ are differentiable in the
Malliavin sense.   Moreover, for any $\eta > 0$ and $p \geq 1$ we have the bounds
\begin{align}
	\E \|\MD_{\tau}^j \JJ_{s, t}\|^p &\leq C \exp (\eta \|U_0\|^2) \,,
		\label{eq:mal:J:est}\\
	\E \|\MD_{\tau}^j \AAA_{s, t}\|^p_{\mathcal{L}(L^2([s, t], \RR^d),H)} &\leq C \exp (\eta \|U_0\|^2) \,,
		\label{eq:mal:A:est}  \\
	\E \|\MD_{\tau}^j \AAA_{s, t}^\ast\|^p_{\mathcal{L}(H,L^2([s, t], \RR^d))} &\leq C \exp (\eta \|U_0\|^2)\,,
		\label{eq:mal:A:star:est}
\end{align}
where $C = C(\eta, p)$.
\end{Lem}

\begin{proof}
The proof of \eqref{eq:mal:J:est}--\eqref{eq:mal:A:star:est} is based on the observation \eqref{eq:Mal:Div:Snd} and the bound \eqref{eq:second:esti}.
Further details can be found in \cite{HairerMattingly06}.
\end{proof}

\subsection{Some Elements of the Malliavin Calculus}
\label{subsect:malliavin}

In this section, we recall in our context and notations some elements of the Malliavin
calculus used in above. For further general background on this vast subject see,
for example, \cite{Bell1987,Malliavin1997,Nualart2006, Nualart2009}.

Fix a stochastic basis $(\Omega, \mathcal{F}, \{\mathcal{F}_t\}_{t \geq 0}, \Prb, W)$, where
$W = (W_1, \dots, W_d)$ is a $d$-dimensional standard Brownian motion, $\{\mathcal{F}_t\}_{t \geq 0}$ is a
filtration to which this process $W$ is adapted.  In application to \eqref{eq:B1:Vort}--\eqref{eq:B3:Vort},
$d =2|\mathcal{Z}|$ represents the number of independent noise processes driving the system.
Fix any $T \in (0, \infty)$.

We first recall the definition of the \emph{Malliavin derivative} $\MD$ which is defined on a
subset of $L^p(\Omega)$ for $p > 1$ (see \cite{Nualart2009,Nualart2006} or \cite{Malliavin1997}).  We begin by explaining how
this operator $\MD$ acts on `smooth random variables'.
For any given $n \geq 1$, consider a Schwartz function $f : \RR^n \to \RR$, that is, $f$ that satisfies
\begin{align*}
 \sup_{z \in \RR^n} |z^\alpha D^\beta f(z)| < \infty
\end{align*}
for any multi-indices $\alpha, \beta$.
For such functions define  $F \in L^p(\Omega)$, $p >1$ by
\begin{align*}
  F = f \left(\int_0^{T} g_1 \cdot dW, \ldots,  \int_0^{T} g_n \cdot dW \right) \,,
\end{align*}
where  $g_1, \cdots, g_n$ are deterministic elements in $L^2([0,T], \RR^d)$.
For such $F$ the Malliavin derivative is defined as
\begin{align}
 \MD F :=
 \sum_{k =1}^n \frac{\pd f}{\pd x_k} \left(\int_0^{T} g_1 \cdot dW, \ldots,  \int_0^{T} g_n \cdot dW \right) g_k.
 \label{eq:MD:ele:def}
\end{align}
Notice that $\MD F \in L^p(\Omega;L^2( [0,T], \RR^d))$.
To extend $\MD$ to a broader class of elements we adopt the norm
\begin{align*}
   \| F  \|_{\Mspc^{1,p}}^p := \E |F|^p + \E \|\DM F\|_{L^2([0,T]; \RR^d)}^p,
\end{align*}
and denote $\mbox{Dom}(\MD) = \Mspc^{1,p}$ be the closure of the above defined functions $F$
under this norm $\|\cdot\|_{\Mspc^{1,p}}$.

We can repeat the above construction for random variables taking values in a
separable Hilbert space $\mathcal{H}$.
In this case start by considering `elementary' functions of the form
\begin{align}
  F := \sum_{i \in \mathcal{I}} f_i \left(\int_0^{T} g_1^i \cdot dW, \ldots,  \int_0^{T} g_{n_i}^i \cdot dW \right) h_i
  =: \sum_{i \in \mathcal{I}} F_i h_i
  \,,\label{eq:gen:F:def}
\end{align}
where $i \in \mathcal{I}$ is a finite index set, $n_i \in \mathbb{N}$, $f_i : \RR^{n_i} \to \RR$ are
Schwartz functions, $h_i$ elements in $\mathcal{H}$ and, as above, $g_1^i, \cdots, g_{n_i}^i$, $i \in \mathcal{I}$ are deterministic elements in $L^2([0,T], \RR^d)$.
Define
\begin{align}\label{eq:def:mal:fun}
 \MD F :=
 \sum_{i \in \mathcal{I}} \MD (F_i) h_i \,.
\end{align}
Then $\MD$ is a closeable operator from $L^p(\Omega, \mathcal{H})$ to $L^p(\Omega; L^2( [0,T], \RR^d) \otimes \mathcal{H})$. With a slight
abuse of notation we denote
\begin{align}\label{eq:def:D12}
   \| F \|_{\Mspc^{1,p}}^p := \E \|F\|_{\mathcal{H}}^p + \E \|\DM F\|_{L^2([0,T], \RR^d) \otimes \mathcal{H}}^p.
\end{align}

As above,
we take $\Mspc^{1,p} = \Mspc^{1,p}(\mathcal{H})$ to be the closure of the  functions $F$ of the form \eqref{eq:def:mal:fun}
under the norm $\|\cdot\|_{\Mspc^{1,p}}$.

For $F \in \Mspc^{1,2}$, we adopt the notations
\begin{align*}
  \MD_s F := (\MD F)(s), \; s \in [0,T], \quad \MD^j F := (\MD F)^j, j = 1, \ldots d,
\end{align*}
i.e. $(\MD F)^j$ is the $j$th component of $\MD F$ as an element in $\RR^d$ (or $\RR^d \otimes \mathcal{H}$), for fixed $\omega, t \in \Omega \times [0,T]$.
Furthermore, in view of \eqref{eq:MD:ele:def}, we have that if $F \in \Mspc^{1,2}$ is $\mathcal{F}_s$ measurable for some $s \in [0,T]$, then
\begin{align}
  \MD_t F =0, \quad \textrm{ for every } t > s.
  \label{eq:Fs:t:meas:zero:cond}
\end{align}

With these basic definitions in place, we now introduce two important elements of Malliavin's
theory, the chain rule and the integration by part formula. The Malliavin chain rule states that
for any $\phi \in C_b^1(\RR^m)$ (continuously differentiable functions with bounded derivatives) and
$F = (F_1, \cdots, F_m)$ with $F_i \in \Mspc^{1,2}$ one has that $\phi (F) \in \Mspc^{1,2}$ and
\begin{align}
   \MD^j \phi(F)  = \nabla \phi(F) \cdot \MD^j F = \sum_{i = 1}^m \partial_{x_i} \phi (F) \MD^j F_i
   \qquad (j \in \{1, \cdots, d\})\,,
\label{eq:M:chain:rule1}
\end{align}
see e.g. \cite[Proposition 1.2.3]{Nualart2006}.  Note also that this chain rule extends  to the
Hilbert space setting; if $\phi \in C_b^1(\mathcal{H})$ and $F \in \Mspc^{1,2}(\mathcal{H})$ then
$\phi(F) \in \Mspc^{1,2}$ and   $\MD^j \phi(F)  = \nabla \phi(F) \cdot \MD^j F$.\footnote{For some of the
estimates  in Section~\ref{sec:proof:main:prop} (cf. \eqref{eq:bound}) we used a more general form of
the product rule that can be found in e.g. \cite{PronkVeraar2013}.}

Next, we introduce the Malliavin integration by part formula which can be understood in terms of the adjoint operator to $\MD$.
For $\MD : \Mspc^{1,2} \subset L^2(\Omega) \to L^2(\Omega; L^2 [0, T], \RR^d)$ define its
adjoint $\MD^\ast : \textrm{Dom}(\MD^\ast) \subset L^2(\Omega ;L^2[0, T], \RR^d) \to L^2(\Omega)$
by
\begin{align}
     \E \langle \MD F,  v \rangle_{L^2([0,T], \RR^d)} =
     \E (F \MD^\ast v)  \,,
\label{eq:M:IBP:rule}
\end{align}
for any $F \in \Mspc^{1,2}$ and any $v \in \textrm{Dom}(\MD^\ast)$. If $F$ has the form
\eqref{eq:gen:F:def} we define
\begin{align*}
     \E \langle \MD F,  v \rangle_{L^2([0,T], \RR^d)} := \sum_{i \in \mathcal{I}} \E \langle \MD F_i ,  v \rangle_{L^2([0,T], \RR^d)} h_i =
     \sum_{i \in \mathcal{I}} \E (F_i \MD^\ast v) h_i = \E \left( \sum_{i \in \mathcal{I}}  F_i h_i \MD^\ast v\right)
     =  \E (F \MD^\ast v) \,,
\end{align*}
and therefore, after passing to the limit, we see that \eqref{eq:M:IBP:rule} holds true for $\mathcal{H}$ valued elements
$F \in \Mspc^{1,2}(\mathcal{H})$ and $v \in \textrm{Dom}(\MD^\ast) \subset L^2(\Omega; L^2 ([0, T], \RR^d))$.
Note in particular that even in this infinite dimensional setting the duality in \eqref{eq:M:IBP:rule} remains in $L^2([0,T], \RR^d)$.

The map $\MD^\ast$ is called the \emph{Skorokhod integral} (see \cite{Nualart2006}) and is often written as
\begin{align}\label{eq:def:D:ast}
\MD^\ast v =: \int_0^T v \cdot dW,
\end{align}
so that \eqref{eq:M:IBP:rule} reads as
\begin{align}
\E \langle \MD F,  v \rangle_{L^2([0,T], \RR^d)} =
     \E \left(F \int_0^T v \cdot dW\right).
     \label{eq:M:IBP:rule:trad:notation}
\end{align}
The reason behind this notation is that if
$v \in L^2(\Omega ;L^2 [0, T], \RR^d)$ and is adapted to $\mathcal{F}_t$, then
$v \in \textrm{Dom}(\MD^\ast)$ and
$\int_0^T v \cdot dW$ is the classical Doeblin-It\={o} integral.  In general, interpreting $v$ as an $\mathcal{H} = L^2([0, T] ; \RR^d)$
valued random variable, we have that $\Mspc^{1,2}(L^2([0,T]; \RR^d)) \subset \textrm{Dom}(\MD^\ast)$.

In order to make quantitative estimates for \eqref{eq:def:D:ast} we finally recall a generalized
form of the classical It\={o} isometry. If $v \in \Mspc^{1,2}(L^2([0,T]; \RR^d))$, then
 $\MD v \in L^2(\Omega; L^2([0,T];\RR^d) \otimes L^2([0,T];\RR^d)) = L^2(\Omega;L^2([0,T]^2; \RR^{d \times d}))$ and the generalized
It\={o} isometry takes the form:
\begin{align}
  \E \left( \int_0^T v \cdot dW \right)^2 =
  \E \|v\|^2_{L^2([0, T], \RR^d)} + \E \int_0^T \int_0^T  \mbox{Tr}( \DM_s v(r) \DM_r v(s)) ds dr\,,
  \label{eq:gen:ito:ineq}
\end{align}
see e.g. \cite[Chapter 1, (1.54)]{Nualart2009}. In view of \eqref{eq:Fs:t:meas:zero:cond}, the classical It\={o} isometry is
recovered from \eqref{eq:gen:ito:ineq} when $v$ is $\mathcal{F}_t$-adapted.  More generally
such observations concerning the $\mathcal{F}_t$ measurability of $v$ in conjunction with  \eqref{eq:Fs:t:meas:zero:cond}, \eqref{eq:gen:ito:ineq}
are used in a crucial fashion to obtain the bound
 \eqref{eq:bound}.

\section*{Acknowledgements}
The authors gratefully acknowledge the support of the Institute for Mathematics and its Applications (IMA) where this work was conceived.
JF, NGH, GR were postdoctoral fellows and ET was a New Directions Professor during the academic year 2012-2013.
We have also benefited from the hospitality of the Department of Mathematics Virginia Tech
and from the Newton Institute for Mathematical Sciences, University of Cambridge
where the final stage of the writing was completed. NGH's work has been partially supported
under the grant NSF-DMS-1313272.

We would like to thank T. Beale, C. Doering, D. Faranda, J. Mattingly, V. \v Sver\' ak,
and V. Vicol for numerous inspiring discussions along with many helpful references.
We would also like to express our appreciation to L. Capogna and H. Bessaih for helping
us to initiate the reading group at the IMA that eventually lead to this work.

{\footnotesize

\bibliographystyle{amsalpha}

\newcommand{\etalchar}[1]{$^{#1}$}
\providecommand{\bysame}{\leavevmode\hbox to3em{\hrulefill}\thinspace}
\providecommand{\MR}{\relax\ifhmode\unskip\space\fi MR }
\providecommand{\MRhref}[2]{%
  \href{http://www.ams.org/mathscinet-getitem?mr=#1}{#2}
}
\providecommand{\href}[2]{#2}

}

\vspace{.3in}
\begin{multicols}{2}
\noindent
Juraj F\"oldes\\
{\footnotesize Institute for Mathematics and its Applications (IMA)\\
University of Minnesota\\
Web: \url{ima.umn.edu/~foldes/}\\
 Email: \url{foldes@ima.umn.edu}} \\[.25cm]
Nathan Glatt-Holtz\\ {\footnotesize
Department of Mathematics\\
Virginia Polytechnic Institute and State University\\
Web: \url{www.math.vt.edu/people/negh/}\\
 Email: \url{negh@vt.edu}} \\[.2cm]
 
\columnbreak 
 
 \noindent Geordie Richards\\ 
{\footnotesize
Department of Mathematics\\
University of Rochester\\
Web: \url{www.math.rochester.edu/grichar5/}\\
 Email: \url{grichar5@z.rochester.edu}}\\[.2cm]
Enrique Thomann\\ {\footnotesize
Department of Mathematics\\
Oregon State University\\
Web: \url{www.math.oregonstate.edu/people/view/thomann}\\
 Email: \url{thomann@math.orst.edu}}
\end{multicols}

\end{document}